\documentclass[a4paper,10pt]{article}
\usepackage{amssymb,amsmath,amsthm}
\usepackage{graphicx}

\usepackage{amsfonts}
\usepackage{latexsym}
\usepackage[english]{babel}
\usepackage{hyperref}

\numberwithin{equation}{section}

\newtheorem{theorem}{Theorem}[section]
\newtheorem{proposition}{Proposition}[section]
\newtheorem{lemma}{Lemma}[section]
\newtheorem{corollary}{Corollary}[section]

\newtheorem{remark}{Remark}[section]
\newtheorem{remarks}{Remark}[section]
\newtheorem{definition}{Definition}[section]
\newcommand{\be}{\begin{equation}}
\newcommand{\ee}{\end{equation}}

\newcommand{\e}{\varepsilon}

\newcommand{\R}{\mathbb R}
\newcommand{\C}{\mathbb C}

\newcommand{\Z}{\mathbb Z}

\newcommand{\T}{\mathbb T}

\renewcommand{\a }{\alpha }
\renewcommand{\b }{\beta }

\newcommand{\ii }{{\rm i} }

\newcommand{\io}{\iota}

\begin{document}

\title{{\bf Canonical coordinates with tame estimates for the defocusing NLS equation on the circle}}

\date{}


\author{
Thomas Kappeler\footnote{Supported in part by the Swiss National Science Foundation.}  ,
Riccardo Montalto\footnote{Supported in part by the Swiss National Science Foundation.} 
}

\maketitle

\noindent
{\bf Abstract.}
In a case study for integrable PDEs, we construct real analytic, canonical coordinates for the defocusing NLS equation
on the circle, specifically taylored towards the needs in perturbation theory. They are defined in neighbourhoods of families of finite dimensional invariant tori and are shown to satisfy together with their derivatives tame estimates. 
When expressed in these coordinates, the dNLS Hamiltonian is in normal form up to order three.

\medskip

\noindent
{\em Keywords:} defocusing NLS equation, integrable PDEs, normal form.

\medskip

\noindent
{\em MSC 2010:} 37K10, 35Q55


\tableofcontents

\section{Introduction}\label{introduzione paper}
\label{1. Introduction}
In form of a case study for integrable PDEs, the goal of this paper is to construct canonical coordinates for the defocusing NLS (dNLS) equation, specifically taylored to the needs in perturbation theory. 
We consider the  dNLS equation in one
space dimension
   \begin{equation}
   \label{1.1} \ii \partial _t u = - \partial ^2_x u + 2|u|^2u\,, \quad x \in \T := \R / \Z
   \end{equation}
on the Sobolev space $H^s_\C \equiv H^s({\mathbb T}, {\mathbb C})$ of complex
valued functions on ${\mathbb T}$, whose distributional derivatives up to order
$s \in {\mathbb Z}_{ \geq 0}$ are in $L^2(\T, \C)$. Equation \eqref{1.1} can be viewed
as a Hamiltonian PDE, obtained by restricting the Hamiltonian system on the
phase space $H^s_c := H^s_\C \times H^s_\C$ with Poisson bracket
and Hamiltonian given by
\begin{equation}\label{Poisson brackets}
\{ {\mathcal F}, {\mathcal G}\} (u, v) 
= - \ii \int ^1_0 (\partial _u {\mathcal F} \partial _{v} {\mathcal G} - 
    \partial _{v} {\mathcal F} \partial _u {\mathcal G})dx, \qquad
     {\mathcal H}^{nls}(u, v) = \int ^1_0 (\partial _x  u \partial _x v + u^2 v^2)dx
\end{equation}
to the real subspace $H^s_r$ of $H^s_c$ consisting of elements $(u, v)$ with $v = \overline u$. Here $\mathcal F, \mathcal G$ are ${\cal C}^1$-smooth complex valued functionals on $H^s_c$ with sufficiently regular $L^2$-gradients. 
Equation \eqref{1.1} can then be
rewritten as $\partial _tu = - \ii \partial _{v}{\mathcal H}^{nls} \mid_{v = \overline u}$. 
The dNLS equation  is an integrable PDE and according to \cite{GK},  admits global Birkhoff
coordinates on $H^s_\C$ with $s \in {\mathbb Z}_{\geq 0}$. To state the main results of this paper we first need to describe these coordinates in more detail: for any $s \in \Z_{\geq 0}$, let  
$$
h^s_\C \equiv h^{s}(\Z, \C) := \big\{ x = (x_n)_{n \in \Z} \subseteq \C : \| x \|_s < + \infty \big\}\,, \quad \| x \|_s := \sum_{n \in \Z}\langle n \rangle^{2 s} |x_n|^2, \quad \langle n \rangle := {\rm max}\{1, |n| \}\,,
$$
$$
h^s \equiv h^s(\Z, \R) := \big\{ (x_n)_{n \in \Z} \in h^s_\C : \,\, x_n \in \R \,\,\, \forall n \in \Z\big\}
$$
and 
$$
h^s_c := h^s_\C \times h^s_\C\,, \qquad h^s_r := h^s \times h^s\,. 
$$
Note that the Sobolev space $H^s_\C$ can then be described by 
$$
H^s_\C = \big\{ u= \sum_{n \in \Z} u_n e^{2 \pi i n x} : (u_n)_{n \in \Z} \in h^s_\C  \big\}\,, \qquad  \| u \|_s : = \| (u_n)_{n \in \Z}\|_s\,.
$$
Furthermore let 
$$
\ell^{1, 2} \equiv \ell^{1, 2}(\Z, \R) := \big\{ x = (x_n)_{n \in \Z} \subset \R : \,\, \| x\|_{1, 2} := \sum_{n \in \Z} \langle n \rangle^2 |x_n| < + \infty \big\}\,,
$$
$$
\ell^{1, 2}_+ := \big\{ (x_n)_{n \in \Z} \in \ell^{1, 2} : \quad  x_n \geq 0\,, \quad \forall n \in \Z \big\}
$$
and define the following version $F_{nls}$ of the Fourier transform, introduced in \cite{GK}, 
\begin{equation}\label{F nls}
F_{nls} : H^0_c \to h^0_c\,, \qquad (u, v) \mapsto 
\big(- \frac{1}{\sqrt{2}} (u_{- n} + v_n), \, - \frac{\ii}{\sqrt{2}}(u_{- n} - v_n) \big)\,,
\end{equation}
where $u_n$ denotes the $n$th Fourier coefficient of $u$, $u_n := \int_0^1 u(x) e^{- 2 \pi \ii n x}\, d x$. 
Note that for $v = \overline u$, one has $v_n = \overline u_{- n}$ for any $n \in \Z$, implying that 
$$
F_{nls}(u, \overline u) = \big( - \sqrt{2}{\rm Re}(u_{- n})\,,\, \sqrt{2} {\rm Im}(u_{- n}) \big)\,.
$$
The inverse of $F_{nls}$ is then given by 
$$
F_{nls}^{- 1} : h^0_c \to H^0_c \,, \quad \big( (x_n)_{n \in \Z}, (y_n)_{n \in \Z} \big) \mapsto 
\big( - \frac{1}{\sqrt{2}} \sum_{n \in \Z} (x_{- n} - \ii y_{- n}) e^{2 \pi \ii n x}\,, 
\, - \frac{1}{\sqrt{2}} \sum_{n \in \Z} (x_n + \ii y_n) e^{2 \pi \ii n x} \big)\,.
$$
Finally we recall that a possibly nonlinear map $F : U \to Y$ of a subset $U$ of a Banach space $X$ into another Banach space $Y$ is said to be bounded if $F(V)$ is bounded for any bounded subset $V$ in $U$. 
\begin{theorem}[\cite{GK}, \cite{KST}] \label{Theorem Birkhoff coordinates}
 There exists a neighbhorhood ${\cal W}$ of $H^0_r$ in $H^0_c$ and an analytic map 
$$
\Phi^{nls} : {\cal W} \to h^0_c\,, \qquad (u, v) \mapsto \big( (x_n)_{n \in \Z}, (y_n)_{n \in \Z}\big)
$$ 
with $\Phi^{nls}(0) = 0$ such that the following holds: 
\begin{description}
\item[{\rm (B1)}] For any $s \in \Z_{\geq 0}$, $\Phi^{nls}(H^s_r) \subseteq h^s_r$ and $\Phi^{nls}: H^s_r \to h^s_r$ is a real analytic diffeomorphism. 
\item[{\rm (B2)}] The map $\Phi ^{nls}$ is canonical, meaning that on ${\cal W}$, $\{ x_n, y_n \} = - 1$ and all the other brackets between coordinate functions vanish.  
\item[{\rm (B3)}] The Hamiltonian $H^{nls} := {\cal H}^{nls} \circ (\Phi^{nls})^{- 1}$, defined on $h_r^1$, is a function of the actions $I_n := (x_n^2 + y_n^2)/2$, $n \in \Z$, only and $H^{nls} : \ell^{1, 2}_+ \to \R, I \mapsto H^{nls}(I)$ 
is real analytic.  
\item[{\rm (B4)}] The differential $d_0\Phi^{nls}$ of $\Phi^{nls}$ at $0$ is the
Fourier transform $F_{nls}$ defined in \eqref{F nls}.
\item[{\rm (B5)}] The nonlinear parts $A^{nls} := \Phi^{nls} - F_{nls}$ of $\Phi^{nls}$ and $B^{nls} := \Psi^{nls} - F_{nls}^{- 1}$ of $\Psi^{nls}:= (\Phi^{nls})^{- 1}$ are one smoothing, meaning that for any $s \in \Z_{\geq 1}$,
$$
A^{nls} : H^s_r \to h^{s + 1}_r \quad \text{and} \quad B^{nls} : h^s_r \to H^{s + 1}_r
$$
are real analytic and bounded.
\end{description}

\noindent
The maps $\Phi^{nls}, \Psi^{nls}$ are referred to as Birkhoff maps and the coordinates 
$((x_n, y_n))_{n \in \Z}$ as Birkhoff coordinates for the dNLS equation.
\end{theorem}

Birkhoff coordinates are a tool to study perturbations of the dNLS equation far away from the equilibrium. In particular, in \cite{BKM} they were used to show the existence of finite dimensional invariant tori of large size for 
Hamiltonian perturbations of this equation, involving no derivatives of $u$. 
So far, no such results have been obtained for perturbations involving $\partial_x u$ (and possibly $\partial_x^2 u$) --  
see \cite{BKM}, \cite{B3}, \cite{CW}, \cite{GY1}, \cite{GY2}, \cite{GK1}, \cite{KL}, \cite{KP2}
 for results on perturbations of the dNLS equation on the circle obtained so far.
 In view of the recent results in \cite{BBM3} concerning the existence of small
 quasi-periodic solutions of {\em quasi-linear} Hamiltonian perturbations of the KdV equation
 and our results in \cite{BKM} described above,
 we expect that Hamiltonian perturbations of the
 dNLS equation, involving $\partial_x u$ (and possibly $\partial_x^2 u$),
also admit {\em large} quasi-periodic solutions, also referred to as multi-solitons. For this purpose, the scheme developed in \cite{BKM}
 has to be considerably refined. In particular, canonical coordinates are needed which together 
 with their derivatives satisfy tame estimates. 
In \cite{Molnar}, such estimates were derived for $\Phi^{nls}$ and its inverse, but so far are not
available for their derivatives. 
In this paper, we prove how to use the Birkhoff coordinates  to construct 
 near bounded, integrable, finite dimensional subsystems of the dNLS equation, {\em local}  canonical coordinates
 so that they satisfy tame estimates and the dNLS Hamiltonian, when expressed in these coordinates,
 is in normal form up to order three -- see Theorem \ref{modified Birkhoff map} for a precise statement. In future work, we will use these coordinates as a starting point for applying 
 a KAM scheme to reduce certain linear operators with tame estimates , which come up in the Nash Moser iteration, to operators with constant coefficients.  Recently, such schemes have been further developed in significant ways.
 In the context of the dNLS equation, results of this type in  \cite{Berti-Montalto} 
 will be particularly relevant.

To state our main result, we need to introduce some more notation.
For any $S \subseteq \Z$ with $|S| < + \infty$, let $S^\bot := \Z \setminus S$. By a slight abuse of notation, we identify $h^s_c$ with $\C^S \times \C^S \times h^s_{ \bot c}$ and $h^s_r$ with $\R^S \times \R^S \times h^s_{\bot r}$ where 
$$
h^s_{\bot c} := h^s(S^\bot, \C) \times h^s(S^\bot, \C)\,, \qquad h^s_{\bot r} := h^s(S^\bot, \R) \times h^s(S^\bot, \R)\,.
$$
Accordingly, an element $z \in h^0_c$ is written as 
$$
z = (z_S, z_\bot)\,, \qquad z_S = \big( (x_j)_{j \in S}, (y_j)_{j \in S} \big), \quad z_\bot = \big( (x_j)_{j \in S^\bot}, (y_j)_{j \in S^\bot} \big)\,, 
$$
and as norm we choose $
\| z \|_s := \| z_S\| + \| z_\bot\|_s
$ where
$$
\| z_S\| \equiv \| z_S\|_0 := \Big( \sum_{j \in S} |x_j|^2 + |y_j|^2 \Big)^{\frac12}\,, \qquad \| z_\bot\|_s := \Big( \sum_{j \in S^\bot} \langle j \rangle^{2 s}(|x_j|^2 + |y_j|^2) \Big)^{\frac12}\,.
$$
Furthermore, we introduce the bilinear form
\begin{align}\label{bilinear form on bot}
(z_\bot , z'_\bot)_r := \sum_{j \in S^\bot} x_j x_j' + y_jy_j' \,, \quad 
z_\bot=(x_\bot, y_\bot), \,\,  z'_\bot=(x'_\bot, y'_\bot) \in h_{\bot c}\,.
\end{align}
and write the sequence of actions $I = (I_k)_{k \in \mathbb Z}$ as $(I_S, I_\bot)$ where
$$
I_S := (I_k)_{k \in S}, \qquad I_\bot := (I_k)_{k \in S^\bot} \, , \quad I_k \equiv I_k(z) = \frac{|z_k|^2}{2} = \frac{x_k^2 + y_k^2}{2}\,, \quad \forall k \in \Z\,.
$$
Finally, we introduce the dNLS frequencies 
 \begin{equation}\label{definizione frequenze NLS introduzione}
 \omega_k^{nls}(I) := \partial_{I_k} H^{nls}(I) \,, \qquad k \in \Z\,.
 \end{equation}
 They satisfy asymptotics of the form $\omega_k(I) = 4 k^2 \pi^2 + O(1)$
 as $k \to \pm \infty$.
 More precisely, the map
  $$
  \ell^{1, 2}_+ \rightarrow \ell ^\infty , \ (I_k)
      _{k \in {\mathbb Z}} \mapsto (\omega ^{ nls}_n(I) - 4 \pi ^2 n^2)
      _{n \in {\mathbb Z}}
 $$
is real analytic and bounded -- see Proposition \ref{Asintotica frequenze integrabile} in 
Subsection \ref{Hamiltoniana trasformata} below. 
The main result of this paper is the following one.

\begin{theorem}\label{modified Birkhoff map}
Let $S \subseteq \Z$ be finite. For any compact subset ${\cal K} \subseteq \R^S \times \R^S$, there exists an open, bounded, complex neighbourhood ${\cal V} \subseteq h^0_c$ of ${\cal K} \times \{ 0 \}$ and a
bounded analytic map 
$$
\Psi : {\cal V} \to H^0_c\,, \,\,\, (z_n)_{n \in \Z} \mapsto w
$$ 
so that the following holds: 
\begin{description}
\item[({\rm C1})] For any $s \in \Z_{\geq 0}$, $\Psi({\cal V} \cap h^s_r) \subseteq H^s_r$ and $\Psi : {\cal V} \cap h^s_r \to H^s_r$ is a real analytic diffeomorphism onto its image. 
\item[({\rm C2})] $\Psi$ is canonical, meaning that on $\Psi({\cal V} \cap h^0_r)$,  $\{ x_n, y_n \} = - 1$ for any $n \in \Z$, whereas all the other brackets between coordinate functions vanish. 
\item[({\rm C3})] The transformation $\Psi$ is related to $\Psi^{nls} = (\Phi^{nls})^{- 1}$ by 
$$
\Psi\mid_{ {\cal K} \times \{ 0 \}} = \Psi^{nls}\mid_{ {\cal K} \times \{ 0 \}}\,, \quad d \Psi(z) = d \Psi^{nls}(z)\,, \quad \forall z \in {\cal K} \times \{ 0 \}\,.
$$
\item[({\rm C4})] The Hamiltonian ${\cal H} := {\cal H}^{nls} \circ \Psi$, defined on ${\cal V} \cap h^1_r$, is in normal form up to order three. More  precisely,
$$
{\cal H}(z) = H^{nls}(I_S, 0) + \sum_{n \in S^\bot}   {\omega}^{nls}_n(I_S, 0)I_n(z) \, + \, {\cal P}_3(z)
$$
where 
the Hamiltonian ${\cal P}_3 : {\cal V} \cap h^0_r \to \R$ is real analytic. Furthermore, ${\cal P}_3$ satisfies the following tame estimates: for any $s \in \Z_{\geq 0}$, $z \in {\cal V} \cap h^s_r\,,$
 $\widehat z   \in h^s_c$,
\begin{align}\label{estimate error term}
&  \| \nabla {\cal P}_3(z)\|_s \lesssim_s \| z_\bot\|_s \| z_\bot\|_0\,, \quad \| d \nabla {\cal P}_3(z) [\widehat z] \|_s \lesssim_s \| z_\bot\|_s \| \widehat z\|_0 + \| z_\bot\|_0 \| \widehat z\|_s
\end{align}
and for any $k \in \mathbb Z_{\geq 2}$, $ \widehat z_1, \ldots, \widehat z_k \in h^s_c$,
$$
\| d^k \nabla {\cal P}_3(z)[\widehat z_1, \ldots, \widehat z_k]\|_s \lesssim_s \sum_{j = 1}^k \| \widehat z_j\|_s \prod_{i \neq j} \| \widehat z_i\|_0 + \| z_\bot\|_s \prod_{j = 1}^k \| \widehat z_j\|_0\,.
$$
Here, the meaning of $\lesssim_s$ is the standard one. So e.g.  $\| \nabla {\cal P}_3(z)\|_s \lesssim_s \| z_\bot\|_s \| z_\bot\|_0$ says that there exists a constant
$C \equiv C(s) > 0$ so that
$$
\| \nabla {\cal P}_3(z)\|_s \leq C \| z_\bot\|_s \| z_\bot\|_0\,, \quad \forall \, z \in {\cal V} \cap h^s_r \,.
$$
\item[({\rm C5})] The nonlinear maps $B:= \Psi - F_{nls}^{- 1} : {\cal V} \cap h^0_r \to H^0_r$ and $A := \Psi^{- 1} - F_{nls} : \Psi({\cal V}) \cap H^0_r \to h^0_r$ are real analytic maps and so is 
$${\cal A} : {\cal V} \cap h^0_r \to {\cal L}(H^0_c, h^0_c), \, z \mapsto {\cal A}(z) := d \Psi(z)^{- 1} - F_{nls}.
$$
On ${\cal V} \cap h^0_r$,  the maps $B$ and ${\cal A }$ satisfy the following estimates: for any $z \in {\cal V} \cap h^0_r$, $k \in \Z_{\geq 1}$, $\widehat z_1, \ldots, \widehat z_k \in h^0_c$, 
and $\widehat w \in H^0_c$, 
$$
\| B(z)\|_0 \lesssim 1 \,, \quad \| d^k B(z)[\widehat z_1, \ldots, \widehat z_k]\|_0 \lesssim_k \prod_{j = 1}^k\|\widehat z_j \|_0\,,
$$
$$
\| {\cal A}(z)[\widehat w]\|_0 \lesssim \| \widehat w\|_0\,, \quad \| d^k \big( {\cal A}(z)[\widehat w] \big)[\widehat z_1, \ldots, \widehat z_k] \|_{0} \lesssim_k \| \widehat w\|_0\prod_{j = 1}^k \| \widehat z_j\|_0 \,.
$$
Furthermore,  $B$ is one smoothing, meaning that for any $s \in \Z_{\geq 1}$, $B: {\cal V} \cap h^s_r \to H^{s+1}_r$ is real analytic, and satisfies the following  tame estimates:   for any $k \in \Z_{\geq 1}$,
$ z \in {\cal V} \cap h^s_r$, and   $ \widehat z_1, \ldots , \widehat z_k \in h^s_c$,
 $$ \| B(z) \|_{s + 1} \lesssim_s \, 1 +  \| z_\bot\|_s\,, \quad
    \| d^k B(z)[\widehat z_1, \ldots, \widehat z_k]\|_{s + 1} \lesssim_{s, k} 
    \sum_{j = 1}^k \| \widehat z_j\|_s \prod_{i \neq j} \| \widehat z_i\|_0 
    + \|  z_\bot\|_s \prod_{j = 1}^k \|\widehat z_j \|_0\,.
$$
 
\noindent
Similarly, the maps  $A$ and ${\cal A }$ are one smoothing, meaning that for any $s \in \Z_{\geq 1}$, 
$A: \Psi( {\cal V} \cap h^s_r) \to h^{s+1}_r$ and 
${\cal A } : {\cal V} \cap h^s_r \to {\cal L}(H^s_c, h^{s + 1}_c)$
are real analytic. Moreover, ${\cal A }$  satisfies the following tame estimates: 
for any $z \in {\cal V} \cap h^s_r$, $\widehat w \in H^s_c$,
$$
 \| {\cal A}(z)[\widehat w] \|_{s+1} \lesssim_{s}  \| z_\bot\|_s \| \widehat w \|_0  +  \| \widehat w\|_s
$$
and for any $k \in \Z_{\geq 1},$  $\widehat z_1, \ldots, \widehat z_k \in h^s_c$,
$$
\| d^k \big( {\cal A}(z)[\widehat w] \big)[\widehat z_1, \ldots, \widehat z_k] \|_{s + 1}
 \lesssim_{s, k}  
\,\, \big(   \| z_\bot\|_s \| \widehat w \|_0  +  \| \widehat w\|_s \big) \, \prod_{j = 1}^k \| \widehat z_j\|_0  \, 
+ \,    \| \widehat w\|_0 \, \sum_{j = 1}^k  \| \widehat z_j\|_s \prod_{i \neq j} \| \widehat z_i\|_0  \,.
$$  
\end{description}
\end{theorem}

\medskip

\medskip

\begin{remark}
Note that in Theorem \ref{modified Birkhoff map}, apart from being compact, no further assumptions on ${\cal K}$ are being made. In particular, ${\cal K}$ may contain the equilibrium point $0$ in which case ${\cal K}$ does not admit action-angle coordinates. In subsequent work, the estimates for ${\cal A}(z) =d \Psi(z)^{- 1} - F_{nls}$ will be used to study 
perturbations of the dNLS equation. Since such estimates are not needed for $A(\Psi(z))$,  we have not included them in Theorem \ref{modified Birkhoff map}.

\end{remark}

\medskip

\noindent
{\em Outline of the construction of $\Psi$:}
Let ${\cal V}$ be of the form ${\cal V} = {\cal V}_S \times {\cal V}_\bot \subset h^0_c$ where ${\cal V}_S$ is a
bounded, open neighbourhood of ${\cal K}$ in $\C^S \times \C^S$ and ${\cal V}_\bot$ an open ball in $h^0_{\bot c}$,
centered at $\{ 0 \}$. By Theorem \ref{Theorem Birkhoff coordinates}, ${\cal V}_S$ and ${\cal V}_\bot$ can be
chosen so that the Birkhoff map $\Psi^{nls}$ is defined on ${\cal V}$ and all the estimates of $\Psi^{nls}$ and its derivatives used in the sequel are uniform on ${\cal V}$. The canonical map 
$\Psi$ is then defined to be the composition $\Psi := \Psi_L \circ \Psi_C$ where 
$\Psi_L$ is the Taylor expansion of $\Psi^{nls}$ of order one in the normal directions $z_\bot$ 
around $(z_s, 0)$,
\begin{equation}\label{prima def Psi L}
\Psi_L(z_S, z_\bot) : = \Psi^{nls}(z_S, 0) + d \Psi^{nls}(z_S, 0)[0, z_\bot]\,,
\end{equation}
and $\Psi_C$, referred to as symplectic corrector, is chosen so that $\Psi_L \circ \Psi_C$ becomes symplectic
and satisfies the claimed tame estimates.\\
\noindent
In his pioneering work \cite{K}, Kuksin presents a general scheme for proving KAM type theorems for semilinear Hamiltonian perturbations of integrable PDEs in one space dimension, such as the Korteweg de Vries (KdV) or the sine Gordon (sG) equations, which possess a Lax pair formulation and admit finite dimensional integrable subsystems, foliated by invariant tori. One of the key elements of his work is a normal form theory for such PDEs. Expanding on work of Krichever \cite{Krichever}, Kuksin considers bounded integrable finite dimensional subsystems (iSS) of such an integrable PDE (iPDE) which admit action-angle coordinates. In the case of the KdV and the sG equations, the angle variables are given by the celebrated Its Matveev formulas. These action-angle coordinates are complemented by infinitely many coordinates whose construction is based on a set of time periodic solutions, referred to as Floquet solutions, of the partial differential equation obtained by linearizing iPDE along solutions in iSS. The resulting coordinate transformation, denoted in \cite{K} by $\Phi$, is typically not symplectic and to obtain canonical coordinates, an additional coordinate transformation needs to be applied. 
In \cite{K}, Kuksin constructs such a transformation, which he denotes by $\phi$, 
using arguments of Moser and Weinstein in the given infinite dimensional setup -- see \cite{K}, Lemma 1.4
and Section 7.1. To construct the map $\Psi_C$ we follow the same scheme of proof. 
Actually, the following result holds.

 \begin{theorem}\label{comparison Kuksin}
 Assume that in addition to the assumptions made in Theorem \ref{modified Birkhoff map},
 the set ${\cal K}$ is contained in $(\mathbb R \setminus 0)^S \times (\mathbb R \setminus 0)^S.$
 Then, up to normalizations and natural identifications,  $\Psi_L$ coincides with the map $\Phi$,
 obtained by applying the scheme of construction in \cite{K} to the dNLS equation.
 As a consequence, so does  $\Psi = \Psi_L \circ \Psi_C$ with $\Phi \circ \phi$.
\end{theorem}

\medskip
Since Birkhoff coordinates provide a concise,
self-contained, and efficient framework for proving Theorem \ref{modified Birkhoff map}
-- in particular  the claimed tame estimates, the main goal of our study -- 
Theorem \ref{comparison Kuksin} also provides in the case of the dNLS equation 
a valuable alternative for proving the normal form result for this equation,
obtained by applying the scheme of proof  in $\cite{K}$.
Note also that the assumptions on ${\cal K}$ in Theorem  \ref{modified Birkhoff map} are slightly weaker than
the ones made in the setup of \cite{K}. 
\medskip

\noindent
{\em Organization:}
The maps $\Psi_L$ and $\Psi_C$ are introduced and studied in Sections \ref{section Psi L} and 
\ref{costuzione mappa Psi simplettica} respectively, after a short Section \ref{formalismo hamiltoniano}, describing the Hamiltonian setup. 
In Section \ref{Proof of main theorem}, we prove Theorem \ref{modified Birkhoff map}: 
in Subsection \ref{sec new coordinates}, we show that the  composition $\Psi = \Psi_L \circ \Psi_C$ 
satisfies the analytic properties, stated in Theorem \ref{modified Birkhoff map}, and in the subsequent
Subsection \ref{Hamiltoniana trasformata},  the expansion of the dNLS Hamiltonian in the new coordinates is computed up to order three. In Subsection \ref{summary proof}  we summarize the proof of Theorem \ref{modified Birkhoff map}.
Finally, in Section \ref{comparison with Kuksin's map} we prove Theorem \ref{comparison Kuksin}.
In Appendix A, we recall an infinite dimensional version of the Poincar\'e Lemma, 
needed in Section \ref{costuzione mappa Psi simplettica} (cf from \cite{K}, \cite{Lang}).

\medskip

\noindent
{\em Notation:}
For any $C^1$ map $F : h^0_c \to X$ with $X$ being a Banach space, we denote by $d_\bot F(z)$ the differential of $F$ at $z$ with respect to the variable $z_\bot $,  
$$
d_\bot F(z) [\widehat z_\bot] = \sum_{j \in S^\bot }\widehat x_j \partial_{x_j} F(z)\,  + \widehat y_j\partial_{y_j} F(z)\, \,,\qquad \widehat z_\bot := \big( (\widehat x_j)_{j \in S^\bot}, (\widehat y_j)_{j \in S^\bot} \big) \in h^0_{\bot c}\,,
$$
where for any $j \in S^\bot$, $\partial_{x_j} F,\, \partial_{y_j} F \in X$ denote the partial derivatives of $F$ with respect to the variables $x_j$ respectively $y_j$. Similarly, we  define the gradient with respect to the variable 
$z_\bot$ as 
$$
\nabla_\bot F := \big( (\partial_{x_j} F)_{j \in S^\bot}, (\partial_{y_j} F)_{j \in S^\bot} \big)\,.
$$
The gradient of $F$ with respect to $z_S$ is denoted by
$$
\nabla_S F := \big( (\partial_{x_j} F)_{j \in S}, (\partial_{y_j} F)_{j \in S} \big)\,
$$
and the differential of $F$ at $z$ with respect to $z_S$ by $d_SF(z)$, 
$$
d_S F(z) [\widehat z_S] = \sum_{j \in S }\widehat x_j \partial_{x_j} F(z)\,  + \widehat y_j\partial_{y_j} F(z)\, \,,\qquad \widehat z_S := \big( (\widehat x_j)_{j \in S}, (\widehat y_j)_{j \in S} \big) \in \C^S \times \C^S\,.
$$
\noindent
For the partial derivatives of $F$ with respect to $z_j$, $j \in S$, we use the multi-index notation and write for any  $\alpha, \beta \in \Z^S_{\geq 0}$
$$
\partial_S^{\alpha, \beta} F := \Big( \prod_{j \in S} \partial_{x_j}^{\alpha_j} \partial_{y_j}^{\beta_j}  \Big) F\,.
$$
If not stated otherwise, ${\cal K}$ denotes a compact subset of $\R^S \times \R^S$ and ${\cal V}$ an open,
bounded neighborhood of ${\cal K}\times \{0\}$ in $h_c^0$ of the form ${\cal V}_S \times {\cal V}_\bot$
where ${\cal V}_\bot$ is a ball in $h_{\bot c}$, centered at $0$. We write ${\cal V}_\bot (\delta)$
to indicate that the radius of the ball ${\cal V}_\bot$  is $\delta > 0.$ Finally, we frequently will use the symbols 
$ \lesssim$,  $\lesssim_{s}$, \ldots \,
to express that a quantity is bounded by another one  up to a constant which is 'universal', respectively depends
only on the Sobolev index $s$. E.g.,  given two real valued functionals $A, B$ on ${\cal V}$
we write  $A \lesssim_{s} B$ if there is a constant $C\equiv C(s)$ so that 
$A(z) \le C B(z)$ for any $z \in {\cal V} \cap h_r^s$.

\medskip

\section{Hamiltonian setup}\label{formalismo hamiltoniano}

In this preliminary section we discuss the Hamiltonian setup, introduced in Section \ref{introduzione paper},
 in more detail and introduce some additional notations.

\noindent
The Hamiltonian vector field associated to a sufficiently smooth functional ${\cal F} : H^0_c \to \C$ and the Poisson bracket \eqref{Poisson brackets} on $H^0_c$ is denoted by
$$
X_F = \ii {\mathbb J} \nabla {\cal F}\,, \qquad \nabla{\cal F} :=
(\nabla_u {\cal F} ,
\nabla_v {\cal F})
$$
where $\nabla_u {\cal F}$, $\nabla_v {\cal F}$ denote the $L^2$ gradients with respect to $u$ and $v$, namely 
$$
d F [(\widehat u, 0)] = \int_\T \nabla_u F\,\widehat u\, d x \,, \quad d F[(0, \widehat v)] = \int_\T \nabla_v F\, \widehat v\, d x\,,
$$ 
\begin{equation}\label{J spazio di funzioni}
{\mathbb J} := \begin{pmatrix}
0 & - {\rm Id}\\
 {\rm Id} & 0
\end{pmatrix} : H^0_c \to H^0_c\,,
\end{equation}
and ${\rm Id} : H^0_\C \to H^0_\C$ is the identity operator. Furthermore we introduce the non degenerate bilinear form 
$$
\langle \cdot, \cdot \rangle_r : H^0_c \times H^0_c \to \C
$$
defined for any $w= (u, v)$, $w' = (u', v') \in H^0_c$ by
\begin{equation}\label{definizione forma bilineare H0 c}
\langle w, w' \rangle_r := \int_\T u(x) u'(x)\, d x + \int_\T v(x) v'(x)\, d x\,.
\end{equation}
The subscript $r$ indicates that in the latter integrals, no complex conjugation appears. The Poisson bracket
\eqref{Poisson brackets} then reads
$$
\{ {\cal F}, {\cal G} \} = \langle \nabla {\cal F}, \,  \ii {\mathbb J}  \nabla  {\cal G} \rangle_r
$$
and the symplectic form, associated to it,  is the two form  
\begin{equation}\label{definizione 2 forma L2}
\Lambda [\widehat w, \widehat w'] := - \ii \langle {\mathbb J}^{- 1} \widehat w, \widehat w' \rangle_r = 
\ii \langle {\mathbb J} \widehat w, \widehat w' \rangle_r  = 
\ii \int_\T \big(\widehat u \widehat v' -  \widehat v \widehat u'  \big)\, d x \,, \qquad 
\forall\, \widehat w = (\widehat u, \widehat v), \,\,\widehat w' = (\widehat u', \widehat v') \in H^0_c\,.
\end{equation}
 For any sufficiently smooth functionals ${\cal F}, {\cal G} : H^0_c \to \C$, one has  
$$
\Lambda(X_{\cal F}, X_{\cal G}) = \{ {\cal F}, {\cal G} \}\,.
$$
\noindent
In terms of the Fourier coefficients of $\widehat w$ and $\widehat w'$, $\Lambda[\widehat w, \widehat w']$ can be expressed as  
$$
\Lambda [\widehat w, \widehat w'] = \ii \sum_{k \in \Z} (\widehat u_k \widehat v_{-k}' - \widehat v_{- k} \widehat u_k')
$$
and hence $\Lambda$ can be conveniently written as 
$$
\Lambda =  \ii \sum_{k \in \Z} d u_k \wedge d v_{- k}\,,
$$
where 
$$
\big( d u_k \wedge d v_{- k} \big)[(\widehat u, \widehat v), (\widehat u', \widehat v')] = 
\widehat u_k \widehat v_{-k}' -  \widehat v_{- k} \widehat u_k'\,, \quad \forall k \in \Z\,.
$$
In addition, we define the one form $\lambda$ on $H^0_c$ as 
$$
\lambda \equiv  \lambda(w) = \ii \sum_{k \in \Z} u_k d v_{-k}\,.
$$
Its action on a function $\widehat w = (\widehat u,  \widehat v) \in H^0_c$ is given by 
$$
\lambda[\widehat w] =  \ii \int_\T u(x) \widehat v(x)\, d x  =   \ii \sum_{k \in \Z} u_k \widehat v_{-k}\,.
$$
The exterior differential of $\lambda$, defined by $d \lambda = \ii \sum_{k \in \Z} d u_k \wedge d v_{- k}$, 
thus satisfies $d \lambda = \Lambda$. 

\bigskip

The Poisson bracket on the model space $h^0_c$ is determined by defining it for the coordinate functions,  
$$
\{ x_n, y_m \}_M = - \delta_{nm}\,, \quad \{ y_n, x_m \}_M =  \delta_{nm}\,, \quad
\{ x_n, x_m \}_M = 0\,, \quad \{ y_n, y_m \}_M = 0\,, \quad \forall n, m \in \Z\,.
$$
By a slight abuse of terminology in connection with the definition \eqref{bilinear form on bot}, we also denote by
$\big(\cdot, \cdot \big)_r$ the non degenerate bilinear form 
$
\big(\cdot, \cdot \big)_r : h^0_c \times h^0_c \to \C
$
\begin{equation}\label{forma bilineare h 0c}
\big( z, z' \big)_r := x \cdot x' + y \cdot y'\,, \qquad \forall z = (x, y), \, \,z' = (x', y') \in h^0_c
\end{equation}
where $x \cdot x' := \sum_{k \in \Z} x_k x_k'$.
Given two sufficiently smooth functionals $F, G : h^0_c \to \C$, one has 
$$
\{ F, G \}_M = - \sum_{k} \Big( \partial_{x_k} F \partial_{y_k} G - \partial_{y_k} F \partial_{x_k} G \Big) 
= \big( \nabla F,  J \nabla G \big)_r
$$
where
$$
J := \begin{pmatrix}
0 & - {\rm Id} \\
{\rm Id} & 0
\end{pmatrix} : h^0_c \to h^0_c\,, 
$$
${\rm Id} : h^0_c \to h^0_c$ is the identity operator and
$$
 \nabla F = (\nabla_x F, \nabla_y F)\,, \qquad  \nabla_x F = (\partial_{x_k} F)_{k \in \Z}\,, \qquad \nabla_y F = (\partial_{y_k} F)_{k \in \Z}\,.
$$
 The Hamiltonian vector field $X_F$ of $F : h^0_c \to \C$,  corresponding to the Poisson bracket
$\{ \cdot,   \cdot \}_M$, is then given by 
$$
X_F = J \nabla F
$$
and the symplectic form $\Lambda_M$, associated to it, by 
\begin{equation}\label{definizione due forma h0 c}
\Lambda_M [\widehat z, \widehat z'] := \big( J^{- 1} \widehat z, \widehat z'  \big)_r 
=  \widehat y \cdot \widehat x' - \widehat x \cdot \widehat y'\,, 
\quad \forall \widehat z = (\widehat x, \widehat y), \,\,  \widehat z' = (\widehat x', \widehat y') \in h^0_c\,.
\end{equation}
Note that 
$$
\Lambda_M = - \sum_{k \in \Z} d x_k \wedge d y_k
$$
where as above, for any $k \in \Z$, the two form $d x_k \wedge d y_k$ is defined as 
$$
(d x_k \wedge d y_k) [(\widehat x, \widehat y), (\widehat x', \widehat y')] = 
\widehat x_k \widehat y_{k}' - \widehat y_k \widehat x_{k}' \,.
$$
Then  
$$
\Lambda_M(X_F, X_G) = \big( \nabla F, J \nabla G \big)_r  = \{ F, G\}_M\,.
$$
The one form associated to $\Lambda_M$ is defined as 
\begin{equation}\label{definizione uno forma h0 c}
\lambda_M \equiv \lambda_M (z) := \sum_{k \in \Z} y_k \, d \, x_k\,.
\end{equation}
Its action on a vector $\widehat z = (\widehat x, \widehat y) \in h_c^0$ is given by 
$$
\lambda_M [\widehat z] = \sum_{k \in \Z} y_k\, \widehat x_k\,.
$$
The exterior differential of $\lambda_M$ then satisfies $d \lambda_M = \Lambda_M$.


\section{The map $\Psi_L$}\label{section Psi L}
In this section, we study the map $\Psi_L$ introduced in \eqref{prima def Psi L}. In particular, we prove tame estimates and one smoothing properties for $\Psi_L$. First we introduce some more notations. Denote by $\Pi_S$ and $\Pi_\bot$ the standard projections
\begin{equation}\label{Pi S}
\Pi_S : (\C^S \times \C^S) \times h^0_{\bot c} \to (\C^S \times \C^S) \times \{ 0 \}\,, \, z = (z_S, z_\bot) \mapsto (z_S, 0) 
\end{equation}
\begin{equation}\label{Pi bot}
\Pi_\bot : (\C^S \times \C^S) \times h^0_{\bot c} \to \{ 0 \} \times h^0_{\bot c} \,, \, z = (z_S, z_\bot) \mapsto (0, z_\bot) \,.
\end{equation}
The formula \eqref{prima def Psi L} for $\Psi_L(z)$  with $z = (z_S, z_\bot)$ then reads 
\begin{equation}\label{seconda def Psi L}
\Psi_L(z) = \Psi^{nls}(\Pi_S z) + d_\bot \Psi^{nls}(\Pi_S z) [ z_\bot]\,.
\end{equation}
For a quite explicit formula for $d_\bot \Psi^{nls}(\Pi_S z) [ z_\bot], $ we refer to Appendix B.
The map $\Psi_L$ is defined on
$$
{\cal V}^{max} := {\cal V}_S^{max} \times h^s_{\bot c} \,, \qquad {\cal V}_S^{max} := \Pi_S \Phi^{nls}({\cal W})
$$
where ${\cal W} \subseteq H^0_c$ is the domain of definition of the Birkhoff map $\Phi^{nls}$ of Theorem \ref{Theorem Birkhoff coordinates}. Note that 
$$
\R^S \times \R^S \subseteq {\cal V}_S^{max} \subseteq \C^S \times \C^S,  \qquad h^0_r \subset {\cal V}^{max} \subset h^0_c,\qquad \Psi_L(0) = 0\,.
$$
Furthermore, the differential $d \Psi_L(z)$ of $\Psi_L$ at $z = (z_S, z_\bot) \in {\cal V}^{max}$ applied to a vector $\widehat z = (\widehat z_S, \widehat z_\bot) \in h^0_c$ is given by 
\begin{align}
d \Psi_L (z)[\widehat z_S, \widehat z_\bot] & = d_S \Psi^{nls}(\Pi_S z)[\widehat z_S] + 
d_\bot \Psi^{nls}(\Pi_S z)[\widehat z_\bot] + d_S \big( d_\bot \Psi^{nls}(\Pi_S z)[ z_\bot] \big) [\widehat z_S] 
\label{differenziale mappa Psi L 0}\\
& = d \Psi^{nls}(\Pi_S z)[\widehat z] + d^2 \Psi^{nls}(\Pi_S z)[\Pi_S \widehat z, \Pi_\bot z]\,. \label{differenziale mappa Psi L}
\end{align}
Note that the latter expression is independent of $\Pi_\bot \widehat z$ and that by Theorem \ref{Theorem Birkhoff coordinates}, $d \Psi_L(0) = d \Psi^{nls}(0) = F_{nls}^{- 1}$.
First we establish the following auxiliary results.
\begin{lemma}\label{Psi L 1}
$(i)$ The map $\Psi_L : {\cal V}^{max} \to H^0_c$ is analytic and for any $s \in \Z_{\geq 0}$, the restriction $\Psi_L \mid_{h^s_r} :  h^s_r \to H^s_r$ is real analytic. Furthermore, for any $z_S \in \R^S \times \R^S$ and any $s \in \Z_{\geq 0}$, $d \Psi_L(z_S, 0): h^s_c \to H^s_c$ is a linear isomorphism. 

\noindent
$(ii)$ For any compact subset ${\cal K} \subseteq \R^S \times \R^S$, there exists a ball ${\cal V}_\bot$ in $h^0_{\bot r}$, centered at $0$, so that the restriction $\Psi_L : {\cal K}  \times {\cal V}_\bot \to H^0_r$ is one to one. 
Furthermore, after shrinking the radius of the ball ${\cal V}_\bot$, if necessary, the map
$\Psi_L : {\cal K}  \times {\cal V}_\bot \to H^0_r$ is a local diffeomorphism.
\end{lemma}
\begin{proof}
$(i)$ The claimed analyticity follows from the definition of $\Psi_L$ and the corresponding properties of $\Psi^{nls}$, stated in Theorem \ref{Theorem Birkhoff coordinates}. Concerning the statement on the differential $d \Psi_L(z_S, 0)$, note that by \eqref{differenziale mappa Psi L}, $d \Psi_L(z_S, 0) = d \Psi^{nls}(z_S, 0)$ and hence by Theorem \ref{Theorem Birkhoff coordinates}, $d \Psi_L(z_S, 0) : h^s_c \to H^s_c$ is a linear isomorphism for any $s \in \Z_{\geq 0}$. 

\noindent
$(ii)$ Let ${\cal K} \subseteq \R^S \times \R^S$ be a given compact subset. Assume that there exists no ball
 ${\cal V}_\bot$ in $h^0_{\bot r}$, centered at $0$, so that $\Psi_L\mid_{ {\cal K} \times {\cal V}_\bot}$ is 1-1. Then there exist two sequences $z^{(j)} = (z^{(j)}_S, z_\bot^{(j)})$, ${j \geq 1}$, and 
$\tilde z^{(j)} = (\tilde z^{(j)}_S, \tilde z_\bot^{(j)})$, ${j \geq 1}$, in ${\cal K} \times h^0_{\bot r}$ such that for any $j \geq 1$ 
$$
z^{(j)} \neq \tilde z^{(j)}, \quad \Psi_L(z^{(j)}) = \Psi_L (\tilde z^{(j)}), \quad \lim_{j \to \infty} z^{(j)}_\bot  
= \lim_{j \to \infty} \tilde z^{(j)}_\bot = 0 \,.
$$ 
Since by assumption ${\cal K}$ is compact, there exist subsequences of $(z^{(j)})_{j \geq 1}, (\tilde z^{(j)})_{j \geq 1}$, denoted for simplicity in the same way, such that $(z^{(j)}_S)_{j \geq 1}, (\tilde z^{(j)}_S)_{j \geq 1}$ converge. Denote their limits by $z_S^{(\infty)}$ and $\tilde z_S^{(\infty)}$, respectively. 
Then 
$$
\lim_{j \to \infty} z^{(j)} = (z_S^{(\infty)}, 0), \qquad \lim_{j \to \infty} \tilde z^{(j)} = (\tilde z_S^{(\infty)}, 0)
$$
are elements in ${\cal K} \times \{ 0 \}$. By the continuity of $\Psi_L$, one has $\Psi_L(z_S^{(\infty)}, 0) = \Psi_L(\tilde z_S^{(\infty)}, 0)$ and since $\Psi_L$ and $\Psi^{nls}$ coincide on ${\cal V}_S^{max} \times \{ 0 \}$ it then follows from Theorem \ref{Theorem Birkhoff coordinates} that $z_S^{(\infty)} = \tilde z_S^{(\infty)}$. By item $(i)$ and the local inversion theorem one then concludes that in contradiction to our assumption, $z^{(j)} = \tilde z^{(j)}$ for $j$ sufficiently large. This proves the first part of item $(ii)$. 
Since according to item (i), for any given $z_S \in {\cal K}$,
$d \Psi_L(z_S, 0) : h^0_c \to H^0_c$ is a linear isomorphism,   
$d \Psi_L(z)$ is such an operator for $z$  in a whole neighborhood of $(z_S, 0)$. 
Using that ${\cal K}$ is compact it then follows that after shrinking the radius of the ball ${\cal V}_\bot$,
if necessary, $\Psi_L : {\cal K}  \times {\cal V}_\bot \to H^0_r$ is a local diffeomorphism.
\end{proof}
\begin{proposition}\label{Psi L 2}
For any compact subset ${\cal K} \subseteq \R^S \times \R^S$ there exists an open complex neighbourhood ${\cal V}$ of ${\cal K} \times \{ 0 \}$ in $h^0_c$ of the form ${\cal V}_S \times {\cal V}_\bot$ where $\overline{\cal V}_S$ is compact with $\overline{\cal V}_S \subseteq {\cal V}^{max}_S$ and ${\cal V}_\bot \subset h^0_{\bot c}$ 
is an open ball, centered at $0$, so that the restriction of $\Psi_L$ to ${\cal V}$ has the following properties:  

\noindent
$(L1)$ $\Psi_L$ is analytic on ${\cal V}$ and 
\begin{equation}\label{uguaglianza Psi L Psi nls}
\Psi_L \mid_{ {\cal V}_S \times \{ 0 \}} = \Psi^{nls}\mid_{ {\cal V}_S \times \{ 0 \}}\,, \qquad d \Psi_L(z_S, 0) = d \Psi^{nls}(z_S, 0)\,, \quad \forall z_S \in {\cal V}_S\,.
\end{equation} 
Furthermore, $\Psi_L : {\cal V} \cap h^0_r \to H^0_r$ is a real analytic diffeomorphism onto its image.

\noindent
$(L2)$ The map $B_L := \Psi_L - F_{nls}^{- 1} : {\cal V} \to H^0_c$ is analytic and one smoothing. More precisely, the analytic map $B_L$ is given by 
\begin{equation}\label{espressione BL(z)}
B_L (z) = B^{nls}(\Pi_S z) + d_\bot B^{nls}(\Pi_S z)[z_\bot]
\end{equation}
with $B^{nls}$ being the map introduced in Theorem \ref{Theorem Birkhoff coordinates}, and for any $s \in \Z_{\geq 1}$, $B_L : {\cal V} \cap h^s_r \to H^{s + 1}_r$ is real analytic. Furthermore 
$$
d_\bot B_L(z) = d_\bot B^{nls}(\Pi_S z)\,, \qquad d^2_\bot B_L(z) = 0\,, \qquad \forall z \in {\cal V}
$$
and for any $z \in {\cal V} \cap h^0_r$, $\alpha, \beta \in \Z_{\geq 0}^S$, 
\begin{equation}\label{stime BL Lemma 0}
\| \partial_{S}^{\alpha, \beta} B_L(z)\|_{0} \lesssim_{ \a, \b} \,\, 1 \,, \qquad 
\| \partial_{S}^{\alpha, \beta} d_\bot B_L(z) [\widehat z_\bot] \|_{0} 
\lesssim_{ \a, \b} \,  \| \widehat z_\bot\|_0\,, \quad \forall \widehat z_\bot \in h^0_{\bot c} \,
\end{equation}
and for  any $s \in \Z_{\geq 1}$, $z \in {\cal V} \cap h^s_{r}$, 
\begin{equation}\label{stime BL Lemma}
\| \partial_{S}^{\alpha, \beta} B_L(z)\|_{s + 1} \lesssim_{s, \a, \b} \, 1 + \| z_\bot\|_s\,, \qquad 
\| \partial_{S}^{\alpha, \beta} d_\bot B_L(z) [\widehat z_\bot] \|_{s + 1} 
\lesssim_{s, \a, \b} \,  \| \widehat z_\bot\|_s\,, \quad \forall \widehat z_\bot \in h^s_{\bot c} \,.
\end{equation}

\noindent
$(L3)$ For any $s \in \Z_{\geq 1}$, the restriction $\Psi_L \mid_{{\cal V} \cap h^s_r}$ is a map ${\cal V} \cap h^s_r \to H^s_r$ which is a real analytic diffeomorphism onto its image. 

\noindent
$(L4)$ The map $A_L :=  \Psi_L^{- 1} - F_{nls}: \Psi_L({\cal V}) \to h^0_c$ is analytic and one smoothing,  meaning that for any $s \in \Z_{\geq 1}$, $A_L : \Psi_L ({\cal V}) \cap H^s_r \to h^{s + 1}_r$ is real analytic.
\end{proposition}
\begin{remark} For convenience, in the sequel, we always choose ${\cal V}_\bot$ to be a 
ball of radius smaller than one.
\end{remark}

\begin{proof}
Choose ${\cal V}_S$ to be an open bounded neighbourhood of ${\cal K}$ in $\C^S \times \C^S$ so that $\overline{\cal V}_S \subseteq {\cal V}_S^{max}$ and ${\cal V}_\bot$ an open ball in $h^0_{\bot c}$,
centered at $0$, so that item $(ii)$ of Lemma \ref{Psi L 1} applies to ${\cal V} := {\cal V}_S \times {\cal V}_\bot$, implying that $\Psi_L : {\cal V} \cap h^0_r \to H^0_r$, is 1-1 and a local diffeomorphism. The identities \eqref{uguaglianza Psi L Psi nls}  hold by the definition of $\Psi_L$ and the analyticity of $\Psi_L$, stated in $(L1)$, follows by Lemma \ref{Psi L 1}$(i)$. One then concludes that  
$$
\Psi_L : {\cal V} \cap h^0_r \to H^0_r
$$
is a real analytic diffeomorphism onto its image. 
$(L2)$ follows from the definition of $\Psi_L$, Theorem \ref{Theorem Birkhoff coordinates}, the compactness of 
$\overline{\cal V}_S$, and standard estimates in Sobolev spaces. Concerning $(L3)$, first note that
by Theorem \ref{Theorem Birkhoff coordinates}, for any $s \in \Z_{\ge 1}$,
 the restriction $\Psi_L \mid_{{\cal V} \cap h^s_r}$ is a map with values in $ H^s_r$
 and as such real analytic. By item (L1), $\Psi_L \mid_{{\cal V} \cap h^s_r}$ is $1-1$ 
and so is its differential $d \Psi_L (z) :  h^s_c \to H^s_c$ at any point $z \in {\cal V} \cap h^s_r$. 
Since by  $(L2)$ the map $B_L$ is one smoothing,
 $d \Psi_L(z) : h^s_c  \to H^s_c$ is Fredholm and hence a linear isomorphism, implying that
 $\Psi_L : {\cal V} \cap h^s_r \to H^s_r$ is a real analytic diffeomorphism onto its image.
Finally,  item $(L4)$ follows from $(L3)$ and Theorem \ref{Theorem Birkhoff coordinates}.
\end{proof}
\noindent
Whereas the tame estimates \eqref{stime BL Lemma} for $B_L$ are an immediate consequence of the definition of $\Psi_L$, Theorem \ref{Theorem Birkhoff coordinates} and the compactness of $\overline{\cal V}_S$, this is not so for $A_L$. Actually, for the applications in perturbation theory considered in subsequent work,
we only need to derive tame estimates for 
 \begin{equation}\label{cuba 0}
   {\cal A}_L: {\cal V} \cap h^0_r \to {\cal L}(H^0_c, h^0_c)\,, \qquad z \mapsto 
   {\cal A}_L(z) := dA_L(\Psi_L(z)) =  d \Psi_L(z)^{- 1} - F_{nls}
 \end{equation}
 with $\cal V$ denoting the neighborhood of ${\cal K} \times  \{ 0 \}$ of Proposition \ref{Psi L 2}.
 By formula \eqref{differenziale mappa Psi L}, for any $z \in {\cal V} \cap h^0_r,$ the operator
 $d \Psi_L(z) \in {\cal L}(h^0_c, H^{0}_c)$ can be written as 
 \begin{equation}\label{peter pan 0}
d \Psi_L(z) = {\cal T}(z) + {\cal R}(z)\, , \qquad   {\cal T}(z) := d \Psi^{nls}(\Pi_S z)
\end{equation}
with ${\cal R}(z) \in {\cal L}(h^0_c ,  H^{0}_c)$ given by
\begin{equation}\label{definizione cal R (w) nel lemma}
{\cal R}(z) : h^0_c \to  H^{0}_c\,, \quad  \widehat z \mapsto  {\cal R}(z)[\widehat z] := d^2 \Psi^{nls}(\Pi_S z)[\Pi_S \widehat z, \Pi_\bot z]= d^2 B^{nls}(\Pi_S z)[\Pi_S \widehat z, \Pi_\bot z]\,.
\end{equation}
Since  by Theorem \ref{Theorem Birkhoff coordinates}, respectively Proposition \ref{Psi L 2},
the operators ${\cal T}(z)$,  $d \Psi_L(z) : h^0_c \to  H^{0}_c$ are invertible, so is
${\cal T}(z)^{-1} d \Psi_L(z)= {\rm Id} + {\cal T}(z)^{- 1}{\cal R}(z)$, implying that
 \begin{equation}\label{peter pan 1}
d \Psi_L(z)^{-1}= \big( {\rm Id} + {\cal T}(z)^{- 1}{\cal R}(z) \big) ^{-1}{\cal T}(z)^{-1}
= {\cal T}(z)^{-1} - {\cal T}(z)^{- 1}{\cal R}(z) {\cal S}(z)
\end{equation}
where 
 \begin{equation}\label{definizione cal S (w) 0}
 {\cal S}(z) := \big( {\rm Id} + {\cal T}(z)^{- 1}{\cal R}(z) \big)^{- 1} {\cal T}(z)^{-1} \in {\cal L}(H^0_c, h^0_c) \,.
 \end{equation}
Furthermore, by Theorem \ref{Theorem Birkhoff coordinates}
$$
{\cal T}(z)^{- 1} = \big( d \Psi^{nls}(\Pi_S z) \big)^{- 1} = 
d \Phi^{nls} (\Pi_S z) = F_{nls} + d A^{nls}(\Psi^{nls}(\Pi_S z))\,.
$$
Altogether, it follows that for any $z \in {\cal V} \cap h^0_r$, 
the operator ${\cal A}_L(z) = d \Psi_L(z)^{-1} - F_{nls} : h^0_c \to  H^{0}_c$
can be written as
 \begin{equation}\label{forma finale d Psi L (w)}
  {\cal A}_L(z) = d A^{nls}(\Psi^{nls}(\Pi_S z)) - {\cal T}(z)^{- 1} {\cal R}(z){\cal S}(z) \,. 
 \end{equation}
 Finally we note that by $(L4)$ of Proposition \ref{Psi L 2}, $ {\cal A}_L = d A_L \circ \Psi_L$ is one smoothing.
 More precisely, for any $s \in \Z_{\ge 1}$, the restriction of $ {\cal A}_L$ to
 ${\cal V} \cap h^s_r$ is a real analytic map,
 $$
  {\cal A}_L: {\cal V} \cap h^s_r \to {\cal L}( H^s_c,\, h^{s+1}_c)\,, \,\,\, z \mapsto {\cal A}_L(z)\,.
 $$
\begin{proposition}[Tame estimates for ${\cal A}_L$] \label{estimates of d Psi L inv}
After shrinking, if necessary, the radius of the ball $\cal V_\bot$ in ${\cal V} = {\cal V}_S \times {\cal V}_\bot$
of Proposition \ref{Psi L 2}, the map $ {\cal A}_L$ satisfies  for any 
$z \in {\cal V} \cap h^0_r$, $\widehat w \in h^0_c$,
$$
\| {\cal A}_L(z)[\widehat w] \|_0 \lesssim \| \widehat w\|_0
$$
and for any $k \in \Z_{\geq 1}$, $\widehat z_1, \ldots, \widehat z_k \in h^0_c$,  
$$
\| d^k \big( {\cal A}_L(z)[\widehat w] \big)[\widehat z_1, \ldots, \widehat z_k] \|_0 \lesssim_k \| \widehat w\|_0 \prod_{j = 1}^k \| \widehat z_j\|_0\,.
$$
Furthermore, for any $s \in \Z_{\ge 1}$, $z \in {\cal V} \cap h^s_r$, $\widehat w  \in H^s_c$,
\begin{equation}\label{estimate for A L}
 \| {\cal A}_L(z)[\widehat w] \|_{s+1} \lesssim_{s}  \| z_\bot\|_s \|\widehat w \|_0  +  \| \widehat w\|_s
\end{equation}
and for any $k \ge 1,$ $\widehat z_1, \ldots, \widehat z_k \in h^s_c$, 
\begin{align}\label{estimate for d A L}
\| d^k \big( {\cal A}_L(z)[\widehat w] \big)[\widehat z_1, \ldots, \widehat z_k] \|_{s+1} \lesssim_{s, k}
\,\, \big(   \| z_\bot\|_s \|\widehat w \|_0  +  \| \widehat w\|_s \big) \, \prod_{j = 1}^k \| \widehat z_j\|_0  \, \, +   \| \widehat w \|_0 \, \sum_{j = 1}^k  \| \widehat z_j\|_s \prod_{i \neq j} \| \widehat z_i\|_0  \,.
 \end{align}  
\end{proposition}
\begin{proof}
First we prove estimate \eqref{estimate for A L}.
The starting point is formula \eqref{forma finale d Psi L (w)} for ${\cal A}_L(z)$. The two terms 
$d A^{nls}(\Psi^{nls}(\Pi_S z))$ and ${\cal T}(z)^{- 1} {\cal R}(z){\cal S}(z)$ 
are estimated separately. By Theorem \ref{Theorem Birkhoff coordinates}, 
 $\{ \Psi^{nls}(\Pi_S z ) \,  | \, z  \in {\cal V} \cap h^0_r \}$ 
is a relatively compact subset of $ H^s_r$ for any $s \in \Z_{\geq 0}$,
and $ A^{nls}$, $ B^{nls}$ are one smoothing maps. It implies that 
for any $s \in \Z_{\geq 1}$, 
 \begin{equation}\label{peter pan 3}
 \| d A^{nls}(\Psi^{nls}(\Pi_S z))[\widehat w] \|_{s + 1} \lesssim_s \| \widehat w\|_s\,, 
 \qquad \forall \, z \in {\cal V} \cap h^0_r, \quad \forall \, \widehat w \in H^s_c\,.
 \end{equation}
 Since $\| \Pi_S \widehat z\|_s \lesssim_s \| \Pi_S \widehat z\|_0$ for any $z \in h^0_c$, 
 the linear operator ${\cal R}(z)$, defined in \eqref{definizione cal R (w) nel lemma}, satisfies 
  \begin{equation}\label{one smoothing estimate R}
\| {\cal R}( z)[\widehat z] \|_{s + 1} \lesssim_s \,
 \| z_\bot\|_s \| \Pi_S \widehat z\|_s  \lesssim_s \,
\| z_\bot\|_s \| \Pi_S \widehat z\|_0\
  \qquad \forall z \in {\cal V} \cap h^0_r, \quad \forall  \, \widehat z \in h^0_c\,.
\end{equation}
Furthermore, also by Theorem \ref{Theorem Birkhoff coordinates},
one has for any $s \in \Z_{\geq 0}$,
 \begin{equation}\label{estimate T inverse}
 \| T(z)^{-1}[\widehat w] \|_{s} \lesssim_s \| \widehat w\|_s\,, 
 \qquad \forall \, z \in {\cal V} \cap h^0_r, \quad \forall \, \widehat w \in H^s_c\,.
 \end{equation}
 Combining \eqref{peter pan 3}-\eqref{estimate T inverse}, formula \eqref{forma finale d Psi L (w)} leads to the estimate 
 \begin{align}
 \|{\cal A}_L(z)[\widehat w] \|_{s + 1} & \leq \| d A^{nls}(\Psi^{nls}(\Pi_S z))[\widehat w]\|_{s + 1} + \| {\cal T}(z)^{- 1} {\cal R}(z) {\cal S}(z)[\widehat w]\|_{s + 1} \nonumber\\
 & \lesssim_s \| \widehat w\|_s + \| z_\bot\|_s \| \Pi_S {\cal S}(z)[\widehat w]\|_0\,. \label{first estimate cal AL (z) in lemma}
 \end{align}
It remains to estimate $\| {\cal S}(z)[\widehat w]\|_0$. Recall that by \eqref{definizione cal S (w) 0}, $ {\cal S}(z) = \big( {\rm Id} + {\cal T}(z)^{- 1}{\cal R}(z) \big)^{- 1} {\cal T}(z)^{-1}$. By Theorem \ref{Theorem Birkhoff coordinates} there exists $C_0 > 0$ so that 
  \begin{equation}\label{peter pan 10}
 \| {\cal T}(z)^{- 1} {\cal R}(z)[\widehat z]\|_0 \leq 
 C_0 \| z_\bot\|_0 \| \Pi_S \widehat z\|_0 \,  
 \qquad \forall z \in {\cal V} \cap h^0_r, \quad \forall  \, \widehat z \in h^0_c\,.
 \end{equation}
Shrinking the radius of the ball ${\cal V}_\bot$ in $ h^0_{\bot c}$ , if necessary,
  so that $C_0 \| z_\bot\|_0 \leq 1/2$ for any $z_\bot \in {\cal V}_\bot$, the Neumann series of the operator 
  $\big( {\rm Id} + {\cal T}(z)^{- 1}{\cal R}(z) \big)^{- 1}$ absolutely converges in ${\cal L}(h^0_c, h^0_c)$ 
and the operator norm of $\big( {\rm Id} + {\cal T}(z)^{- 1}{\cal R}(z) \big)^{- 1}$ in ${\cal L}(h^0_c, h^0_c)$ is bounded by $2$.
Hence   
 \begin{equation}\label{definizione cal S (w)}
\| {\cal S}(z)[\widehat w] \|_0 
\lesssim_s \| \widehat w\|_0\,, 
 \qquad \forall \, z \in {\cal V} \cap h^0_r, \quad \forall \, \widehat w \in H^0_c\,,
 \end{equation}
implying together with \eqref{first estimate cal AL (z) in lemma} the claimed estimate \eqref{estimate for A L}. 

\noindent
Finally let us prove the estimate \eqref{estimate for d A L} for the derivatives of ${\cal A}_L(z)$. By formula \eqref{forma finale d Psi L (w)} for any $k, s \in \Z_{\geq 1}$, 
$z \in {\cal V } \cap h^s_r$, $\widehat w \in H^s_c$, and $\widehat z_1, \ldots, \widehat z_k \in h^s_c$, 
\begin{align}
 \|d^k \big({\cal A}_L(z)[\widehat w] \big)[\widehat z_1, \ldots, \widehat z_k]\|_{s + 1} & \leq \| d^k \big( dA^{nls}(\Psi^{nls}(\Pi_S z))[\widehat w]\big) [\widehat z_1, \ldots, \widehat z_k]\|_{s + 1} \nonumber\\
 & \quad + \| d^k \big({\cal T}(z)^{- 1} {\cal R}(z) {\cal S}(z)[\widehat w] \big)[\widehat z_1, \ldots, \widehat z_k]\|_{s + 1}\,.  \label{first estimate cal AL (z) in lemma 1}
 \end{align}
By Theorem \ref{Theorem Birkhoff coordinates}, one concludes that 
 \begin{equation}\label{peter pan 3 k ppp}
 \| \, d^k \big( d A^{nls}(\Psi_L(\Pi_S z))[\widehat w]  \big) 
 [\widehat z_1, \ldots, \widehat z_k] \, \|_{s + 1}  \lesssim_{s, k}  
  \| \widehat w\|_s   \, \prod_{j = 1}^k \| \widehat  z_j\|_0 \,.
 \end{equation}
Furthermore 
\begin{equation}\label{peter pan 3 k pp}
 \| \, d^k \big( {\cal T}( z)^{-1}[\widehat w] \big) 
 [\widehat z_1, \ldots, \widehat z_k] \, \|_{s }  \lesssim_{s, k}  
  \| \widehat w\|_s   \, \prod_{j = 1}^k \|  \widehat  z_j\|_0 \,,
 \end{equation}
 \begin{align}\label{peter pan 3 k p}
 \| \, d^k \big( {\cal R}( z)[\widehat z] \big) &
 [\widehat z_1, \ldots, \widehat z_k] \, \|_{ s+1}  \lesssim_{s, k}  \,\,
  \| z_\bot \|_s   \| \widehat  z \|_0 \, \prod_{j = 1}^k \| \widehat  z_j\|_0 \, \,
  +  \,\,  \|  \widehat  z \|_0 \,   \sum_{j = 1}^k \|  \widehat z_j \|_s \prod_{i \ne j} \|  \widehat  z_j\|_0 \, ,
 \end{align}
 and
  \begin{equation}\label{definizione cal S derivate (w)}
\| d^k \big( {\cal S}(z)[\widehat w] \big)[\widehat z_1, \ldots, \widehat z_k] \|_0 
\lesssim_s \| \widehat w\|_0  \prod_{j = 1}^k \| \widehat z_j\|_0 \,.
 \end{equation}
 Combining the estimates \eqref{peter pan 3 k pp}-\eqref{definizione cal S derivate (w)} and using the product rule implies that 
 \begin{align}
 \| & d^k \big({\cal T}(z)^{- 1} {\cal R}(z)  {\cal S}(z)[\widehat w] \big)[\widehat z_1, \ldots, \widehat z_k]\|_{s + 1}  \, \, \lesssim_{s, k} \, \, \nonumber\\
 & \big(   \| z_\bot\|_s \|\widehat w \|_0  +  \| \widehat w\|_s \big) \, \prod_{j = 1}^k \| \widehat z_j\|_0  
  \,  +  \,  \| \widehat w \|_0 \, \sum_{j = 1}^k  \| \widehat z_j\|_s \prod_{i \neq j} \| \widehat z_i\|_0 \,.
  \label{TRS derivate}
 \end{align}
The three estimates \eqref{first estimate cal AL (z) in lemma 1}, \eqref{peter pan 3 k ppp}, \eqref{TRS derivate} 
together yield \eqref{estimate for d A L}.  
 \end{proof}
In the remaining part of this section we describe the pullback $\Psi_L^* \Lambda$ by $\Psi_L$ of the standard symplectic form $\Lambda$ on $H^0_r$, introduced in \eqref{definizione 2 forma L2}. It turns out that $\Psi_L^* \Lambda$ is not the symplectic form $\Lambda_M$ of \eqref{definizione due forma h0 c}, making it necessary to construct the symplectic corrector $\Psi_C$ -- see Section \ref{costuzione mappa Psi simplettica} 
below. 

\noindent
Given a bounded linear operator ${\cal P} : h^0_c \to h^0_c$, its transpose ${\cal P}^t : h^0_c\to h^0_c$ is defined to be the operator determined by  
\begin{equation}\label{trasposto successioni}
\big( {\cal P}[\widehat z], \widehat z' \big)_r = \big(\widehat z, {\cal P}^t [\widehat z'] \big)_r\,, \quad \forall \widehat z, \widehat z' \in h^0_c\, ,
\end{equation}
where the blilinear form $(\cdot, \cdot)_r$ on $h^0_c$ is defined in \eqref{forma bilineare h 0c}. Similarly, for a bounded linear operator ${\cal Q} : h^0_c \to H^0_c$, we denote its transpose by ${\cal Q }^t : H^0_c \to h^0_c$, determined by 
\begin{equation}\label{trasposto funzioni successioni}
\langle {\cal Q}[\widehat z], \widehat w\rangle_r = \big(\widehat z\,,\, {\cal Q}^t [\widehat w] \big)_r\,, \qquad \forall \widehat z \in h^0_c\,,\,\widehat w \in H^0_c\,,
\end{equation}
where the bilinear form $\langle \cdot, \cdot \rangle_r$ on $H^0_c$ is the one introduced in 
\eqref{definizione forma bilineare H0 c}. 
  We now compute the pullback $\Psi_L^*\Lambda(z)$ at $z = (z_S, z_\bot) \in h^0_c$ applied to 
  $ \widehat z = ( \widehat z_S, \widehat z_\bot)$, $ \widehat z' = ( \widehat z_S', \widehat z_\bot')$. 
  By the definition of the pullback and the one of $\Lambda$ in \eqref{definizione 2 forma L2}
  we have 
  \begin{align}
  & \Psi_L^* \Lambda({z})[ \widehat z, \widehat z']  = 
  \Lambda(\Psi_L(z)) \, [d \Psi_L(z)[\widehat z], \, d \Psi_L(z)[\widehat z'] ] = 
  \ii \langle {\mathbb J} d \Psi_L(z)[\widehat z] , \, d \Psi_L(z)[\widehat z']\rangle_r\,. 
  \end{align}
  By formula \eqref{differenziale mappa Psi L} for $d \Psi_L(z)$, 
  $$
  d \Psi_L(z)[\widehat z] = d \Psi^{nls}(\Pi_S z)[ \widehat z] + d_S \big( d_\bot \Psi^{nls}(\Pi_S z)[  z_\bot]  \big)  [ \widehat z_S] \,,
  $$ 
  one gets 
  \begin{align}
   &\Psi_L^* \Lambda({z})[ \widehat z, \widehat z']   = (I) + (II) + (III) + (IV) 
   \end{align}
   where 
   \begin{align}
   (I) := &\ii \Big\langle {\mathbb J} d \Psi^{nls}(\Pi_S z)[\widehat z], \, d \Psi^{nls}(\Pi_S z)[\widehat z'] \Big\rangle_r 
   = ((\Psi^{nls})^* \Lambda )(\Pi_S z)[\widehat z, \widehat z']\,, \label{cavolo romano 1} 
   \end{align}
   \begin{align}
  (I I) := & \ii \Big\langle {\mathbb J}  d \Psi^{nls}(\Pi_S z)[ \widehat z] , \, 
   d_S \big( d_\bot \Psi^{nls}(\Pi_S z)[  z_\bot]  \big)  [\widehat z_S'] \Big\rangle_r\,.\label{cavolo romano 2}
  \end{align}
  Writing $d \Psi^{nls}(\Pi_S z)[\widehat z]$ as $d_S \Psi^{nls}(\Pi_S z)[\widehat z_S] + d_\bot \Psi^{nls}(\Pi_S z)[\widehat z_\bot]$ one gets
  \begin{align}
  (II) & = \ii \Big\langle {\mathbb J}  d_S \Psi^{nls}(\Pi_S z)[ \widehat z_S] \,,\, d_S \big( d_\bot \Psi^{nls}(\Pi_S z)[ z_\bot]  \big) [  \widehat z_S' ] \Big\rangle_r \nonumber\\
    &  + \ii \Big\langle {\mathbb J}  d_\bot \Psi^{nls}(\Pi_S z)[ \widehat z_\bot] \,,\, d_S \big( d_\bot \Psi^{nls}(\Pi_S z)[ z_\bot]  \big) [  \widehat z_S']  \Big\rangle_r\,. \label{cavolo romano 2}
  \end{align}
  Similarly one has 
  \begin{align}
  (III) := & \ii \Big\langle {\mathbb J} d_S \big( d_\bot \Psi^{nls}(\Pi_S z)[  z_\bot]  \big) [\widehat z_S] \,,\,  d \Psi^{nls}(\Pi_S z)[\widehat z']  \big) \Big\rangle_r \nonumber\\
  & = \ii \Big\langle {\mathbb J} d_S \big( d_\bot \Psi^{nls}(\Pi_S z)[  z_\bot]  \big) [ \widehat z_S] \,,\,  d_S \Psi^{nls}(\Pi_S z)[\widehat z_S']   \Big\rangle_r \nonumber \\
  & +   \ii \Big\langle {\mathbb J} d_S \big( d_\bot \Psi^{nls}(\Pi_S z)[  z_\bot]  \big) [\widehat z_S] \,,\,  d_\bot \Psi^{nls}(\Pi_S z)[\widehat z'_\bot]   \Big\rangle_r \label{cavolo romano 3}
     \end{align}
     and finally 
  \begin{align}
  (I V) := & \ii \Big\langle {\mathbb J}  d_S \big( d_\bot \Psi^{nls}(\Pi_S z)[  z_\bot]  \big) [\widehat z_S] \,,\,   d_S \big( d_\bot \Psi^{nls}(\Pi_S z)[ z_\bot]  \big) [\widehat z_S']  \Big\rangle_r \label{cavolo romano 4}
   \end{align}
 Since by Theorem \ref{Theorem Birkhoff coordinates}, $\Psi^{nls}$ is symplectic, one has $(\Psi^{nls})^* \Lambda = \Lambda_M$. Hence for any $z \in {\cal V}$, $\Psi_L^* \Lambda(z)$ can be written as   
 \begin{equation}\label{definizione delta omega Omega}
\Psi_L^* \Lambda(z) = \Lambda_M +  \Lambda_L(z)\,, \quad  
\Lambda_L(z) [ \widehat z, \widehat z'] := \big( L(z)[\widehat z],  \widehat z' \big)_r \,, 
 \end{equation}
 where $L(z) : \C^S \times \C^S \times h^0_{\bot c} \to \C^S \times \C^S \times h^0_{\bot c}$ is the linear operator of the form  
 \begin{equation}\label{definizione Omega(w)}
 L(z) = \begin{pmatrix}
 L_S^S(z) & L_{S}^\bot(z) \\
 L_\bot^S(z) & 0
 \end{pmatrix}\,.
 \end{equation}
 By the computations above, $L_S^S(z) : \C^S \times \C^S \to \C^S \times \C^S$, $L_S^\bot(z) : h^0_{ \bot c} \to \C^S \times \C^S$, and $L_\bot^S(z) : \C^S \times \C^S \to h^0_{ \bot c}$ are the linear operators defined by 
 ($z \in {\cal V} \cap h^0_r, \, \widehat z_S \in \C^S \times \C^S, \, \widehat z_\bot \in h^0_{\bot c}$)  
 \begin{align}
 & L_S^S(z)[\widehat z_S] := 
 \ii \begin{pmatrix} \big(\big\langle {\mathbb J}  d_S \Psi^{nls}(\Pi_S z)[ \widehat z_S] \,,\,
 \partial_{x_j } d_\bot \Psi^{nls}(\Pi_S z)[ z_\bot]   \big\rangle_r \big)_{j \in S}  \\
 \big(\big\langle {\mathbb J}  d_S \Psi^{nls}(\Pi_S z)[ \widehat z_S] \,,\,
 \partial_{y_j } d_\bot \Psi^{nls}(\Pi_S z)[ z_\bot]   \big\rangle_r \big)_{j \in S} 
 \end{pmatrix} \nonumber\\
 & + \ii \begin{pmatrix}
 \big( \big\langle {\mathbb J} d_S \big( d_\bot \Psi^{nls}(\Pi_S z)[  z_\bot]  \big) [ \widehat z_S] \,,\,  \partial_{x_j} \Psi^{nls}(\Pi_S z)  \big\rangle_r \big)_{j \in S} \\
 \big( \big\langle {\mathbb J} d_S \big( d_\bot \Psi^{nls}(\Pi_S z)[  z_\bot]  \big) [ \widehat z_S] \,,\,  \partial_{y_j} \Psi^{nls}(\Pi_S z)  \big\rangle_r \big)_{j \in S}
 \end{pmatrix} \nonumber\\
 & + \ii \begin{pmatrix} \big(\big\langle {\mathbb J} d_S \big( d_\bot \Psi^{nls}(\Pi_S z)[  z_\bot]  \big) [\widehat z_S]\,,\,   \partial_{x_j}  d_\bot \Psi^{nls}(\Pi_S z)[ z_\bot]     \big\rangle_r \big)_{j \in S} \\
 \big(\big\langle {\mathbb J} d_S \big( d_\bot \Psi^{nls}(\Pi_S z)[  z_\bot]  \big) [\widehat z_S]\,,\,   \partial_{y_j}  d_\bot \Psi^{nls}(\Pi_S z)[ z_\bot]     \big\rangle_r \big)_{j \in S} 
 \end{pmatrix}\, \label{definizione L SS} 
 \end{align}
and similarly 
 \begin{align}
 L_S^\bot(z)[\widehat z_\bot] & :=  
 \ii \begin{pmatrix}\big(\big\langle {\mathbb J}  d_\bot \Psi^{nls}(\Pi_S z)[ \widehat z_\bot] \,,\, \partial_{x_j}  d_\bot \Psi^{nls}(\Pi_S z)[ z_\bot]    \big\rangle_r \big)_{j \in S} \\
 \big(\big\langle {\mathbb J}  d_\bot \Psi^{nls}(\Pi_S z)[ \widehat z_\bot] \,,\, \partial_{y_j}  d_\bot \Psi^{nls}(\Pi_S z)[ z_\bot]    \big\rangle_r \big)_{j \in S}
 \end{pmatrix}\,, \label{definizione L S bot} \\
 L_\bot^S(z)[\widehat z_S] & := \ii \begin{pmatrix}
 \big(\big\langle {\mathbb J} d_S \big(d_\bot \Psi^{nls}(\Pi_S z) [z_\bot] \big)[\widehat z_S]\,,\, \partial_{x_j} \Psi^{nls}(\Pi_S z) \big\rangle_r \big)_{j \in S^\bot} \\
  \big(  \big\langle {\mathbb J} d_S \big(d_\bot \Psi^{nls}(\Pi_S z) [z_\bot] \Big)[\widehat z_S]\,,\, \partial_{y_j} \Psi^{nls}(\Pi_S z) \big\rangle_r \big)_{j \in S^\bot}
 \end{pmatrix}\,.\label{definizione L bot S}
 \end{align}
 The operator valued map $z \mapsto L(z)$ has the following properties:
 \begin{lemma}\label{stime Omega (w)}
The map $L : {\cal V} \cap h^0_r \to {\cal L}(h^0_c, h^0_c)$, $z \mapsto L(z)$ is real analytic.
For any $z \in {\cal V} \cap h^0_r$, $\widehat z \in h^0_c$,
$$
\| L(z)[\widehat z]\|_0 \lesssim \|z_\bot\|_0 \| \widehat z\|_0
$$
and for any $k \in \Z_{\geq 1}$, $\widehat z_1, \ldots, \widehat z_k \in h^0_c$, 
$$
\| d^k \big( L(z)[\widehat z] \big)[\widehat z_1, \ldots, \widehat z_k] \|_{0} \lesssim_k  \| \widehat z\|_0 \prod_{j = 1}^k \| \widehat z_j\|_0\,.
$$
Furthermore, the map $L$ is one smoothing, meaning that for any $s \in \Z_{\geq 1}$,  $L : {\cal V} \cap h^s_r \to {\cal L}(h^0_c, h^{s + 1}_c)$, $z \mapsto L(z)$ is real analytic and satisfies the following estimates:
for any $z \in {\cal V} \cap h^s_r$, $\widehat z \in h^s_c$,
\begin{equation}\label{estimates for L(z)}
 \| L (z) [\widehat z] \|_{s + 1} \lesssim_s \| z_\bot \|_s \| \widehat z\|_0
 \end{equation}
and for any $k \in \Z_{\geq 1}$, $z \in {\cal V} \cap h^s_r$, $ \widehat z_1, \ldots, \widehat z_k \in h^s_c$,
\begin{equation}\label{estimates of derivatives of L(z)}
\| d^k \big( L(z)[\widehat z] \big)[\widehat z_1, \ldots, \widehat z_k] \|_{s + 1} \lesssim_{s, k} 
\| \widehat z\|_0\sum_{j = 1}^k \| \widehat z_j\|_s \prod_{i \neq j} \| \widehat z_i\|_0  +
\| \widehat z\|_0 \| z_\bot\|_s \prod_{j = 1}^k \| \widehat z_j\|_0 \,.
\end{equation}
 In particular, $L(z) = 0$ for any $z \in {\cal V} \cap h^0_r$ with $z_\bot = 0$.
Finally, $L(z) = - L(z)^t$ or, more explicitly, for any
$ z \in {\cal V} \cap h^0_r,$
 \begin{equation}\label{proprieta blocchi L(w)}
 L_S^S(z)^t = - L_S^S(z),\quad L_S^\bot(z)^t = - L_\bot^S(z), \quad L_\bot^S(z)^t = - L_S^\bot(z)\,.
 \end{equation}
 \end{lemma}
 \begin{proof}
 The analyticity of $L$ follows by Theorem \ref{Theorem Birkhoff coordinates}, using again that $d_Sd_\bot\Psi^{nls} = d_S d_\bot B^{nls}$. 
 Since ${\mathbb J}^t = - {\mathbb J}$, one reads off from the expressions \eqref{definizione L SS}-\eqref{definizione L bot S} that \eqref{proprieta blocchi L(w)} holds. 
The estimates \eqref{estimates for L(z)} and \eqref{estimates of derivatives of L(z)} follow from 
Theorem \ref{Theorem Birkhoff coordinates} by differentiating the expressions in the definitions of 
$L_S^S(z)$, $L_S^\bot(z)$, and $L_\bot^S(z)$ with respect to $z$. 
 \end{proof}

 \section{The symplectic corrector $\Psi_C$}\label{costuzione mappa Psi simplettica}
 In this section we construct the coordinate transformation $\Psi_C$ on ${\cal V} \cap h^0_r$ so that the composition $\Psi_L \circ \Psi_C$ is symplectic. As mentioned in the introduction, we follow Kuksin's scheme of proof in \cite{K},
 which uses arguments of Moser and Weinstein in the given infinite dimensional setup.
 The map $\Psi_C$ will be defined as the time one flow of an appropriately chosen non autonomous vector field. 
In the sequel, ${\cal V}$ denotes the neighborhood of ${\cal K}\times {0}$, given by Proposition \ref{Psi L 2}
and Proposition \ref{estimates of d Psi L inv}.

For any $z \in {\cal V}$ define the following two and one forms on $h^0_c$,
 \begin{equation}\label{definizione omega 0 1}
\Lambda_0 := \Lambda_M\,, \qquad \Lambda_1 (z) := \Psi_L^* \Lambda(z) = \Lambda_M+ \Lambda_L(z)\,,
 \end{equation}
 \begin{equation}\label{definizione lambda 0 1}
 \lambda_0 := \lambda_M\,, \qquad \lambda_1 (z) := \Psi_L^* \lambda(z) \,.
 \end{equation}
{\bf Analysis of the two form $\Lambda_1 (z)$:} 
Note that $d \lambda_i = \Lambda_i$, $i = 0,1$, and 
\begin{equation}\label{closeness Lambda 1 0}
\Lambda_1 - \Lambda_0 = \Lambda_L = d(\lambda_1 - \lambda_0)\,.
\end{equation}
In particular, the two form $\Lambda_L$ is closed. 
By \eqref{definizione due forma h0 c}, \eqref{definizione delta omega Omega} one has 
$$
\Lambda_1(z) [\widehat z, \widehat z'] = \big( {\cal L}_1(z)[\widehat z], \, \widehat z' \big)_r\,, \qquad {\cal L}_1(z) := J^{- 1} + L(z)\,.
$$
For any $\tau \in [0, 1]$, define the two form $\Lambda_\tau = \Lambda_\tau(z)$,
\begin{equation}\label{definizione Lambda tau}
\Lambda_\tau := \tau \Lambda_1 + (1 - \tau) \Lambda_0\,,
\end{equation}
which can be written as 
\begin{equation}\label{rappresentazione omega t}
\Lambda_\tau(z)[\widehat z, \widehat z'] = \big( {\cal L}_\tau(z)[\widehat z], \, \widehat z' \big)_r\,, \qquad {\cal L}_\tau(z) = J^{- 1} + \tau L(z)\,.
\end{equation}

\noindent
 It turns out that for any $\tau \in [0, 1]$ and $z \in {\cal V} \cap h^0_r$, the map ${\cal L}_\tau(z)$ is invertible and one smoothing. More precisely, the following holds:
  \begin{lemma}\label{stime cal Lt inverso}
After shrinking the ball ${\cal V}_\bot \subset h^0_{\bot c}$ in ${\cal V} = {\cal V}_S \times {\cal V}_\bot$, if necessary, one has that for any $s \in \Z_{\geq 0}$, $z \in {\cal V} \cap h^s_r$, and $\tau \in [0, 1]$, the operator ${\cal L}_\tau(z) : h^s_c \to h^s_c$ is invertible and 
for any $k \in \Z_{\geq 1}$, $z \in {\cal V} \cap h^0_r$, $\widehat z, \widehat z_1, \ldots, \widehat z_k \in h^0_c$,
$$
 \| ({\cal L}_\tau (z)^{- 1} - J) [\widehat z] \|_{0} \lesssim \| z_\bot \|_0 \| \widehat z\|_0\,, \qquad
 \| d^k \big( {\cal L}_\tau (z)^{- 1} [\widehat z] \big)[\widehat z_1, \ldots, \widehat z_k] \|_{0} \lesssim_{ k} \| \widehat z\|_0  \prod_{j = 1}^k \| \widehat z_j\|_0\,.
 $$
 Moreover for any $s \in \Z_{\geq 1}$ and $\tau \in [0, 1]$, the map 
$$
{\cal L}_\tau^{- 1} - J : {\cal V} \cap h^s_r \to {\cal L}(h^s_c, h^{s + 1}_c), \quad z \mapsto {\cal L}_\tau(z)^{- 1} - J
$$
is real analytic and the following tame estimates hold: for any $k \in \Z_{\geq 1}$, $z \in {\cal V} \cap h^s_r$, 
$\widehat z, \widehat z_1, \ldots, \widehat z_k \in h^s_c$,
$$
 \| ({\cal L}_\tau (z)^{- 1} - J) [\widehat z] \|_{s + 1} \lesssim_s \| z_\bot \|_s \| \widehat z\|_0\,,
 $$
 $$
 \| d^k \big( {\cal L}_\tau (z)^{- 1} [\widehat z] \big)[\widehat z_1, \ldots, \widehat z_k] \|_{s + 1} \lesssim_{s, k} \| \widehat z\|_0 \sum_{j = 1}^k \| \widehat z_j\|_s \prod_{i \neq j} \| \widehat z_i\|_0  +  \| \widehat z\|_0 \| z_\bot\|_s \prod_{j = 1}^k \| \widehat z_j\|_0\,.
 $$

 \end{lemma}
 \begin{proof} 
 For any $\tau \in [0, 1]$, we write 
 $$
 {\cal L}_\tau(z) = J^{- 1} \big({\rm Id} + L_\tau(z) \big)\,, \qquad L_\tau(z) := \tau J L(z)\,.
 $$
 By  \eqref{definizione Omega(w)} and Theorem \ref{Theorem Birkhoff coordinates}, the operator $L_\tau (z)$ satisfies the estimate $\| L_\tau(z)[\widehat z]\|_0 \leq C_0 \| z_\bot\|_0 \| \widehat z\|_0$, for any $z \in {\cal V} \cap h^0_r$ and $\widehat z \in h^0_c$ for some constant $C_0 > 0$. By shrinking the ball ${\cal V}_\bot$, if necessary, one has that for any $z_\bot \in {\cal V}_\bot$, $C_0\| z_\bot\|_0 \leq 1/2$, implying that the operator ${\cal L}_\tau(z)$ is invertible and its inverse ${\cal L}_\tau(z)^{- 1}$ is given by the Neumann series 
 \begin{equation}\label{definizione Lt (w)}
 {\cal L}_\tau(z)^{- 1} =  J + \sum_{n \geq 1}(- 1)^n L_\tau(z)^n J\,.
 \end{equation}
 By Lemma \ref{stime Omega (w)}, for any $s, n \in \Z_{\geq 1}$ and $\tau \in [0, 1]$,  one has  
 \begin{align}
 \| L_\tau(z)^n J[\widehat z] \|_{s + 1} & \leq C(s) \| z_\bot\|_s \| L_\tau(z)^{n - 1} J[\widehat z]\|_0 \leq C(s) (C_0 \| z_\bot\|_0)^{n - 1} \| z_\bot\|_s \| \widehat z\|_0\,
 \end{align}
 for some constant $C(s) > 0$.
 Since $C_0 \| z_\bot\|_0 \leq 1/2$, one gets  
 $$
 \|({\cal L}_\tau(z)^{- 1} - J)[\widehat z] \|_{s + 1} \lesssim_s \| z_\bot\|_s \| \widehat z\|_0\,.
 $$
 The estimates for the derivatives $d^k \big({\cal L}_\tau(z)^{- 1}[\widehat z]\big)$ follow by differentiating the expression \eqref{definizione Lt (w)} with respect to $z$ and applying the estimates for $d^k\big( L(z)[\widehat z]\big)$ of Lemma \ref{stime Omega (w)}. 
 \end{proof}

\noindent
Since by \eqref{closeness Lambda 1 0}, the two form $\Lambda_L = \Lambda_1 - \Lambda_0$ is closed and by Lemma  \ref{stime Omega (w)}, for any $z \in {\cal V} \cap h^0_r$, $\Lambda_L (\Pi_S z) = 0$, we can apply Lemma \ref{lemma di poincare} in Appendix A. It says that  the one form    
 \begin{equation}\label{1 forma Kuksin}
 \lambda_L (z) [\widehat z] : = \int_0^1 \Lambda_L (z_S, t z_\bot)[(0, z_\bot), (\widehat z_S, t \widehat z_\bot)]\,d t\,
 \end{equation}
satisfies $d \lambda_L = \Lambda_L$. By \eqref{definizione delta omega Omega}, \eqref{definizione Omega(w)}, the one form $\lambda_L(z)$ can be written as 
 $$
  \lambda_L (z) [\widehat z] =   \int_0^1 \big( L(z_S, t z_\bot)(0, z_\bot), \,  (\widehat z_S, t \widehat z_\bot) \big)_r \, d t \,
=  \int_0^1  L_S^\bot(z_S, t z_\bot)[ z_\bot]\,\cdot\,\widehat z_S  \,d t\,.
 $$
  Moreover, using that by \eqref{definizione L S bot}, $L_S^\bot(z_S, t z_\bot) = t L_S^\bot(z_S,  z_\bot)$, it turns out that 
 \begin{equation}
   \lambda_{L} (z) [\widehat z] = \big({ E}(z) \,,\, \widehat z\big)_r\,, \quad E(z) := (E_S(z), 0)\in \C^S \times \C^S \times h^0_{\bot c}  \label{forma finale 1 forma Kuksin}
 \end{equation}
 where
 \begin{equation}\label{forma finale 1 forma Kuksin 2}
  E_S(z) := \frac12 L_S^\bot(z)[z_\bot] =  \frac{\ii}{2}\begin{pmatrix} \big(\big\langle {\mathbb J}  d_\bot \Psi^{nls}(\Pi_S z)[  z_\bot] \,,\, \partial_{x_j}  d_\bot \Psi^{nls}(\Pi_S z)[ z_\bot]     \big\rangle_r \big)_{j \in S} \\
\big(\big\langle {\mathbb J}  d_\bot \Psi^{nls}(\Pi_S z)[  z_\bot] \,,\, \partial_{y_j}  d_\bot \Psi^{nls}(\Pi_S z)[ z_\bot]     \big\rangle_r \big)_{j \in S}
\end{pmatrix}\,.
 \end{equation}
 One of the features of $\lambda_L(z)$ is that it is quadratic in $z_\bot$. In more detail, we have the following
 \begin{lemma}\label{stima campo vettoriale 1-forma}
 For any $s \in \Z_{\geq 0}$, the map $E : {\cal V} \cap h^0_r \to h^s_r$ is real analytic and satisfies the following tame estimates: for any $z \in {\cal V} \cap h^0_r$, $\widehat z \in h^0_c$,
 $$
 \| E(z) \|_s \lesssim_s \| z_\bot\|_0^2\,, \qquad \| d E(z)[\widehat z] \|_s \lesssim_s \| z_\bot\|_0 \| \widehat z\|_0\,, 
 $$
 and any $k \geq 2$, $\widehat z_1, \ldots, \widehat z_k \in h^0_c$,
 $$
\| d^k E(z)[\widehat z_1, \ldots, \widehat z_k] \|_s \lesssim_{s, k} \prod_{j = 1}^k \| \widehat z_j \|_0\,.
 $$
 \end{lemma}
 \begin{proof}
 The lemma follows by the properties of the map $\Psi^{nls}$, stated in Theorem \ref{Theorem Birkhoff coordinates}, and the fact that $E = \Pi_S E$, $\|\Pi_S z \|_s \lesssim_s \| z \|_0$ for any vector $z \in h^0_c$, and ${\cal V}_\bot \subset h^0_{\bot c}$ is a ball of radius smaller than $1$.
 \end{proof}
 \noindent
{\bf Outline of the construction of $\Psi_C$:} Following arguements of Moser and Weinstein,
our candidate for $\Psi_C$ is $\Psi^{0, 1}_X$
where $X \equiv X(z, \tau) \in h^0_r$ is a non autonomous vector field 
 with well defined flow $\Psi^{\tau_0, \tau}_X$, $0 \le \tau_0, \tau \le 1$, 
 so that $(\Psi^{0, 1}_X)^* \Lambda_1 = \Lambda_0$. 
 Here $z \in {\cal V}$ and the flow is normalized by $\Psi^{\tau_0, \tau_0}_X (z) = z$.
 To see how to choose $X(z, \tau)$,
 consider the pullback of the two form $\Lambda_\tau$
 by $\Psi^{0, \tau}_X$, $(\Psi^{0, \tau}_X)^* \Lambda_\tau$. Since $(\Psi^{0, 0}_X)^* = Id$, one has
 $(\Psi^{0, 0}_X)^* \Lambda_0 = \Lambda_0$. The desired identity $(\Psi^{0, 1}_X)^* \Lambda_1 = \Lambda_0$
 then follows provided that $(\Psi^{0, \tau}_X)^* \Lambda_\tau$ is independent of $\tau$, i.e., 
 $\partial_\tau \big( (\Psi^{0, \tau}_X)^* \Lambda_\tau \big) = 0$. Since 
 $\partial_\tau \Lambda_\tau = \Lambda_1 - \Lambda_0 = d\lambda_L$, 
it turns out that the latter identity holds if
 $\lambda_L + \Lambda_\tau [ X(\cdot , \tau), \, \cdot \,] = 0$. When expressed in terms of the bilinear form 
 $( \cdot , \cdot )_r$ and taking into account the representation \eqref{rappresentazione omega t} of $\Lambda_\tau$
 and  \eqref{forma finale 1 forma Kuksin} of $\lambda_L$, the latter identity reads
 \begin{equation}\label{fundamental identity}
  \big({ E}(z) \,,\, \widehat z\big)_r + \big( {\cal L}_\tau(z)[X(z, \tau)], \, \widehat z \big)_r = 0\,.
\end{equation}
We choose the vector field $X(z, \tau)$ so that \eqref{fundamental identity} is satisfied.

\bigskip

\noindent
{\bf Vector field $X(z, \tau)$ and its flow:} Motivated by \eqref{fundamental identity},
 the non autonomous vector field $X(z, \tau)$ is defined by
 \begin{equation}\label{definizione campo vettoriale ausiliario}
 { X}( z, \tau) := - {\cal L}_\tau(z)^{- 1} E(z)\,, \quad z \in {\cal V}_S \times {\cal V}_\bot\,, \quad \tau \in [0, 1]\,.
 \end{equation}
 Lemmata \ref{stime cal Lt inverso}, \ref{stima campo vettoriale 1-forma} lead to the following 
 \begin{lemma}\label{stime campo vettoriale ausiliario}
The vector field $X : ({\cal V} \cap h^0_r)  \times [0, 1]\to h^0_r$ is real analytic and one smoothing, meaning that for any $s \in \Z_{\geq 1}$   
 $$
 X : ({\cal V} \cap h^s_r)  \times [0, 1] \to h^{s + 1}_r
 $$
 is real analytic. In addition, the following tame estimates hold: 
 for any $\tau \in [0, 1]$, $z \in {\cal V} \cap h^0_r$, $\widehat z \in h^0_c$,
 \begin{equation}\label{estimate vector field and its derivative0}
 \| X(z, \tau) \|_{0} \lesssim \| z_\bot \|_0^2\,,  \quad 
 \| d X(z, \tau)[\widehat z] \|_{0} \lesssim \| z_\bot\|_0 \| \widehat z\|_0
 \end{equation}
 and for any $k \geq 2$, $\widehat z_1, \ldots , \widehat z_k \in h^0_c$,
 $$
 \| d^k X(z, \tau)[\widehat z_1, \ldots, \widehat z_k]\|_{0} \lesssim_k \prod_{j =1 }^k \| \widehat z_j\|_0\,.
  $$
Moreover, for any $s \in \Z_{\geq 1}$, $z \in {\cal V} \cap h^s_r$, $\widehat z \in h^s_c$,
 \begin{equation}\label{estimate vector field and its derivative}
 \| X(z, \tau) \|_{s + 1} \lesssim_s \| z_\bot \|_s \| z_\bot \|_0\,,  \quad 
 \| d X(z, \tau)[\widehat z] \|_{s + 1} \lesssim_s \| z_\bot\|_0 \| \widehat z\|_s + \| z_\bot\|_s \| \widehat z\|_0\,
 \end{equation}
 and for any $k \geq 2$, $\widehat z_1, \ldots , \widehat z_k \in h^s_c$,
 $$
 \| d^k X(z, \tau)[\widehat z_1, \ldots, \widehat z_k]\|_{s + 1} \lesssim_{s, k} \sum_{j = 1}^k \| \widehat z_j\|_s \prod_{i \neq j} \| \widehat z_i\|_0 + \| z_\bot\|_s \prod_{j = 1}^k \|\widehat z_j \|_0\,.
 $$
  \end{lemma}
 \begin{proof}
 The lemma follows from Lemmata \ref{stime cal Lt inverso}, \ref{stima campo vettoriale 1-forma}. 
 \end{proof}
 We now want to study the flow of the non autonomous differential equation 
 \begin{equation}\label{equazione differenziale mappa cal Ft0}
 \partial_\tau z = X(z, \tau)\,.
 \end{equation}
Recall that for any $ r > 0$, we denote by ${\cal V}_{\bot }(r)$ the ball in $h^0_{\bot c}$ of radius $r$, centered at $0$, and for any $\tau_0, \tau \in [0, 1]$ by $\Psi^{\tau_0, \tau}_{X}$ the flow map of the differential equation \eqref{equazione differenziale mappa cal Ft0}, satisfying $\Psi^{\tau_0, \tau_0}_{X} (z) = z$. 
  By a standard contraction argument,  there exists an open neighborhood ${\cal V}_S' \subseteq {\cal V}_S$ of ${\cal K}$ in $\C^S \times \C^S$ and $\delta > 0$ with ${\cal V}_{\bot}(2 \delta) \subset {\cal V}_\bot$ such that for any $\tau, \tau_0 \in [0, 1]$ 
 \begin{equation}\label{domini flusso cal F}
\Psi^{\tau_0, \tau}_{X}  : {\cal V}_\delta' \cap h^0_r \to {\cal V}_{2 \delta} \cap h^0_r\,, \qquad {\cal V}_\delta' := {\cal V}_S' \times {\cal V}_\bot(\delta)\,, \quad {\cal V}_{2 \delta} := {\cal V}_S \times {\cal V}_\bot(2 \delta)\,
  \end{equation}
  is well defined and real analytic. In the next lemma we state the smoothing estimates for 
  $ \Psi^{\tau_0, \tau}_{X} - { \io d}$ where ${ \io d}$ denotes the identity map on ${\cal V}_\delta' \cap h^0_r$.
 \begin{lemma}\label{lemma stima flusso correttore}
By choosing $ 0 < \delta < 1 $ smaller, if necessary, it follows that for any $\tau, \tau_0 \in [0, 1]$,
 the map $\Psi^{\tau_0, \tau}_{X} - {\io d} : {\cal V}_\delta' \cap h^0_r \to h^{0}_r$ is one smoothing, 
 meaning that for any $s \in \Z_{\geq 1}$, the map  
 $$
 \Psi^{\tau_0, \tau}_{X} - {\io d} : {\cal V}_\delta' \cap h^s_r \to h^{s + 1}_r
 $$
 is real analytic. Furthermore, the following tame estimates hold: 
 for any $ z \in {\cal V}_\delta' \cap h^0_r$, $\widehat z \in h^0_c$, 
 \begin{equation}\label{stima flusso e sua derivata0}
 \|  \Psi^{\tau_0, \tau}_{X}(z) - z \|_{0} \lesssim \| z_\bot \|_0^2 \,, \quad 
 \| (d  \Psi^{\tau_0, \tau}_{X}(z) - {\rm Id}) [\widehat z] \|_{0} \lesssim \| z_\bot\|_0 \| \widehat z\|_0
 \end{equation}
 and for any $k \geq 2$, $ \widehat z_1, \ldots , \widehat z_k \in h^0_c$,
 $$
 \| d^k   \Psi^{\tau_0, \tau}_{X}(z)[\widehat z_1, \ldots, \widehat z_k]\|_0 \lesssim_k \,  \prod_{j = 1}^k \|\widehat z_j \|_0\,
 $$
whereas for any $s \in \Z_{\geq 1}$, $ z \in {\cal V}_\delta' \cap h^s_r$, $\widehat z \in h^s_c$,
 \begin{equation}\label{stima flusso e sua derivata}
 \|  \Psi^{\tau_0, \tau}_{X}(z) - z \|_{s + 1} \lesssim_s \| z_\bot \|_s \| z_\bot \|_0\,, \quad 
 \| (d  \Psi^{\tau_0, \tau}_{X}(z) - {\rm Id}) [\widehat z] \|_{s + 1} \lesssim_s \| z_\bot\|_0 \| \widehat z\|_s + \| z_\bot\|_s \| \widehat z\|_0\,
 \end{equation}
 and for any $k \geq 2$, $ \widehat z_1, \ldots , \widehat z_k \in h^s_c$,
 $$
 \| d^k   \Psi^{\tau_0, \tau}_{X}(z)[\widehat z_1, \ldots, \widehat z_k]\|_{s + 1} \lesssim_{s, k} \, 
 \sum_{j = 1}^k \| \widehat z_j\|_s \prod_{i \neq j} \| \widehat z_i\|_0 + \| z_\bot\|_s \prod_{j = 1}^k \|\widehat z_j \|_0\,.
 $$
  \end{lemma}
  \begin{proof}
    For any $\tau_0, \tau \in [0, 1]$ and $z \in {\cal V}_\delta'  \cap h^0_r$, the flow $ \Psi^{\tau_0, \tau}_{X}(z)$ satisfies the integral equation 
    \begin{equation}\label{equazione integrale flusso}
     \Psi^{\tau_0, \tau}_{X}(z) = z + \int_{\tau_0}^\tau X(\Psi^{\tau_0, t}_{X}(z), t)\,d t\,.
    \end{equation}
    In view of the estimate \eqref{estimate vector field and its derivative} of the vector field $X(z, \tau)$, we first estimate
    $\| \Pi_\bot  \Psi^{\tau_0, \tau}_{X}(z)\|_{s}$  for $z \in {\cal V}_\delta' \cap h^s_r$ with $s \in \Z_{\geq 0}$.
Applying the operator $\Pi_\bot$ to both sides of the identity \eqref{equazione integrale flusso}, one gets 
    $$
    \Pi_\bot  \Psi^{\tau_0, \tau}_{X}(z) = \Pi_\bot z + \int_{\tau_0}^\tau \Pi_\bot X(\Psi_{X}^{\tau_0, t}(z), t)\,dt\,.
    $$
By Lemma \ref{stime campo vettoriale ausiliario}, for any $\tau, \tau_0 \in [0, 1]$, one has
    \begin{align}
 \| \Pi_\bot  \Psi^{\tau_0, \tau}_{X}(z)\|_{s} \leq \| z_\bot\|_s + 
 C(s) \Big| \int_{\tau_0}^\tau \|  \Pi_\bot\Psi^{\tau_0, t}_{X}(z) \|_s \| \Pi_\bot\Psi^{\tau_0, t}_{X}(z)\|_0\, d t\, \Big|
    \end{align}
   for some constant $C(s) > 0$, only depending on $s$. Then by shrinking $\delta > 0$, if necessary, so that for $z_\bot \in {\cal V}_\bot(\delta)$, we have $\sup_{\tau_0, \tau \in [0, 1]} \| \Pi_\bot  \Psi^{\tau_0, \tau}_{X}(z)\|_0\leq 1$, the above estimate becomes  
   \begin{align}
 \|    \Pi_\bot  \Psi^{\tau_0, \tau}_{X}(z)\|_{s} \leq \| z_\bot\|_s + C(s) \Big| \int_{\tau_0}^\tau \|  \Pi_\bot \Psi^{\tau_0, t}_{X}(z) \|_s d t \Big|\,.
    \end{align}
   By the Gronwall inequality one then gets 
    \begin{equation}\label{stima Pi bot flusso nella dim}
    \sup_{\tau_0, \tau \in [0, 1]} \|    \Pi_\bot  \Psi^{\tau_0, \tau}_{X}(z)\|_{s} \lesssim_s  \| z_\bot\|_s\,, \qquad \forall z \in {\cal V}_\delta' \cap h^s_r\,. 
    \end{equation}
    Now let us prove \eqref{stima flusso e sua derivata}. 
   By \eqref{equazione integrale flusso}, using again Lemma \ref{stime campo vettoriale ausiliario}, one gets for any $s \in \Z_{\geq 1}$, $\tau_0, \tau \in [0, 1]$, and $z \in {\cal V}_\delta' \cap h^s_r$
    \begin{align}
    \|  \Psi^{\tau_0, \tau}_{X}(z) - z \|_{s + 1} & \leq \Big| \int_{\tau_0}^\tau  \| X(\Psi^{\tau_0, t}_{X}(z), t) \|_{s + 1}\,d t \Big|  \lesssim_s \sup_{t \in [0, 1]} \| \Pi_\bot \Psi_{X}^{\tau_0,t}(z)\|_s \sup_{t \in [0, 1]} \| \Pi_\bot\Psi^{\tau_0, t}_{X}(z) \|_0 \nonumber\\
    & \stackrel{\eqref{stima Pi bot flusso nella dim}}{\lesssim_s} \| z_\bot\|_s \| z_\bot\|_0\,,
    \end{align}
    which is the first claimed inequality in \eqref{stima flusso e sua derivata}. To prove the one for the differential  $d  \Psi^{\tau_0, \tau}_{X} - {\rm Id}$, differentiate \eqref{equazione integrale flusso} with respect to $z$. Using the chain rule one gets
    \begin{equation}\label{derivata equazione integrale flusso}
    d  \Psi^{\tau_0, \tau}_{X}(z)[\widehat z] = \widehat z +  \int_{\tau_0}^\tau d X(\Psi^{\tau_0, t}_{X}(z),t)[d\Psi^{\tau_0, t}_{X}(z)[\widehat z]]\,d t\,.
    \end{equation}
    By applying the estimates of $d X(\cdot, \tau)$ of Lemma \ref{stime campo vettoriale ausiliario}, 
    it follows that for any $s \in \Z_{\geq 0}$ there is a constant $C(s) > 0$ such that  
    \begin{align}
    \|   d  \Psi^{\tau_0, \tau}_{X}(z) [\widehat z] \|_{s} & \leq \| \widehat z\|_s + C(s) \Big| \int_{\tau_0}^\tau \Big( \| \Pi_\bot  \Psi^{\tau_0, t}_{X}(z)\|_s \| d  \Psi^{\tau_0, t}_{X}(z)[\widehat z]\|_0 + \| \Pi_\bot  \Psi^{\tau_0,t}_{X}(z)\|_0   \| d  \Psi^{\tau_0, t}_{X}(z)[\widehat z]\|_s \Big)\, d t \Big|  \nonumber\\
    & \stackrel{\eqref{stima Pi bot flusso nella dim}}{\leq} \| \widehat z\|_s +  C_1(s) \Big| \int_{\tau_0}^\tau\Big(\| z_\bot\|_s  \| d  \Psi^{\tau_0, t}_{X}(z)[\widehat z]\|_0 + \| z_\bot\|_0   \| d  \Psi^{\tau_0, t}_{X}(z)[\widehat z]\|_s  \Big)\, d t \Big|\, \label{san martino campanaro 0}
    \end{align}
    for some constant $C_1(s) > C(s) > 0$. For $s= 0$, using that $\|z_\bot \|_0 \leq \delta < 1$, \eqref{san martino campanaro 0} becomes 
    $$
    \|   d  \Psi^{\tau_0, \tau}_{X}(z) [\widehat z] \|_{0} \leq \| \widehat z\|_0 + 2 C_1(0) \Big| \int_{\tau_0}^\tau   \| d  \Psi^{\tau_0, t}_{X}(z)[\widehat z]\|_0\, d t \Big|\,
    $$
   and hence by the Gronwall inequality 
    $$
   \|   d  \Psi^{\tau_0, \tau}_{X}(z) [\widehat z] \|_{0} \lesssim \| \widehat z\|_0\,.
    $$
    For $s \in \Z_{\geq 1}$, substitute the latter estimate into \eqref{san martino campanaro 0} to get, again using that $\|z_\bot\|_0 < \delta < 1$
    \begin{align}
 \|   d  \Psi^{\tau_0, \tau}_{X}(z) [\widehat z] \|_{s} & \leq \| \widehat z\|_s + C_2(s) \| z_\bot\|_s \| \widehat z\|_0 + C_2(s) \Big|\int_{\tau_0}^\tau  \| d  \Psi^{\tau_0, t}_{X}(z)[\widehat z]\|_s \, dt \Big| \label{san martino campanaro 1}
    \end{align}
    for some constant $C_2(s) > C_1(s)$. Then using again the Gronwall inequality one concludes that for any
    $0 \le \tau_0 \le 1,$
    \begin{equation}\label{san martino campanaro 2}
    \sup_{\tau \in [0, 1]} \|   d  \Psi^{\tau_0, \tau}_{X}(z) [\widehat z] \|_{s} \lesssim_s  \| \widehat z\|_s +  \| z_\bot\|_s \|\widehat z \|_0\,.
    \end{equation}
    We are now ready to prove the second estimate in \eqref{stima flusso e sua derivata}. By \eqref{derivata equazione integrale flusso} and the smoothing estimates on $d X(\cdot, \tau)$ of Lemma \ref{stime campo vettoriale ausiliario}, one gets that for any $s \in \Z_{\geq 1}$, $0 \le \tau_0 \le 1,$
    \begin{align}
    \|\big( d  \Psi^{\tau_0, \tau}_{X}(z) - {\rm Id} \big)[\widehat z] \|_{s + 1} & \lesssim_s \sup_{t \in [0, 1]} \| \Pi_\bot  \Psi^{\tau_0, t}_{X}(z)\|_s\sup_{t  \in [0, 1]} \| d  \Psi^{\tau_0, t }_{X}(z)[\widehat z]\|_0 \nonumber\\
    & \qquad + \sup_{ t \in [0, 1]} \| \Pi_\bot  \Psi^{\tau_0, t}_{X}(z)\|_0 \sup_{t \in [0, 1]}  \| d  \Psi^{\tau_0, t}_{X}(z)[\widehat z]\|_s \nonumber\\ 
    & \stackrel{\eqref{stima Pi bot flusso nella dim},\eqref{san martino campanaro 2}}{\lesssim_s} \| z_\bot\|_s \| \widehat z\|_0 + \| z_\bot\|_0 \| \widehat z\|_s\,, \nonumber
    \end{align}
    where we used again that $\| z_\bot\|_0 < \delta < 1$. Hence the claimed estimate for $d  \Psi^{\tau_0, \tau}_{X}(z) - {\rm Id}$ in \eqref{stima flusso e sua derivata} is established. The estimates for the higher order derivatives $d^k  \Psi^{\tau_0, \tau}_{X}$, $k \geq 2$, follow by similar arguments, 
    differentiating $k$-times the equation \eqref{equazione integrale flusso}  with respect to $z$. 
      \end{proof}
 \noindent 
 {\bf Definition of $\Psi_C$ and its properties:}    
      Our candidate for the symplectic corrector is the time one flow map of $X(z, \tau)$, 
\begin{equation}\label{definizione correttore simplettico}
\Psi_C : = \Psi_{X}^{0, 1} : {\cal V}_\delta' \cap h^0_r \to h^0_r\,.
\end{equation} 
Clearly, $\Psi_C$ is one to one and its inverse is given by the backward flow of the PDE \eqref{equazione differenziale mappa cal Ft0}, namely $\Psi_C^{- 1} = \Psi_{X}^{1, 0}$. Hence the maps $\Psi_C^{\pm 1}$ satisfy the estimates stated in Lemma \ref{lemma stima flusso correttore}. Furthermore, recall that for any $\tau \in [0, 1]$, the two form $\Lambda_\tau$ admits the representation \eqref{rappresentazione omega t}. Then the following Darboux lemma holds.
  \begin{proposition}\label{lemma principale correttore mappa simplettica}
The map $\Psi_C$ is a symplectic corrector, i.e., for any $z \in {\cal V}_\delta' \cap h^0_r$, $ \Psi_C^* \Lambda_1(z) = \Lambda_0$.
  \end{proposition}
  \begin{proof}
  For any $\tau \in [0, 1]$, consider the two form $( \Psi^{0, \tau}_{X})^* \Lambda_\tau$. Since $\Psi_{X}^{0, 0} = {\rm Id}$, one has $(\Psi_{X}^{0, 0})^* \Lambda_0 = \Lambda_0$ and hence it suffices to prove that the map $\tau \mapsto ( \Psi^{0, \tau}_{X})^* \Lambda_\tau$ is constant or, equivalently,
  $$
  \partial_\tau \big( ( \Psi^{0, \tau}_{X})^* \Lambda_\tau \big) = 0\,, \quad \forall \tau \in [0, 1]\,.
  $$
  By Cartan's identity (see for instance Lemma 1.2 in \cite{	K}) and the fact that $\Lambda_\tau$ is closed, it follows that 
  $$
    \partial_\tau \big( ( \Psi^{0, \tau}_{X})^* \Lambda_\tau \big) =
    ( \Psi^{0, \tau}_{X})^* \big( \partial_\tau \Lambda_\tau + 
    d ( \Lambda_\tau [X( \cdot, \tau), \, \cdot \, ] ) \big)\,.
  $$
Since $\partial_\tau \Lambda_\tau \stackrel{\eqref{definizione Lambda tau}}{=} \Lambda_1 - \Lambda_0 = \Lambda_L$  and $\Lambda_L \stackrel{\eqref{1 forma Kuksin}}{=}  d \lambda_L$, it remains to prove that
    $$
  d \big( \lambda_L + \Lambda_\tau\big[X(\cdot, \tau), \, \cdot \, \big] \big) = 0\,.
  $$
  By \eqref{rappresentazione omega t}, \eqref{forma finale 1 forma Kuksin}, \eqref{definizione campo vettoriale ausiliario}, one has for any $\tau \in [0, 1]$, $z \in {\cal V}_\delta' \cap h^0_r$, and $\widehat z \in h^0_c$ 
  $$
  \lambda_L(z)[\widehat z] + \Lambda_\tau [X(z, \tau), \widehat z ] = \big( E(z), \widehat z \big)_r - \big( {\cal L}_\tau(z) {\cal L}_\tau(z)^{- 1} E(z), \widehat z\big)_r = 0\,.
  $$
 It means that 
  $$
  \lambda_L + \Lambda_\tau [ X(\cdot, \tau), \, \cdot \, ] = 0\,, \quad \forall \tau \in [0, 1]\,,
  $$
  proving the proposition. 
  \end{proof}
  As a consequence of Lemma \ref{lemma stima flusso correttore} we get the following 
  \begin{corollary}\label{corollario correttore simplettico}
$(i)$ For any $s \in \Z_{\geq 0}$, the map $\Psi_C : {\cal V}_\delta' \cap h^s_r \to h^s_r$ is a real analytic diffeomorphism onto its image and its nonlinear part is one smoothing, meaning that for any $s \in \Z_{\geq 1}$, the map $B_C := \Psi_C - { \io d} :  {\cal V}_\delta'  \cap h^s_r \to h^{s + 1}_r$ is real analytic. Furthermore,
$B_C$ satisfies the following tame estimates:  
for any $ z \in {\cal V}_\delta' \cap h^0_r$, $\widehat z \in h^0_c$,
 $$
 \|  B_C(z) \|_{0} \lesssim  \| z_\bot \|_0^2\,, \quad 
 \| d B_C(z) [\widehat z] \|_{0} \lesssim \| z_\bot\|_0 \| \widehat z\|_0 
 $$
 and for any $k \geq 2$, $\widehat z_1, \ldots , \widehat z_k \in h^0_c$,
 $$
 \| d^k  B_C(z)[\widehat z_1, \ldots, \widehat z_k]\|_{0} \lesssim_k  \prod_{j = 1}^k \|\widehat z_j \|_0\,,
 $$
whereas for any $s \in \Z_{\geq 1}$, $ z \in {\cal V}_\delta' \cap h^s_r$, $\widehat z \in h^s_c$, 
 $$
 \|  B_C(z) \|_{s + 1} \lesssim_s \| z_\bot \|_s \| z_\bot \|_0\,, \quad 
 \| d B_C(z) [\widehat z] \|_{s + 1} \lesssim_s \| z_\bot\|_0 \| \widehat z\|_s + \| z_\bot\|_s \| \widehat z\|_0\,
 $$
 and for any $k \geq 2$, $\widehat z, \widehat z_1, \ldots , \widehat z_k \in h^s_c$,
 $$
 \| d^k  B_C(z)[\widehat z_1, \ldots, \widehat z_k]\|_{s + 1} \lesssim_{s, k} \sum_{j = 1}^k \| \widehat z_j\|_s \prod_{i \neq j} \| \widehat z_i\|_0 + \| z_\bot\|_s \prod_{j = 1}^k \|\widehat z_j \|_0\,.
 $$

 \noindent
 $(ii)$ The map $A_C :=\Psi_C^{- 1} - {  \io d} : \Psi_C({\cal V}_\delta') \cap h^0_r \to h^{0}_r$ 
 is real analytic and satisfies the following tame estimates: 
 for any $ z \in \Psi_C({\cal V}_\delta') \cap h^0_r$, $\widehat z \in h^0_c$,
 $$
 \|  A_C(z) \|_{0} \lesssim \| z_\bot \|_0^2 \,, \quad 
 \| d A_C(z) [\widehat z] \|_{0} \lesssim \| z_\bot\|_0 \| \widehat z\|_0 
 $$
 and for any $k \geq 2$, $ \widehat z_1, \ldots , \widehat z_k \in h^0_c$,
 $$
 \| d^k  A_C(z)[\widehat z_1, \ldots, \widehat z_k]\|_0 \lesssim_k \prod_{j = 1}^k \|\widehat z_j \|_0\,. 
 $$
Furthermore, for any $s \in \Z_{\geq 1}$, 
$A_C : \Psi_C({\cal V}_\delta') \cap h^s_r \to h^{s + 1}_r$ is real analytic and satisfies the following tame estimates:  for any $ z \in \Psi_C({\cal V}_\delta') \cap h^s_r$, $\widehat z \in h^s_c$, 
 $$
 \|  A_C(z) \|_{s + 1} \lesssim_s \| z_\bot \|_s \| z_\bot \|_0 \,, \quad 
 \| d A_C(z) [\widehat z] \|_{s + 1} \lesssim_s \| z_\bot\|_0 \| \widehat z\|_s + \| z_\bot\|_s \| \widehat z\|_0\,
 $$
 and for any $k \geq 2$, $ \widehat z_1, \ldots , \widehat z_k \in h^s_c$,
 $$
 \| d^k  A_C(z)[\widehat z_1, \ldots, \widehat z_k]\|_{s + 1} \lesssim_{s, k} \sum_{j = 1}^k \| \widehat z_j\|_s \prod_{i \neq j} \| \widehat z_i\|_0 + \| z_\bot\|_s \prod_{j = 1}^k \|\widehat z_j \|_0\,. 
 $$
  \end{corollary}
  \begin{proof}
  The claimed results are a special case of Lemma \ref{lemma stima flusso correttore}, since 
  $\Psi_C = \Psi^{0, 1}_{X}$ and $\Psi_C^{- 1} = \Psi_{X}^{1, 0}$. 
  \end{proof}
  An immediate consequence of Corollary \ref{corollario correttore simplettico} is the following result, needed in Subsection \ref{Hamiltoniana trasformata}.
  \begin{corollary}\label{espansione quadratica cubica cal R Psi}
The Taylor expansion of the map $B_C = \Psi_C - {\io d}$ around $\Pi_S z$ up to order three is of the form 
$$
B_C(z) = B^C_{2}(z) + B^C_{3}(z)\,, \quad z \in {\cal V}_\delta' \cap h^0_r\,,
$$
where 
\begin{equation}\label{definizione cal R Psi (2)}
B^C_{2}(z) := \frac12 d^2 B_C(\Pi_S z)[\Pi_\bot z, \Pi_\bot z] 
\end{equation}
and $B^C_{3}(z)$ is the Taylor remainder term
\begin{equation}\label{definizione cal R Psi (3)} 
B^C_{3}(z) := \frac12 \int_0^1 (1 - t)^2 d^3 B_C(\Pi_S z + t \Pi_\bot z)[\Pi_\bot z, \Pi_\bot z, \Pi_\bot z]\,d t\,.
\end{equation}
The maps $B^C_i : {\cal V}_\delta' \cap h^0_r \to h^{0}_r$, $i = 2, 3$,  are real analytic and $B^C_{3}$ satisfies the following estimates: for any $ z \in {\cal V}_\delta' \cap h^0_r$, $\widehat z , \widehat z_1, \widehat z_2 \in h^0_c$, 
$$
\| B^C_3(z)\|_{0} \lesssim  \| z_\bot\|_0^3 \,, \quad \| d B^C_3(z)[\widehat z] \|_{0} \lesssim \| z_\bot\|_0^2 \| \widehat z\|_0 \,, \quad
\| d^2 B^C_3(z)[\widehat z_1, \widehat z_2]\|_{0} \lesssim \| z_\bot\|_0 \| \widehat z_1\|_0 \| \widehat z_2\|_0 
$$
and for any $k \geq 3$, $ \widehat z_1, \ldots , \widehat z_k \in h^0_c$,
$$
\| d^k B^C_3(z)[\widehat z_1, \ldots, \widehat z_k]\|_0 \lesssim_k  \prod_{j = 1}^k \|\widehat z_j \|_0\,.
 $$
Furthermore, for any $s \in \Z_{\geq 1}$, $B^C_{i} : {\cal V}_\delta' \cap h^s_r \to h^{s + 1}_r$,
$i = 2, 3$, are real analytic and $B^C_3$ satisfies the following tame estimates: for any 
$ z \in {\cal V}_\delta' \cap h^s_r$, $\widehat z , \widehat z_1, \widehat z_2 \in h^s_c$,
$$
\| B^C_3(z)\|_{s + 1} \lesssim_s \| z_\bot\|_s \| z_\bot\|_0^2 \,, \quad \| d B^C_3(z)[\widehat z] \|_{s + 1} \lesssim_s \| z_\bot\|_0^2 \| \widehat z\|_s + \| z_\bot\|_s \| z_\bot\|_0 \| \widehat z\|_0\,,
$$
$$
\| d^2 B^C_3(z)[\widehat z_1, \widehat z_2]\|_{s + 1} \lesssim_s \| z_\bot\|_0 \big(\| \widehat z_1\|_0 \| \widehat z_2\|_s + \| \widehat z_1\|_s \| \widehat z_2\|_0  \big) + \| z_\bot\|_s \| \widehat z_1\|_0 \| \widehat z_2\|_0
$$
and for any $k \geq 3$, $\widehat z, \widehat z_1, \ldots , \widehat z_k \in h^s_c$,
$$
\| d^k B^C_3(z)[\widehat z_1, \ldots, \widehat z_k]\|_{s + 1} \lesssim_{s, k} \sum_{j = 1}^k \| \widehat z_j\|_s \prod_{i \neq j} \| \widehat z_i\|_0 + \| z_\bot\|_s \prod_{j = 1}^k \|\widehat z_j \|_0\,.
 $$
 \end{corollary}
 \begin{proof}
Note that by Corollary \ref{corollario correttore simplettico}, $B_C(\Pi_S z) = 0$ and $d B_C(\Pi_S z) = 0$. Thus $B_C(z) = B^C_2(z) + B^C_3(z)$ is the Taylor expansion of $B_C$ around $\Pi_S z$ with Taylor remainder term given by \eqref{definizione cal R Psi (3)}. The claimed analyticity and tame estimates follow from Corollary \ref{corollario correttore simplettico}.
 \end{proof}
  

\section{Proof of Theorem \ref{modified Birkhoff map}.}\label{Proof of main theorem}
 In this section we prove Theorem \ref{modified Birkhoff map}. First we introduce and discuss our new canonical coordinates and then express the Hamiltonian of the defocusing NLS equation in the new coordinates.  
 
 
 \noindent
\subsection{New canonical coordinates}\label{sec new coordinates}
 
 \noindent
 Our candidate of the canonical transformation  is the map
 \begin{equation}\label{definizione mappa simplettica finale}
 \Psi := \Psi_L \circ \Psi_C\, : {\cal V}_\delta' \to H^0_c
 \end{equation} 
 where ${\cal V}_\delta'$ is the neighborhood introduced in \eqref{domini flusso cal F}. 
  \begin{proposition}\label{Lemma finale correttore simplettico}
By shrinking $0 < \delta < 1$, if necessary, it follows that for any $s \in \Z_{\geq 0}$, $\Psi : {\cal V}_\delta' \cap h^s_r \to H^s_r$ is a real analytic symplectic diffeomorphism onto its image with the property that
its nonlinear part $B  := \Psi - F_{nls}^{- 1}: {\cal V}_\delta'  \cap h^0_r \to H^{0}_r$ satisfies the following estimates: 
for any $k \in \Z_{\geq 1}$, $ z \in {\cal V}_\delta' \cap h^0_r$,  $ \widehat z_1, \ldots , \widehat z_k \in h^0_c$,
    $$ \| B(z) \|_{0} \lesssim \, 1 \, , \quad
    \| d^k B(z)[\widehat z_1, \ldots, \widehat z_k]\|_{0} \lesssim_k  \prod_{j = 1}^k \|\widehat z_j \|_0\,.
    $$
Furthermore, $B$ is one smoothing, meaning that for any $s \in \Z_{\geq 1}$, the map 
$B: {\cal V}_\delta'  \cap h^s_r \to H^{s + 1}_r$ is real analytic, and it satisfies the following tame estimates: 
for any $k \in \Z_{\geq 1}$, $ z \in {\cal V}_\delta' \cap h^s_r$, and $ \widehat z_1, \ldots , \widehat z_k \in h^s_c$,
    $$ \| B(z) \|_{s+1} \lesssim_{s} \, 1 +  \|z_\bot\|_s\, , \quad
    \| d^k B(z)[\widehat z_1, \ldots, \widehat z_k]\|_{s + 1} \lesssim_{s ,k} \, 
    \sum_{j = 1}^k \| \widehat z_j\|_s \prod_{i \neq j} \| \widehat z_i\|_0 + \| z_\bot\|_s \prod_{j = 1}^k \|\widehat z_j \|_0\,.
    $$
    \end{proposition}
   \begin{proof}
  By Proposition \ref{Psi L 2} and Corollary \ref{corollario correttore simplettico} one has that for any $s \in \Z_{\geq 0}$, the map $\Psi = \Psi_L \circ \Psi_C : {\cal V}_\delta' \cap h^s_r \to H^s_r$ is real analytic
  and 
  \begin{equation}\label{simpletticita mappa finale}
  \Psi^* \Lambda = (\Psi_L \circ \Psi_C)^* \Lambda = \Psi_C^* \Psi_L^* \Lambda \stackrel{\eqref{definizione omega 0 1}}{=} \Psi_C^* \Lambda_1 \stackrel{Proposition\,\, \ref{lemma principale correttore mappa simplettica}}{=} \Lambda_0 \stackrel{\eqref{definizione omega 0 1}}{=} \Lambda_M\,,
  \end{equation}
implying that $\Psi$ is symplectic. Recalling that $\Psi_L = F_{nls}^{- 1} + B_L$ (see \eqref{espressione BL(z)}) and using that, by Corollary \ref{corollario correttore simplettico}, $\Psi_C = { \io d} + B_C$, a direct calculation shows that for any $z \in {\cal V}_\delta' \cap h^0_r$
\begin{align}
B(z) = \Psi(z) -  F_{nls}^{- 1} (z) & = F_{nls}^{- 1} (B_C(z)) + B_L(\Psi_C(z)) \label{definition cal R Psi}
\end{align}
The claimed estimates for $B$ then follow from the estimates of Proposition \ref{Psi L 2} and the ones of Corollary \ref{corollario correttore simplettico}.  
\end{proof}
Substituting formula \eqref{espressione BL(z)} for $B_L$ one gets
\begin{align}
 \Psi(z) =  & F_{nls}^{- 1} (z) + F_{nls}^{- 1}( B_C(z)) + 
 B^{nls}\big( \Pi_S z + \Pi_SB_C(z) \big) + d_\bot B^{nls}\big( \Pi_S z + \Pi_S B_C(z) \big)[z_\bot + \pi_\bot B_C(z) ]
\end{align}
where according to Corollary \ref{espansione quadratica cubica cal R Psi},
$$
B_C(z) =  \frac12 d^2 B_C(\Pi_S z)[\Pi_\bot z, \Pi_\bot z] 
+ \frac12 \int_0^1 (1 - t)^2 d^3 B_C(\Pi_S z + t \Pi_\bot z)[\Pi_\bot z, \Pi_\bot z, \Pi_\bot z]\,d t\,.
$$
Next, we state and prove the one smoothing property and tame estimates for the map
 \begin{equation}\label{definizione cal A Psi (w)}
 {\cal A}(z) := d\Psi(z)^{- 1} - F_{nls}\,, \qquad z \in {\cal V}_\delta' \cap h^0_r\,.
 \end{equation}
By the chain rule, 
 \begin{equation}\label{formula d Psi inverse}
 d \Psi(z)^{- 1} = d \Psi_C(z)^{- 1} \, \big( d \Psi_L(\Psi_C(z)) \big)^{- 1}\,.
 \end{equation}
Note that by Corollary \ref{corollario correttore simplettico}, 
$$
d \Psi_C(z)^{- 1} = d \Psi_C^{- 1}(\Psi_C(z))= {\rm Id} + d A_C(\Psi_C(z))\,,
$$
and that by \eqref{cuba 0}, $d \Psi_L(z)^{- 1}  = F_{nls} + {\cal A}_L(z)$. Hence \eqref{formula d Psi inverse} can be written as 
 \begin{equation}\label{definizione cal A Psi}
 d \Psi(z)^{- 1}  = F_{nls} + {\cal A}(z)\,, \quad {\cal A}(z) := {\cal A}_L(\Psi_C(z)) + d A_C(\Psi_C(z)) d \Psi_L(\Psi_C(z))^{- 1}\,.
 \end{equation}
 \begin{proposition}[Tame estimates for ${\cal A}$]\label{stima tame cal A Psi}
 For any $s \in \Z_{\geq 1}$, the map ${\cal A} : {\cal V}_\delta' \cap h^s_r \to {\cal L}(H^s_c , h^{s + 1}_c)$ is real analytic and satisfies the following tame estimates: for any $z \in {\cal V}_\delta' \cap h^0_r$,  
 $\widehat w \in H^0_c$, 
$$
 \| {\cal A}(z)[\widehat w] \|_{0} \lesssim   \| \widehat w \|_0 
$$
and for any $k \ge 1$, $\widehat z_1, \ldots, \widehat z_k \in h^0_c$,  
\begin{align*}
\| d^k \big( {\cal A}(z)[\widehat w] \big)[\widehat z_1, \ldots, \widehat z_k] \|_{0}
 \lesssim_{ k}  & \,\,
\| \widehat w\|_0 \prod_{j = 1}^k \| \widehat z_j\|_0\,.
 \end{align*}  
Moreover, for any $s \in \Z_{\geq 1}$, $z \in {\cal V}_\delta' \cap h^s_r$,  $w \in H^s_c$, 
$$
 \| {\cal A}(z)[\widehat w] \|_{s+1} \lesssim_{s}  \| z_\bot\|_s \| \widehat w \|_0  +  \| \widehat w\|_s
$$
and for any $k \ge 1$, $\widehat z_1, \ldots, \widehat z_k \in h^s_c$,   
\begin{align*}
\| d^k \big( {\cal A}(z)[\widehat w] \big)[\widehat z_1, \ldots, \widehat z_k] \|_{s + 1}
 \lesssim_{s, k}  &
\,\, \Big(   \| z_\bot\|_s \|\widehat w \|_0  +  \| \widehat w\|_s \Big) \, \prod_{j = 1}^k \| \widehat z_j\|_0  +   \| \widehat w\|_0 \, \sum_{j = 1}^k  \| \widehat z_j\|_s \prod_{i \neq j} \| \widehat z_i\|_0  \,.
 \end{align*}  
\end{proposition}
\begin{proof}
The claimed estimates for ${\cal A}$ follow by Lemma \ref{estimates of d Psi L inv} and Corollary \ref{corollario correttore simplettico} by the chain and product rules. 
\end{proof}


  \subsection{The defocusing NLS Hamiltonian in new coordinates}\label{Hamiltoniana trasformata}
In this subsection we prove the expansion of ${\cal H}^{nls} \circ \Psi$, stated in $(C3)$ of Theorem \ref{modified Birkhoff map}. 
Recall from \eqref{Poisson brackets} that the Hamiltonian of the defocusing NLS equation is given by 
  $$
  {\mathcal H}^{nls}(w) = \int_{0}^1 (\partial _x  u \partial _x v + u^2 v^2)dx\,, \quad w = (u, v) \in H^1_r\,. 
  $$
    By Theorem \ref{Theorem Birkhoff coordinates}, $H^{nls} := {\cal H}^{nls} \circ \Psi^{nls}$
  only depends on the actions. By a slight abuse of notation we write 
  \begin{equation}\label{Hamiltoniana integrabile Birkhoff}
  H^{nls} = H^{nls}(I)\,, \qquad I  = (I_k)_{k \in \Z} \in \ell^{1, 2}_+\,, 
  \quad I_k \equiv I_k(z) = |z_k|^2/2 = (x_k^2 + y_k^2)/2 \quad \forall k \in \Z\,
  \end{equation}
 and denote by $ \omega_k^{nls} (I)$ the dNLS frequencies, 
 \begin{equation}\label{definizione frequenze NLS introduzione}
 \omega_k^{nls} (I):= \partial_{I_k} H^{nls}(I)\,, \qquad k \in \Z\,.
 \end{equation}
The  properties of the frequency map $ I \mapsto \omega (I) := (\omega_k(I))_{k \in \Z}$ , needed in the sequel, are summarized in the following
   \begin{proposition} {\bf (dNLS frequencies)}
\label{Asintotica frequenze integrabile} 
The map
   \begin{equation}\label{asintotics frequencies A}
   \ell^{1, 2}_+ \rightarrow \ell ^\infty , \ (I_k)
      _{k \in {\mathbb Z}} \mapsto (\omega ^{ nls}_n(I) - 4 \pi ^2 n^2)
      _{n \in {\mathbb Z}}
   \end{equation}
is real analytic and bounded. 


\end{proposition}
\begin{proof}
See e.g. Theorem 3.2 in \cite{BKM}. 
\end{proof}

With the notation introduced above, the $L^2$-gradient $\nabla H^{nls}(z)$ is then given  by  
  $$
  \nabla H^{nls}(z) = \Omega^{nls}(I)[z]\,, \quad z \in h^1_r\,,  \,\,\, I \equiv I(z) = (I_n(z))_{n \in \Z}
  $$
where for any $I \in \ell^{1,2}_+$, \, ${ \Omega}^{nls}(I ): h^1_r \to h^{-1}_r$ is the diagonal operator
  \begin{equation}\label{definizione bf Omega nls}
  { \Omega}^{nls}(I ) := \begin{pmatrix}
 {\rm diag}_{k \in \Z} \omega_k^{nls}(I) & 0 \\
 0 & {\rm diag}_{k \in \Z} \omega_k^{nls}(I) 
  \end{pmatrix}\,.
  \end{equation}
 Further note that since $H^{nls}(z) = {\cal H}^{nls}(\Psi^{nls}(z))$ one has by the chain rule 
  \begin{equation}\label{identita importante 1}
  { \Omega}^{nls}(I) [z] =   \nabla H^{nls}(z) =  (d \Psi^{nls}(z))^t \nabla {\cal H}^{nls}(\Psi^{nls}(z))\,, \quad \forall z \in {\cal V} \cap h^1_r\,
  \end{equation}
  where ${\cal V}$ is the neighborhood of $h^0_r$ in $h^0_c$ of Theorem \ref{Theorem Birkhoff coordinates}, ${\cal V} = \Phi^{nls}({\cal W})$. 
 For later use we record that \eqref{identita importante 1}, evaluated at $z$ with $z = \Pi_S z$, reads
  $$
  { \Omega}^{nls}(I_S, 0)[ \Pi_S z] = (d \Psi^{nls}(\Pi_S z))^t \nabla {\cal H}^{nls}(\Psi^{nls}(\Pi_S z))
  $$ 
  implying that 
  \begin{equation}\label{identita importante 2}
  \Pi_\bot (d \Psi^{nls}(\Pi_S z))^t \nabla {\cal H}^{nls}(\Psi^{nls}(\Pi_S z)) = 0\,.
  \end{equation}
  \noindent 
The equations of motion, associated to the Hamiltonian $H^{nls}$ are given by  
  \begin{equation}\label{equazione integrabile NLS}
  \partial_t z = J { \Omega}^{nls}(I) [z]\,, \quad J = \begin{pmatrix}
  0 & - {\rm Id} \\
  {\rm Id} & 0 
  \end{pmatrix}\,.
  \end{equation}
    According to the splitting $z = (z_S, z_\bot) \in \C^S \times \C^S \times h^0_{ \bot c}$, we can decompose the equation \eqref{equazione integrabile NLS} as
  \begin{equation}\label{equazione integrabile NLS decomposta}
  \begin{cases}
  \partial_t z_S = J \Omega_S^{nls}(I) [z_S] \\
  \partial_t z_\bot = J \Omega_\bot^{nls}(I) [z_\bot]
  \end{cases}
  \end{equation}
  where 
  \begin{equation}\label{splitting Omega}
  \Omega_S^{nls}(I ) := \begin{pmatrix}
 {\rm diag}_{k \in S} \omega_k^{nls}(I) & 0 \\
 0 & {\rm diag}_{k \in S} \omega_k^{nls}(I) 
  \end{pmatrix}\,, \qquad \Omega_\bot^{nls}(I ) := \begin{pmatrix}
 {\rm diag}_{k \in S^\bot} \omega_k^{nls}(I) & 0 \\
 0 & {\rm diag}_{k \in S^\bot} \omega_k^{nls}(I) 
  \end{pmatrix}\,.
  \end{equation}
 \noindent
Similarly,  by a slight abuse of terminology, we identify $I = (I_k)_{k \in \Z}$ with  $(I_S, I_\bot)$,
  \begin{equation}\label{splitting azioni}
  I = (I_S, I_\bot)\,, \qquad I_S := (I_k)_{k \in S}\,, \qquad I_\bot := (I_k)_{k \in S^\bot}\,.
  \end{equation}
    \noindent
    Note that although the frequencies $\omega_k(I)$ are functions of all the action variables 
    $I_n$, $n \in Z$, the system
    \eqref{equazione integrabile NLS decomposta} decouples since the action variables are invariant in time
    and depend only on the initial data.
  Now let us assume that $z(t) = (z_S(t), 0)$ is a solution of \eqref{equazione integrabile NLS decomposta} with initial data $z (0) = (z_S^{(0)}, 0)$ and 
  consider the equation obtained from \eqref{equazione integrabile NLS decomposta} by linearizing it along
  $(z_S(t), 0)$ with initial data given by $\widehat z^{(0)} = ( 0, \widehat z_\bot^{(0)})$ and 
  $\widehat z_\bot^{(0)} \in h^1_{\bot r}$ and denote by $\widehat z(t)$ the corresponding solution
  which evolves in $h^1_r$. By a straightforward computation one verifies that the differential of $\Omega^{nls}(I )$ at 
  $(z_S^{(0)}, 0)$ in direction $( 0, \widehat z_\bot^{(0)})$ vanishes, implying that
  $\widehat z(t) = (0, \widehat z_\bot(t) )$ where 
  $\widehat z_\bot(t)$ is the solution of 
  \begin{equation}\label{linearization in z variable}
  \partial_t \widehat z_\bot(t) = J \Omega_\bot(I_S, 0) [\widehat z_\bot(t)]\,, \qquad
   \widehat z_\bot (0) = \widehat z_\bot^{(0)} \, .
 \end{equation}
Since  by Theorem \ref{Theorem Birkhoff coordinates}, $\Psi^{nls}: h^1_r \to H^1_r$ is symplectic 
it follows that
  \begin{equation}\label{maradona 0}
  \widehat w(t) := d \Psi^{nls}(z_S(t), 0)[(0, \widehat z_\bot(t))] = d_\bot \Psi^{nls}(z_S(t), 0)[\widehat z_\bot(t)]
  \end{equation}
  is a solution of the equation obtained by linearizing  the dNLS equation along $\Psi^{nls}(z_S(t), 0)$.
  More precisely, 
  \begin{equation}\label{maradona 2}
  \partial_t \widehat w(t) = \ii {\mathbb J} d \nabla {\cal H}^{nls}(\Psi^{nls}(z_S(t), 0))[\widehat w(t)]\,,
  \qquad \widehat w(0) = d \Psi^{nls}(z_S^{(0)}, 0)[(0, \widehat z_\bot^{(0)})] \, .
  \end{equation}
  On the other hand, by differentiating formula \eqref{maradona 0} with respect to $t$, one gets  
  \begin{align}
  \partial_t \widehat w (t)& = d_\bot \Psi^{nls}(z_S(t), 0)[\partial_t \widehat z_\bot(t)] + d_S \big( d_\bot \Psi^{nls}(z_S(t), 0)[\widehat z_\bot(t)] \big) [\partial_t z_S(t) ]\nonumber\\
  & = d_\bot \Psi^{nls}(z_S(t), 0) \big[ J \Omega_\bot^{nls}(I_S, 0) \widehat z_\bot(t) \big] +
  d_S \big( d_\bot \Psi^{nls}(z_S(t), 0)[\widehat z_\bot(t)] \big)[ J \Omega_S^{nls}(I_S, 0) z_S(t)]\,. 
  \label{maradona 1} 
  \end{align}
  Comparing \eqref{maradona 2} and \eqref{maradona 1} 
  one gets 
  \begin{align}
\ii {\mathbb J} d \nabla {\cal H}^{nls}(\Psi^{nls}(z_S(t) , 0)) & \big[ d_\bot \Psi^{nls}(z_S(t), 0) \widehat z_\bot  \big] 
=  d_\bot \Psi^{nls}(z_S(t), 0) \big[ J \Omega^{nls}_\bot(I_S, 0) \widehat z_\bot(t) \big] \,  \nonumber\\
  & \quad + \, d_S \big(d_\bot \Psi^{nls}(z_S(t), 0)[\widehat z_\bot(t) ] \big)[ J \Omega_S^{nls}(I_S, 0) z_S(t)]\,.  \label{maradona 3}
  \end{align}
  The latter identity implies that for any $z_S \in \R^S \times \R^S$, 
  $\widehat z_\bot \in h^1_{\bot r}$,
  \begin{align}
  \ii {\mathbb J} d \nabla {\cal H}^{nls}(\Psi^{nls}(z_S, 0)) & \big[ d \Psi^{nls}(z_S, 0)[(0, \widehat z_\bot)] \big]
   =  d \Psi^{nls}(z_S, 0) J \Omega^{nls}(I_S, 0) [(0, \widehat z_\bot )]  \,  \nonumber\\
  & \quad  + \, d_S \big(d_\bot \Psi^{nls}(z_S, 0)[\widehat z_\bot] \big) [ J \Omega_S^{nls}(I_S, 0) z_S]\,.  \label{maradona 4}
  \end{align}
Solving for $J \Omega^{nls}(I_S, 0)[(0, \widehat z_\bot)] $, one gets
  \begin{align}
  J \Omega^{nls}(I_S, 0)[(0, \widehat z_\bot )] & = (d \Psi^{nls}(z_S, 0))^{- 1}  \ii {\mathbb J} d \nabla {\cal H}^{nls}(\Psi^{nls}(z_S, 0)) \big[d \Psi^{nls}(z_S, 0) \, (0, \widehat z_\bot) \big] \nonumber\\
  & \qquad - \, (d \Psi^{nls}(z_S, 0))^{- 1} d_S \big(d_\bot \Psi^{nls}(z_S, 0)[\widehat z_\bot] \big) [ J \Omega_S^{nls}(I_S, 0) z_S]\,. \label{maradona 5}
  \end{align}
  Since $ \Psi^{nls}$ is symplectic, one has 
  $$
  (d \Psi^{nls}(z_S, 0))^{- 1}  \ii {\mathbb J}  = J (d \Psi^{nls}(z_S, 0))^t\,, \quad  (d \Psi^{nls}(z_S, 0))^{- 1} 
  = J (d \Psi^{nls}(z_S, 0))^{ t} \ii {\mathbb J}
  $$
  and hence \eqref{maradona 5} reads 
  \begin{align}
  \Omega^{nls}(I_S, 0)[(0, \widehat z_\bot)]  &  = (d \Psi^{nls}(z_S, 0))^t 
   d \nabla {\cal H}^{nls}(\Psi^{nls}(z_S, 0))d \Psi^{nls}(z_S, 0)[(0, \widehat z_\bot)] - 
   {\cal R}^{(1)}(z_S)[\widehat z_\bot]   \label{identita importante 3}
  \end{align}
  where $ {\cal R}^{(1)}(z_S) : h^{0}_{\bot c} \to h^{0}_{ c}$ is the bounded linear operator, defined by
  \begin{equation}\label{definizione cal M (wS)}
  {\cal R}^{(1)}(z_S)[\widehat z_\bot] :=  \, d \Psi^{nls}(z_S, 0)^{ t} \ii {\mathbb J} 
  d_S \big(d_\bot \Psi^{nls}(z_S, 0)[\widehat z_\bot] \big) [ J \Omega_S^{nls}(I_S, 0) z_S]\,.
  \end{equation}
  For later use we record the following estimates for ${\cal R}^{(1)}(z_S)$. 
  \begin{lemma}\label{stime cal M(wS)}
  The map ${\cal V}_S \cap (\R^S \times \R^S) \to {\cal L}(h^{0}_{\bot c}, h^{0}_c)$, $z_S \mapsto {\cal R}^{(1)} (z_S)$ is real analytic and bounded. Moreover it is one smoothing, meaning that for any $s \in \Z_{\geq 1}$, ${\cal V}_S \cap (\R^S \times \R^S) \to {\cal L}(h^{s}_{\bot c}, h^{s + 1}_c)$, $z_S \mapsto {\cal R}^{(1)}(z_S)$ is real analytic. Furthermore, for any $s \in \Z_{\geq 1}$, $\alpha, \beta \in \Z_{\geq 0}^S$, $z_S \in {\cal V}_S \cap (\R^S \times \R^S)$,  
  $$
   \| \partial_S^{\alpha, \beta}{\cal R}^{(1)}(z_S)\|_{{\cal L}(h^0_{\bot c}, h^{0}_c)} \lesssim_{ \alpha, \beta} \,\,1\,, \quad  \| \partial_S^{\alpha, \beta}{\cal R}^{(1)}(z_S)\|_{{\cal L}(h^s_{ \bot c}, h^{s + 1}_c)} \lesssim_{s, \alpha, \beta} \,\,1\,.
  $$
  \end{lemma}
  \begin{proof}
 By Theorem \ref{Theorem Birkhoff coordinates}, $\Psi^{nls} = F_{nls}^{- 1} + B^{nls}$ and hence 
  $d_S \big( d_\bot \Psi^{nls}(z_S, 0)[\widehat z_\bot] \big)= d_S \big( d_\bot B^{nls}(z_S, 0)[\widehat z_\bot] \big)$.
The claimed statements then follow from Theorem \ref{Theorem Birkhoff coordinates}.
    \end{proof}
  
  \bigskip
  
  \noindent
  We also need to record some properties of the operator  $\Omega_\bot^{nls}(I)$
  for $I = (I_S, 0)$. Write 
  \begin{equation}\label{splitting Omega bot (IS)}
  \Omega_\bot^{nls}(I_S, 0) = D^2_\bot + \Omega_\bot^{(0)}(I_S, 0)\,, 
  \end{equation}
  where  
  \begin{equation}\label{definizione D bot}
  D_\bot := \begin{pmatrix}
   {\rm diag}_{n \in S^\bot} (2 \pi n) & 0 \\
   0 &  {\rm diag}_{n \in S^\bot} (2 \pi n)
   \end{pmatrix}\,,
  \end{equation}
  and
  \begin{equation}\label{definizione Omega bot (0) (IS)}
  \Omega_\bot^{(0)}(I_S, 0) := \begin{pmatrix}
   {\rm diag}_{n \in S^\bot} (\omega_n^{nls}(I_S, 0) - 4 \pi^2 n^2) & 0 \\
   0 &  {\rm diag}_{n \in S^\bot} (\omega_n^{nls}(I_S, 0) - 4 \pi^2 n^2)
   \end{pmatrix}\,.
  \end{equation}
   
  \begin{lemma}\label{stime Omega bot (0) (IS)}
 For any $s \in \Z_{\geq 0}$, the map ${\cal V}_S \cap (\R^S \times \R^S) \to {\cal L}(h^s_{\bot c}, h^s_{\bot c})$, 
 $z_S \mapsto \Omega_\bot^{(0)}(I_S(z_S), 0)$
 is real analytic and bounded.
  \end{lemma}
  \begin{proof}
  The lemma is a straightforward application of Proposition \ref{Asintotica frequenze integrabile}, since for any $\alpha, \beta \in \Z_{\geq 0}^S$ 
  $$
  \sup_{n \in S^\bot} |\partial_S^{\alpha, \beta} \big( \omega_n^{nls}(I_S, 0) - 4 \pi^2 n^2 \big)| 
  \lesssim_{\alpha, \beta} \,\,1
  $$
  and 
  $$
  \| \partial_S^{\alpha, \beta} \Omega_\bot^{(0)}(I_S, 0)  \|_{{\cal L}(h^s_{\bot c}, \, h^s_{\bot c})} 
  \lesssim \sup_{n \in S^\bot} |\partial_S^{\alpha, \beta} \big( \omega_n^{nls}(I_S, 0) - 4 \pi^2 n^2 \big)| \lesssim_{\alpha, \beta} \,\,1
  $$
  uniformly on ${\cal V}_S \cap (\R^S \times \R^S)$. 
  \end{proof}

  \bigskip

  \noindent
    After this preliminary discussion, we can now study the transformed Hamiltonian ${\cal H}^{nls} \circ \Psi$ where $\Psi = \Psi_L \circ \Psi_C$ is the symplectic transformation introduced in Subsection \ref{sec new coordinates}. We split the analysis into two parts. First we expand ${\cal H}^{(1)} := {\cal H}^{nls} \circ \Psi_L$ and then we analyze ${\cal H}^{(2)} = {\cal H}^{(1)} \circ \Psi_C$. 
   
   \bigskip
   
  \noindent 
 {\bf Expansion of ${\cal H}^{nls} \circ \Psi_L$ }
 
 \smallskip
 
 \noindent 
To expand ${\cal H}^{nls} \circ \Psi_L$, it is useful to write ${\cal H}^{nls}$ in the form   
  \begin{equation}\label{forma compatta hamiltoniana d-NLS}
  {\cal H}^{nls}(w) ={\cal H}_2^{nls} (w)+ {\cal H}_4^{nls}(w)  
  \end{equation}
  where 
  \begin{equation}\label{definizione cal H 2 4 nls}
  {\cal H}_2^{nls}(w) := \frac12 \big\langle {\cal D}_2 w\,,\,w\big\rangle_r\,, 
  \qquad {\cal H}_4^{nls}(w) := \int_{\T} u^2 v^2 \,dx\,,
  \end{equation}
  and the operator ${\cal D}_2$ is defined as 
  $$
  {\cal D}_2 := \begin{pmatrix}
  0 & - \partial_{xx} \\
  - \partial_{xx} & 0
  \end{pmatrix}\,.
  $$
  Note that ${\cal D}_2 = {\cal D}^t_2$. The Hamiltonian equations associated to \eqref{forma compatta hamiltoniana d-NLS} can be written as 
  \begin{equation}
  \partial_t w = \ii {\mathbb J} \nabla {\cal H}^{nls}(w)\,, \qquad {\mathbb J} = \begin{pmatrix}
  0 & - {\rm Id} \\
  {\rm Id} & 0
  \end{pmatrix}\,, \qquad  \nabla {\cal H}^{nls} = (\nabla_u {\cal H}^{nls}, \nabla_v {\cal H}^{nls})
  \end{equation}
where 
  \begin{equation}\label{espressione gradiente Hamiltoniana originaria nls}
  \nabla {\cal H}^{nls}(w) = {\cal D}_2 w + \nabla {\cal H}_4^{nls}(w)\,, \quad   
  d \nabla {\cal H}^{nls}(w) = {\cal D}_2 + d \nabla {\cal H}_4^{nls}(w)\, .
  \end{equation}

  \bigskip
  
  \noindent
  The Taylor expansion of ${\cal H}^{nls}$ around $\Psi^{nls}(\Pi_S z)$ up to order three reads
\begin{align}
{\cal H}^{nls}(\Psi^{nls}(\Pi_S z) + w) & = {\cal H}^{nls}(\Psi^{nls}(\Pi_S z)) + 
\langle \nabla {\cal H}^{nls}(\Psi^{nls}(\Pi_S z)), w \rangle_r + 
\frac12 \langle d \nabla {\cal H}^{nls}(\Psi^{nls}(\Pi_S z))[w]\,,\, w \rangle_r \nonumber\\
& \quad +   {\cal T}_{3}^{(1)}(z_S, w) \label{espansione taylor H4 nls}
\end{align}
where ${\cal T}_{3}^{(1)}(z_S, w) $ is the Taylor remainder term of order three, given by
\begin{align}
{\cal T}_{3}^{(1)}(z_S, w) & := \frac12 \int_0^1 (1 - t)^2 d^3 {\cal H}^{nls}(\Psi^{nls}(\Pi_S z) + t w )[w, w, w]\, d t \nonumber\\
 & \stackrel{\eqref{forma compatta hamiltoniana d-NLS}, \eqref{definizione cal H 2 4 nls}}{=} \frac12 \int_0^1 (1 - t)^2 d^3 {\cal H}^{nls}_4(\Psi^{nls}(\Pi_S z) + t w )[w, w, w]\, d t \,. \label{resto ordine 3 taylor H4 nls}
\end{align}
For later use we record that the third derivative of $ {\cal H}^{nls}_4$ at
 $w_0 = (u_0, v_0) \in H^1_r$ in direction $w = (u, v)$ in $H^1_r$ can be computed as 
\begin{equation}\label{formula for third derivative of cal H4nls}
d^3 {\cal H}^{nls}_4(w_0)[w, w, w] = 12 \int_0^1 \big( u_0 u v^2 + u^2 v_0 v  \big) dx\,.
\end{equation}
Substituting for $w$ the function $d_\bot \Psi^{nls}(\Pi_S z)[z_\bot]$ $(= d \Psi^{nls}(\Pi_S z)[\Pi_\bot z])$ and taking into account that by \eqref{seconda def Psi L}, $\Psi_L(z) = \Psi^{nls}(\Pi_S z) + d_\bot \Psi^{nls}(\Pi_S z)[z_\bot]$ yields  
  \begin{align}
 {\cal H}^{(1)}(z)  = {\cal H}^{nls}(\Psi_L(z)) & = {\cal H}^{nls}(\Psi^{nls}(\Pi_S z)) + \langle \nabla {\cal H}^{nls}(\Psi^{nls}(\Pi_S z)), d \Psi^{nls}(\Pi_S z)[\Pi_\bot z] \rangle_r \nonumber\\
  & \quad + \frac12 \Big\langle d \nabla {\cal H}^{nls}(\Psi^{nls}(\Pi_S z))\big[ d \Psi^{nls}(\Pi_S z)[\Pi_\bot z] \big]\,,\, d \Psi^{nls}(\Pi_S z)[\Pi_\bot z] \Big\rangle_r \nonumber\\
  & \qquad +{\cal T}_3^{(1)}\big(z_S, d \Psi^{nls}(\Pi_S z)[\Pi_\bot z] \big) \,.  \nonumber
  \end{align}
  Writing the right hand side of the latter identity in a more convenient form one gets 
  \begin{align}
   {\cal H}^{(1)}(z)  & = {\cal H}^{nls}(\Psi^{nls}(\Pi_S z)) + 
  \big( \Pi_\bot (d \Psi^{nls}(\Pi_S z))^t \nabla {\cal H}^{nls}(\Psi^{nls}(\Pi_S z)), \, \Pi_\bot z \big)_r \nonumber\\
  & \quad + \frac12 \big( \Pi_\bot (d \Psi^{nls}(\Pi_S z))^t  d \nabla {\cal H}^{nls}(\Psi^{nls}(\Pi_S z))
  d \Psi^{nls}(\Pi_S z)[\Pi_\bot z] \,,\,  \Pi_\bot z \big)_r \nonumber\\
  & \qquad + {\cal T}_3^{(1)}\big(z_S, d \Psi^{nls}(\Pi_S z)[\Pi_\bot z] \big)\,.\label{prima espansione H4 circ Psi}
  \end{align}
   Recall that $H^{nls} ={\cal H}^{nls} \circ \Psi^{nls}$. Hence by 
   Theorem \ref{Theorem Birkhoff coordinates} one gets
  \begin{align}
   {\cal H}^{nls}(\Psi^{nls}(\Pi_S z)) = H^{nls}(I_S, 0)\,. \label{forma finale parte ordine 0 in w bot}
  \end{align}
Furthermore by \eqref{identita importante 2}, 
\begin{equation}\label{parte lineare in w bot = 0}
\Pi_\bot (d \Psi^{nls}(\Pi_S z))^t \nabla {\cal H}^{nls}(\Psi^{nls}(\Pi_S z)) = 0\,.
\end{equation} 
Next, the term in \eqref{prima espansione H4 circ Psi}, which is quadratic in $z_\bot$, can be written as  
  \begin{align}
  &  \frac12 \big( \Pi_\bot (d \Psi^{nls}(\Pi_S z))^t d \nabla {\cal H}^{nls}(\Psi^{nls}(\Pi_S z))
  d \Psi^{nls}(\Pi_S z)[\Pi_\bot z] , \,  \Pi_\bot z \big)_r \nonumber\\
  & \stackrel{\eqref{identita importante 3}}{=} \frac12 \big( { \Omega}^{nls}(I_S, 0)[\Pi_\bot z], \, \Pi_\bot z \big)_r + 
  \frac12 \big( {\cal R}^{(1)}(z_S)[z_\bot] , \, z_\bot \big)_r\,.\label{parte quadratica buona forma normale}
  \end{align}
   Substituting \eqref{forma finale parte ordine 0 in w bot}-\eqref{parte quadratica buona forma normale} into \eqref{prima espansione H4 circ Psi} then yields   
  \begin{align}
& {\cal H}^{(1)} (z) = H^{nls}(I_S, 0) + 
\frac12 \big( { \Omega}^{nls}_\bot(I_S, 0)[ z_\bot], \,  z_\bot \big)_r + {\cal P}_2^{(1)}(z) + {\cal P}_3^{(1)}(z)    \label{forma semifinale H nls circ Psi}
  \end{align}
  where 
  \begin{align}
  & {\cal P}_2^{(1)}(z) :=  \frac12 \big( {\cal R}^{(1)}(z_S)[z_\bot], \, z_\bot \big)_r\,, \quad 
  {\cal P}_3^{(1)}(z) := {\cal T}_3^{(1)}\big(z_S, \, d \Psi^{nls}(\Pi_S z)[\Pi_\bot z] \big)\,. 
  \label{perturbazione composizione con Psi L}
  \end{align}
  \begin{lemma}\label{lemma stima cal P2 P3}
  $(i)$ For any $s \in \Z_{\geq 0}$, ${\cal P}_2^{(1)} : {\cal V}\cap h^s_r \to \R$ is real analytic and the following estimates hold: for any $s \in \Z_{\geq 0}$, $z \in {\cal V} \cap h^s_r$,
  $$
  \| \nabla {\cal P}_2^{(1)} (z)\|_{s } \lesssim_s \| z_\bot\|_{s}
  $$
 and for any $k \in \Z_{\geq 1}$, $\widehat z_1, \ldots, \widehat z_k \in h^s_c$,
  $$
  \| d^k \nabla {\cal P}_2^{(1)}(z)[\widehat z_1, \ldots, \widehat z_k]\|_{s } \lesssim_{s, k} \,\,
  \sum_{j = 1}^k \| \widehat z_j\|_s\prod_{i \neq j} \|\widehat z_i \|_0 + \| z_\bot\|_s \prod_{j = 1}^k \|\widehat z_j \|_0\,.
  $$

  \noindent
  $(ii)$ For any $s \in \Z_{\geq 0}$, ${\cal P}_3^{(1)} : {\cal V}\cap h^s_r \to \R$ is real analytic and the following estimates hold: for any $s \in \Z_{\geq 0}$, $z \in {\cal V} \cap h^s_r$, $\widehat z \in h^s_c$,
 $$
  \| \nabla {\cal P}_3^{(1)}(z)\|_s \lesssim_s \| z_\bot\|_s \| z_\bot\|_0\,, \quad \| d  \nabla {\cal P}_3^{(1)}(z)[\widehat z]\|_s \lesssim_s \| z_\bot\|_s \| \widehat z\|_0 + \| z_\bot\|_0 \| \widehat z\|_s 
  $$
  and for any $k \in \Z_{\geq 2}$, $\widehat z_1, \ldots, \widehat z_k \in h^s_c$,
  $$
  \| d^k \nabla {\cal P}_3^{(1)}(z)[\widehat z_1, \ldots, \widehat z_k]\|_{s} \lesssim_{s, k} \sum_{j = 1}^k \| \widehat z_j\|_s\prod_{i \neq j} \|\widehat z_i \|_0 + \| z_\bot\|_s \prod_{j = 1}^k \|\widehat z_j \|_0\,.
  $$
  \end{lemma}
  \begin{proof}
  Item $(i)$ follows from Lemma \ref{stime cal M(wS)} and item $(ii)$ from 
  \eqref{perturbazione composizione con Psi L}, \eqref{resto ordine 3 taylor H4 nls}, 
  \eqref{formula for third derivative of cal H4nls}, and Theorem \ref{Theorem Birkhoff coordinates}.  
  \end{proof}
  
  \bigskip

\noindent  
{\bf Expansion of ${\cal H}^{(2)} := {\cal H}^{(1)} \circ \Psi_C$}

\smallskip

\noindent
To study the expansion of the composition ${\cal H}^{(2)}= {\cal H}^{(1)} \circ \Psi_C$ of the Hamiltonian ${\cal H}^{(1)}$ with the symplectic corrector $\Psi_C$, constructed in Section \ref{costuzione mappa Psi simplettica}, 
we separately expand the compositions of the terms on the right hand side of the identity \eqref{forma semifinale H nls circ Psi} with $\Psi_C$. 
In addition to the projectors $\Pi_S, \Pi_\bot$, defined in \eqref{Pi S}, \eqref{Pi bot}, we also introduce
the following versions of them, 
\begin{equation}\label{pi S}
  \pi_S : \C^S \times \C^S \times h^0_{ \bot c} \to \C^S \times \C^S\,, \quad z = (z_S, z_\bot) \to z_S\,,
  \end{equation}
  \begin{equation}\label{pi bot}
  \pi_\bot : \C^S \times \C^S \times h^0_{ \bot c} \to h^0_{ \bot c}\,, \quad z= (z_S, z_\bot) \to z_\bot\,.
  \end{equation}
\noindent

\noindent
{\bf Term $H^{nls}(I_S, 0)$:} It is convenient to define
\begin{equation}\label{h nls H nls}
h^{nls}(z_S) := H^{nls}(I_S, 0)
\end{equation}
where we recall that by \eqref{Hamiltoniana integrabile Birkhoff}, \eqref{splitting azioni}
$$
I_S = I_S(z_S) = \big( \frac{1}{2} (x_j^2 + y_j^2) \big)_{j \in S}\,, \quad 
z_S = \big( (x_j)_{j \in S}, (y_j)_{j \in S} \big)\in \R^S \times \R^S\,.
$$
By Corollaries \ref{corollario correttore simplettico}, \ref{espansione quadratica cubica cal R Psi} $\Psi_C(z)$,
defined for $z \in {\cal V}_\delta' \cap h^0_r ,$ is of the form
$\Psi_C(z) = z + B_C(z) = z + B_2^C(z) + B_3^C(z)$. Hence the Taylor expansion of $h^{nls}(\pi_S\Psi_C(z))$ around $z_S$  reads 
\begin{align}\label{h nls circ PsiC}
h^{nls}(\pi_S \Psi_C(z)) & = h^{nls}(z_S) + \nabla_S h^{nls}(z_S) \cdot \pi_S B^C_2 (z) + {\cal P}^{(2a)}_3(z)\,,
\end{align}
where ${\cal P}^{(2a)}_3(z)$ is the Taylor remainder term of order three, given by
\begin{align}
{\cal P}^{(2a)}_3(z) & := \nabla_S h^{nls}(z_S) \cdot \pi_S B^C_3(z) + 
\int_0^1 (1 - t)d_S \nabla_S h^{nls}(z_S + t \, \pi_S B_C(z))[\pi_S B_C(z)] 
\cdot \pi_S B_C(z)\, d t \label{definizione h3 nls}\,.
\end{align}
In the next lemma we provide estimates for the Hamiltonian ${\cal P}^{(2a)}_3(z)$.
\begin{lemma}\label{stime tame h3 nls}
For any $s \in \Z_{\geq 0}$, ${\cal P}^{(2a)}_3 \circ \Psi_C : {\cal V}'_\delta \cap h^s_r \to \R$ is real analytic. Furthermore, $\nabla {\cal P}^{(2a)}_3$ satisfies the following tame estimates: for any $s \in \Z_{\geq 0}$, 
$z \in {\cal V}'_\delta \cap h^s_r$, $\widehat z \in h^s_c$,
  $$
  \| \nabla {\cal P}^{(2a)}_3(z)\|_{s} \lesssim_s  \, \| z_\bot\|_s \| z_\bot\|_0\,, \quad 
  \| d  \nabla {\cal P}^{(2a)}_3(z) [\widehat z]\|_{s } \lesssim_s \, \| z_\bot\|_s \| \widehat z\|_0 + 
  \| z_\bot\|_0 \| \widehat z\|_s  
  $$
  and for any $k \in \Z_{\geq 2}$, $\widehat z_1, \ldots, \widehat z_k \in h^s_c$,
  $$
  \| d^k \nabla {\cal P}^{(2a)}_3(z) [\widehat z_1, \ldots, \widehat z_k]\|_{s } \lesssim_{s, k} \,
  \sum_{j = 1}^k \| \widehat z_j\|_s\prod_{i \neq j} \|\widehat z_i \|_0 + 
  \| z_\bot\|_s \prod_{j = 1}^k \|\widehat z_j \|_0\,.
  $$
\end{lemma}
\begin{proof}
The Lemma follows by differentiating ${\cal P}^{(2a)}_3$ and applying the estimates of 
Corollaries \ref{corollario correttore simplettico}, \ref{espansione quadratica cubica cal R Psi}. 
\end{proof}
  \bigskip
  
  \noindent
  {\bf Term ${\cal H}_\Omega(z) := \frac12 \big( \Omega_\bot^{nls}(I_S, 0) z_\bot, z_\bot \big)_r $:} 
To begin with let us point out that the expansion of the composition of the term ${\cal H}_\Omega(z)$ 
with the transformation $\Psi_C$ needs special care. To expain this in more detail, write $2{\cal H}_\Omega(z)$
in the form
  $$
    \big( \Omega_\bot^{nls}(I_S, 0) z_\bot, \, z_\bot \big)_r  =  
   \ \big( D_\bot^2 z_\bot, \, z_\bot  \big)_r 
   +  \ \big(  \Omega_\bot^{(0)}(I_S, 0) z_\bot, \, z_\bot  \big)_r 
  $$
  where $D_\bot$ is the diagonal operator defined in \eqref{definizione D bot}.
When composed with $\Psi_C = { \io d} + B_C$, the term $  \big( D_\bot^2z_\bot, z_\bot \big)_r$ becomes 
\begin{align}
&  \big( D_\bot^2 [z_\bot + \pi_\bot B_C(z)], \, z_\bot + \pi_\bot B_C(z) \big)_r =  
 \big( D_\bot^2 [z_\bot] , \, z_\bot  \big)_r +  
 \big( D_\bot^2 [z_\bot] , \,  \pi_\bot B_C(z) \big)_r  \nonumber\\
& \qquad +   \big( D_\bot^2 [ \pi_\bot B_C(z)], \, z_\bot  \big)_r +  
 \big( D_\bot^2 [ \pi_\bot B_C(z)],  \, \pi_\bot B_C(z) \big)_r
\end{align}
where $\pi_\bot$ is defined in \eqref{pi bot}. By
\eqref{definizione cal R Psi (2)} - \eqref{definizione cal R Psi (3)}, it then follows that the difference
$$
 \frac12  \big( D_\bot^2 [z_\bot + \pi_\bot B_C(z)], \, z_\bot + \pi_\bot B_C(z) \big)_r  - \,
\frac12   \big( D_\bot^2 [z_\bot] , \, z_\bot  \big)_r
$$
 belongs
to the error term ${\cal P}_3(z)$ in Theorem \ref{modified Birkhoff map}. 
Since $B_C$ is only one smoothing, the two terms
$$
\big( D_\bot^2 [z_\bot] ,  \, \pi_\bot B_C(z) \big)_r \,,  \qquad 
 \big( D_\bot^2 [ \pi_\bot B_C(z)], \, z_\bot  \big)_r
$$ 
could prevent that  ${\cal P}_3$ satisfies the estimates \eqref{estimate error term},
stated in Theorem \ref{modified Birkhoff map}. 

To proceed, recall that $\Psi_C = \Psi^{0, 1}_{X}$ where $\Psi_{X}^{\tau_0, \tau}$ is the flow map, defined in 
\eqref  {domini flusso cal F}. We have 
  \begin{equation}\label{parte pericolosa Psi C}
  {\cal H}_\Omega(\Psi_C(z)) = {\cal H}_\Omega(z) + {\cal P}_3^{(2b)}(z)\,, 
  \qquad {\cal P}_3^{(2b)}(z) := {\cal H}_\Omega(\Psi_C(z)) - {\cal H}_\Omega(z)\,. 
  \end{equation}
  Using the mean value theorem and recalling \eqref{domini flusso cal F}, one has
  \begin{align}
  & {\cal P}_3^{(2b)}(z) = \int_0^1 {\cal P}_{\Omega}(\Psi_{X}^{0, \tau}(z), \tau) \, d \tau   \label{caniggia 0}
  \end{align}
where for any $\tau \in [0, 1]$, the Hamiltonian ${\cal P}_{\Omega}(z,  \tau)$ is defined by
  \begin{equation}\label{definition cal H Omega tau}
{\cal P}_{\Omega} (z, \tau) := \big( \nabla {\cal H}_\Omega(z), \, X(z, \tau) \big)_r \,.
  \end{equation}
  One has that 
  \begin{align}
   \big( \nabla {\cal H}_\Omega(z), \, X(z, \tau) \big)_r &  = 
   \frac12  \nabla_S {\cal H}_\Omega(z) \cdot \pi_S X(z, \tau) + 
   \big( \Omega^{nls}_\bot(I_S, 0) z_\bot, \,  \pi_\bot {X}(z, \tau)  \big)_r\,. \label{termine delicato sergey}
  \end{align}
   By \eqref{definizione campo vettoriale ausiliario}, the vector field $X(z, \tau)$ was chosen to be 
  $$
  X(z, \tau) = - {\cal L}_\tau(z)^{- 1} E(z)
  $$
  where $E(z)$ is given by \eqref{forma finale 1 forma Kuksin 2} and ${\cal L}_\tau(z)^{- 1}$ by the Neumann series \eqref{definizione Lt (w)} in Lemma \ref{stime cal Lt inverso}. Hence 
  \begin{align}
  X(z, \tau) & = - {\cal L}_\tau(z)^{- 1} E(z) = - J E(z) -  \sum_{n \geq 1}(- 1)^n  \tau^n (J L(z))^n JE(z) \nonumber\\
  & = - J E(z) + \tau J L(z) \sum_{n \geq 0} (- 1)^n \tau^n (J L(z))^n JE(z) \nonumber\\
  & = - J E(z) + \tau J L(z) X(z, \tau)\,. \label{proprieta Neumann F tau w}
  \end{align}
 Since $E = \Pi_S E$ and $J^t = - J$, the last term in \eqref{termine delicato sergey} becomes 
  \begin{align}
  \big( \Omega^{nls}_\bot(I_S, 0) z_\bot,  \, \pi_\bot X(z, \tau) \big)_r & = 
  \big( \Omega^{nls}_\bot(I_S, 0) z_\bot, \, \pi_\bot \tau J L(z) X(z, \tau) \big)_r 
   \nonumber\\
  & = - \tau  \big( J \Omega^{nls}_\bot(I_S, 0) z_\bot, \, \pi_\bot  L(z) X(z, \tau) \big)_r \,. \label{torres 0}
  \end{align}
  By \eqref{definizione Omega(w)}, the component $L_\bot^\bot(z)$ of $L(z)$ vanishes. 
  Hence using the projections introduced in \eqref{pi S}, \eqref{pi bot}, one has 
  $$
  \pi_\bot L(z) X(z, \tau) = L_\bot^S(z) \pi_S X(z, \tau)\,.
  $$
  Substituting the latter expression into \eqref{torres 0} then leads to
  \begin{align}
    \big( \Omega^{nls}_\bot(I_S, 0) z_\bot, \, \pi_\bot X(z, \tau) \big)_r   & 
    = - \tau  \big( J \Omega^{nls}_\bot(I_S, 0) z_\bot, \, L_\bot^S(z) \pi_S X(z, \tau) \big)_r \nonumber\\
   & = - \tau\,\,  L_\bot^S(z)^t J \Omega^{nls}_\bot(I_S, 0) z_\bot  \cdot \pi_S X(z, \tau) \nonumber\\
   & \stackrel{\eqref{proprieta blocchi L(w)}}{=} 
   \tau\,\, L_S^\bot(z) J \Omega^{nls}_\bot(I_S, 0) z_\bot  \cdot \pi_S X(z, \tau)\,. \label{cappello 10}
  \end{align}
 Note that by the definition \eqref{definizione L S bot},
  \begin{align}
  &  L_S^\bot(z) J \Omega^{nls}_\bot(I_S, 0) z_\bot = \begin{pmatrix}
  \big( \ii \Big\langle {\mathbb J}  d_\bot \Psi^{nls}(\Pi_S z)[ J \Omega^{nls}_\bot(I_S, 0) z_\bot ] \,,\, \partial_{x_j}  d_\bot \Psi^{nls}(\Pi_S z)[ z_\bot]   \Big\rangle_r \big)_{j \in S} \\
   \big( \ii \Big\langle {\mathbb J}  d_\bot \Psi^{nls}(\Pi_S z)[ J \Omega^{nls}_\bot(I_S, 0) z_\bot ] \,,\, \partial_{y_j}  d_\bot \Psi^{nls}(\Pi_S z)[ z_\bot]     \Big\rangle_r \big)_{j \in S}
  \end{pmatrix}\,. \label{cappello 0}
  \end{align}
 Let us take a closer look at the expression
 $$
 d_\bot \Psi^{nls}(\Pi_S z)[ J \Omega^{nls}_\bot(I_S, 0) z_\bot ] = 
 d \Psi^{nls}(\Pi_S z)[ J \Omega^{nls}(I_S, 0) (0, z_\bot) ]\,.
 $$
Substituting for $J \Omega^{nls}(I_S, 0)(0, z_\bot)$ the right hand side of 
the identity \eqref{identita importante 3}, one gets
\begin{align}
 d \Psi^{nls}(\Pi_S z)[ J \Omega^{nls}(I_S, 0) (0, z_\bot) ] = 
& \,  d \Psi^{nls}(\Pi_S z) J (d \Psi^{nls}(\Pi_S z))^t 
   d \nabla {\cal H}^{nls}(\Psi^{nls}(\Pi_S z)d \Psi^{nls}(\Pi_S z)[(0, \widehat z_\bot)]  \nonumber\\ 
& - \,  d \Psi^{nls}(\Pi_S z) J  {\cal R}^{(1)}(z_S)[\widehat z_\bot]\,. \nonumber
\end{align}
Note that the first term on the right hand side of the latter identity can be simplified. 
Since $\Psi^{nls}$ is symplectic, 
  $$
d \Psi^{nls}(\Pi_S z) J (d \Psi^{nls}(\Pi_S z))^t 
= \ii {\mathbb J} \,, 
  $$
one has
\begin{align}
& d \Psi^{nls}(\Pi_S z) J (d \Psi^{nls}(\Pi_S z))^t 
   d \nabla {\cal H}^{nls}(\Psi^{nls}(\Pi_S z))d \Psi^{nls}(\Pi_S z)[(0, \widehat z_\bot)] = \nonumber \\
 & \,    \ii {\mathbb J} d \nabla {\cal H}^{nls}(\Psi^{nls}(\Pi_S z))d \Psi^{nls}(\Pi_S z)[ (0,  z_\bot)]   
 =     \ii {\mathbb J} d \nabla {\cal H}^{nls}(\Psi^{nls}(\Pi_S z))d_\bot \Psi^{nls}(\Pi_S z)[z_\bot] \,.
\end{align}
Combining the above identities, the component
$\ii \big\langle {\mathbb J}  d_\bot \Psi^{nls}(\Pi_S z)[ J \Omega^{nls}_\bot(I_S, 0) z_\bot ] \,,\, \partial_{x_j}  d_\bot \Psi^{nls}(\Pi_S z)[ z_\bot]     \big\rangle_r$ on the right hand side of \eqref{cappello 0}
becomes, for $j \in S$ arbitrary,
  \begin{align}
  & \ii \Big\langle {\mathbb J}  d_\bot \Psi^{nls}(\Pi_S z)[ J \Omega^{nls}_\bot(I_S, 0) z_\bot ] \,,\, \partial_{x_j}  d_\bot \Psi^{nls}(\Pi_S z)[ z_\bot]     \Big\rangle_r  \nonumber\\
  & =   \Big\langle   d \nabla {\cal H}^{nls}(\Psi^{nls}(\Pi_S z))d_\bot \Psi^{nls}(\Pi_S z)[  z_\bot]  \,,\, \partial_{x_j}  d_\bot \Psi^{nls}(\Pi_S z)[ z_\bot]    \Big\rangle_r \nonumber\\
  &  \quad - \, \ii \Big\langle {\mathbb J}  d \Psi^{nls}(\Pi_S z) J {\cal R}^{(1)}(z_S) [z_\bot ] \,,\, \partial_{x_j}  d_\bot \Psi^{nls}(\Pi_S z)[ z_\bot]     \Big\rangle_r  
  \end{align}
  which in view of
  $d \nabla {\cal H}^{nls}(w) = {\cal D}_2 + d \nabla {\cal H}_4^{nls}(w)$
  (cf  \eqref{espressione gradiente Hamiltoniana originaria nls})
  leads to
  \begin{align}
  &  \ii \Big\langle {\mathbb J}  d_\bot \Psi^{nls}(\Pi_S z)[ J \Omega^{nls}_\bot(I_S, 0) z_\bot ] \,,\, \partial_{x_j}  d_\bot \Psi^{nls}(\Pi_S z)[ z_\bot]     \Big\rangle_r  \nonumber\\
  &=    \Big\langle   {\cal D}_2d_\bot \Psi^{nls}(z_S, 0)[ z_\bot]  \,,\, \partial_{x_j}  d_\bot \Psi^{nls}(\Pi_S z)[ z_\bot]     \Big\rangle_r \nonumber\\
  & \, + \,  \Big\langle   d \nabla {\cal H}_4^{nls}(\Psi^{nls}(\Pi_S z))d_\bot \Psi^{nls}(\Pi_S z)[ z_\bot]  \,,\, \partial_{x_j}  d_\bot \Psi^{nls}(\Pi_S z)[ z_\bot]     \Big\rangle_r \nonumber\\
  &  \quad - \,  \ii \Big\langle {\mathbb J}  d \Psi^{nls}(\Pi_S z) J {\cal R}^{(1)}(z_S) [z_\bot ] \,,\, \partial_{x_j} d_\bot \Psi^{nls}(\Pi_S z)[ z_\bot]     \Big\rangle_r\,.  \label{cappello -1}
    \end{align}
  Since ${\cal D}_2 = {\cal D}_2^t$, the first term on the right hand side on the latter identity can be written as
  \begin{align}
   \big\langle   {\cal D}_2 d_\bot \Psi^{nls}(\Pi_S z)[  z_\bot]  \,,\, 
 &  \partial_{x_j} d_\bot \Psi^{nls}(\Pi_S z)[ z_\bot] \big\rangle_r  = 
  \frac12 \partial_{x_j} \big\langle   {\cal D}_2 d_\bot \Psi^{nls}(\Pi_S z)[  z_\bot]  \,,\,  
  d_\bot \Psi^{nls}(\Pi_S z)[ z_\bot]   \big\rangle_r \nonumber\\
  & = \, \frac12 \partial_{x_j} \big\langle   {\cal D}_2 d \Psi^{nls}(\Pi_S z)[(0,   z_\bot)]  \,,\,  
  d \Psi^{nls}(\Pi_S z)[(0, z_\bot)]   \big\rangle_r \nonumber\\
  & = \, \frac12 \partial_{x_j} \big(  (d \Psi^{nls}(\Pi_S z))^t {\cal D}_2 d \Psi^{nls}(\Pi_S z)[(0,   z_\bot)]  \,,\,  
  (0, z_\bot)   \big)_r  \,, \label{cappello 2}
  \end{align}
  which can be further transformed as follows: 
  using  $ {\cal D}_2 \, = \,  d \nabla {\cal H}^{nls} -  d \nabla {\cal H}_4^{nls} $ 
  (cf \eqref{espressione gradiente Hamiltoniana originaria nls}) and taking into account
  that by \eqref{identita importante 3},
  $$
     (d \Psi^{nls}(z_S, 0))^t d \nabla {\cal H}^{nls}(\Psi^{nls}(z_S, 0))d \Psi^{nls}(z_S, 0)[(0, z_\bot)]
     =   \Omega^{nls}(I_S, 0)[(0, z_\bot)] + 
   {\cal R}^{(1)}(z_S)[ z_\bot] 
  $$
  one is lead to
  \begin{align}
  &
  \frac12 \partial_{x_j} \big(  (d \Psi^{nls}(\Pi_S z))^t {\cal D}_2 d \Psi^{nls}(\Pi_S z)[(0,   z_\bot)]  \,,\,  
  (0, z_\bot)   \big)_r 
   = \frac12  \partial_{x_j} \big( \Omega^{nls}(I_S, 0) [( 0, z_\bot)] , \, ( 0, z_\bot) \big)_r  \, + \nonumber\\
 & \frac12 \partial_{x_j}  \big(  {\cal R}^{(1)}(z_S) [z_\bot ], (0, z_\bot) \big)_r 
   - \frac12 \partial_{x_j} \Big\langle   d \nabla {\cal H}_4^{nls}(\Psi^{nls}(\Pi_S z))d_\bot \Psi^{nls}(z_S, 0)[  z_\bot]  \,,\,  d_\bot \Psi^{nls}(\Pi_S z)[ z_\bot]   \Big\rangle_r\,. \nonumber
  \end{align}
  Let us analyze 
  $ \partial_{x_j} \big( \Omega^{nls}(I_S, 0) [( 0, z_\bot)] , \, ( 0, z_\bot) \big)_r  = 
  \big( \partial_{x_j}\Omega_\bot^{nls}(I_S, 0) z_\bot, \, z_\bot \big)_r $ 
  in more detail.
 Substituting for  $\Omega_\bot^{nls}(I_S, 0)$ the expression $D^2_\bot + \Omega_\bot^{(0)}(I_S, 0)$
 (cf \eqref{splitting Omega bot (IS)})
 and using that $\big( \partial_{x_j} D^2_\bot z_\bot, \, z_\bot \big)_r =0 $ for any $j \in S,$ one concludes that
 $$
 \big( \partial_{x_j}\Omega_\bot^{nls}(I_S, 0) z_\bot, \, z_\bot \big)_r = 
 \big( \partial_{x_j}\Omega_\bot^{(0)}(I_S, 0) z_\bot, z_\bot \big)_r\,, \qquad \forall \, j \in S\,.
 $$
 The above identities then imply that \eqref{cappello 2}  becomes 
  \begin{align}
  & \Big\langle   {\cal D}_2 d_\bot \Psi^{nls}(\Pi_S z)[  z_\bot]  \,,\, 
  \partial_{x_j} d_\bot \Psi^{nls}(\Pi_S z)[ z_\bot] \Big\rangle_r = 
  \frac12 \big( \partial_{x_j}\Omega_\bot^{(0)}(I_S, 0) z_\bot, z_\bot \big)_r + 
  \frac12  \big( \partial_{x_j} {\cal R}^{(1)}(w_S) z_\bot, \, (0, z_\bot) \big)_r \nonumber\\
  & - \frac12 \partial_{x_j} \Big\langle   d \nabla {\cal H}_4^{nls}(\Psi^{nls}(\Pi_S z))d_\bot \Psi^{nls}(\Pi_S z)[  z_\bot]  \,,\,  d_\bot \Psi^{nls}(\Pi_S z)[ z_\bot]   \Big\rangle_r\,. \label{cappello 3}
  \end{align}
With \eqref{cappello 2} - \eqref{cappello 3}, the identity \eqref{cappello -1}
becomes
  \begin{align}
 &  \ii \Big\langle {\mathbb J}  d_\bot \Psi^{nls}(\Pi_S z)[ J \Omega^{nls}_\bot(I_S, 0) z_\bot ] \,,\, \partial_{x_j}  d_\bot \Psi^{nls}(\Pi_S z)[ z_\bot]     \Big\rangle_r 
  = \big( {\cal R}_{x_j}(z_S)[z_\bot], z_\bot \big)_r\,,\label{cappello 4}
  \end{align}
  where for any $j \in S$,  ${\cal R}_{x_j}(z_S):  h^{0}_{\bot c} \to h^{0}_{\bot c} $ is 
  the linear operator defined by 
  \begin{align}
 &  \frac12 \partial_{x_j} \Omega_\bot^{(0)}(I_S, 0) + \, 
  \frac12 \pi_\bot \partial_{x_j} {\cal R}^{(1)}(z_S)  
  - \, \frac12 \partial_{x_j} \big( (d_\bot \Psi^{nls}(z_S, 0))^t \, 
  d \nabla {\cal H}_4^{nls}(\Psi^{nls}(\Pi_S z))d_\bot \Psi^{nls}(\Pi_S z) \big) \, + \nonumber\\
  & \, \big( \partial_{x_j}  d_\bot \Psi^{nls}(\Pi_S z)\big)^t \, d \nabla {\cal H}_4^{nls}(\Psi^{nls}(\Pi_S z))d_\bot \Psi^{nls}(\Pi_S z) 
   - \, \ii \big( \partial_{x_j} d_\bot \Psi^{nls}(\Pi_S z) \big)^t {\mathbb J}  \, 
  d \Psi^{nls}(\Pi_S z) J {\cal R}^{(1)}(z_S) \,. \label{definizione cal R xj (zS)}
  \end{align}
  Arguing similarly as above one obtains 
  \begin{align}
 &  \ii \Big\langle {\mathbb J}  d_\bot \Psi^{nls}(\Pi_S z)[ J \Omega^{nls}_\bot(I_S, 0) z_\bot ] \,,\, \partial_{y_j}  d_\bot \Psi^{nls}(\Pi_S z)[ z_\bot]     \Big\rangle_r = \big( {\cal R}_{y_j}(z_S)[z_\bot]\,,\, z_\bot \big)_r 
 \label{cappello 5}
  \end{align}
  where ${\cal R}_{y_j}(z_S): h^{0}_{\bot c} \to h^{0}_{\bot c}$ is given by
  \begin{align}
  & \frac12 \partial_{y_j} \Omega_\bot^{(0)}(I_S, 0) + \,  
  \frac12 \pi_\bot \partial_{y_j} {\cal R}^{(1)}(z_S)  - \, \frac12 \partial_{y_j} \big((d_\bot \Psi^{nls}(z_S, 0))^t \, 
  d \nabla {\cal H}_4^{nls}(\Psi^{nls}(\Pi_S z))d_\bot \Psi^{nls}(\Pi_S z) \big) \, + \nonumber\\
  &  \big( \partial_{y_j}  d_\bot \Psi^{nls}(\Pi_S z)\big)^t \, 
  d \nabla {\cal H}_4^{nls}(\Psi^{nls}(\Pi_S z))d_\bot \Psi^{nls}(\Pi_S z)
  - \,  \ii \big( \partial_{y_j} d_\bot \Psi^{nls}(\Pi_S z) \big)^t \, 
  {\mathbb J}  d \Psi^{nls}(\Pi_S z) J {\cal R}^{(1)}(z_S) \,. \label{definizione cal R yj (zS)}
  \end{align}
  In the next lemma we state estimates for the operators ${\cal R}_{x_j}(z_S)$ and  ${\cal R}_{y_j}(z_S)$. 
  \begin{lemma}\label{stime R xj yj (zS)}
For any $j \in S$ and $s \in \Z_{\geq 0}$, the maps 
$$
{\cal R}_{x_j} :
{\cal V}_S \cap (\R^S \times \R^S) \to {\cal L}(h^{s}_{\bot c}, h^{s }_{\bot c}), \, z_S \mapsto {\cal R}_{x_j}(z_S) \, ,
 \quad
{\cal R}_{y_j} : 
{\cal V}_S \cap (\R^S \times \R^S) \to {\cal L}(h^{s}_{\bot c}, h^{s }_{\bot c}), \, 
z_S \mapsto {\cal R}_{y_j}(z_S)
$$
are real analytic and bounded. Furthermore, for any $\alpha, \beta \in \Z_{\ge 0}^S$,
  $$
  \| \partial_S^{\alpha, \beta}{\cal R}_{x_j}(z_S)\|_{{\cal L}(h^s_{\bot c}, h^{s}_{\bot c})}, \,\,\,
  \| \partial_S^{\alpha, \beta}{\cal R}_{y_j}(z_S)\|_{{\cal L}(h^s_{\bot c}, h^{s}_{\bot c})} \, 
  \lesssim_{s, \alpha, \beta}  \,  \, 1\,.
$$
  \end{lemma}
  \begin{proof}
  The lemma follows from Theorem \ref{Theorem Birkhoff coordinates} and Lemmata \ref{stime cal M(wS)}\,,
  \ref{stime Omega bot (0) (IS)}.
  \end{proof}
 Finally, by \eqref{definition cal H Omega tau}, \eqref{termine delicato sergey}, \eqref{cappello 10}, \eqref{cappello 0}, \eqref{cappello 4}, \eqref{cappello 5} and writing 
 $$
 \pi_S X(z, \tau) =  \big( \big( X_{j, +}(z, \tau) \big)_{j \in S} \,,
  \big( X_{j, -}(z, \tau) \big)_{j \in S} \big) \in \R^S \times \R^S 
 $$
 one sees 
 that the Hamiltonian ${\cal P}_{\Omega} (z, \tau)$, defined by \eqref{definition cal H Omega tau}, can be written in the form 
 \begin{align}
 & \frac12  \nabla_S {\cal H}_\Omega(z) \cdot \pi_S  X(z, \tau)  + 
 \sum_{j \in S} X_{j, +}(z, \tau) \, \big({\cal R}_{x_j}(z_S)[z_\bot], \, z_\bot \big)_r + 
 \sum_{j \in S} X_{j, -}(z, \tau) \,  \big( {\cal R}_{y_j} (z_S)[z_\bot], \, z_\bot \big)_r  \,.
 \label{espressione finale cal H Omega tau 1}
 \end{align}
 In the next lemma we state estimates for the Hamiltonian ${\cal P}_3^{(2b)}$, defined in \eqref{caniggia 0}. 
 \begin{lemma}\label{lemma cal P Omega C}
  For any $s \in \Z_{\geq 0}$, the Hamiltonian ${\cal P}_3^{(2b)} : {\cal V}'_\delta \cap h^s_r \to \R$ is real analytic. Moreover, it satisfies the following tame estimates: for any $s \in \Z_{\geq 0}$, 
   $z \in {\cal V}'_\delta \cap h^s_r$, $\widehat z, \widehat z_1,  \widehat z_2 \in h^s_c$,
  $$
  \| \nabla {\cal P}_3^{(2b)}(z)\|_{s } \lesssim_s \| z_\bot\|_s \| z_\bot\|_0^2\,, \quad 
  \| d  \nabla{\cal P}_3^{(2b)}(z)[\widehat z]\|_{s } \lesssim_s  \, 
  \| z_\bot\|_s \| z_\bot\|_0 \| \widehat z\|_0 + \| z_\bot\|_0^2 \| \widehat z\|_s \,,
  $$
  $$
\| d^2  \nabla{\cal P}_3^{(2b)}(z)[\widehat z_1, \widehat z_2]\|_{s } \lesssim_s \, 
\| z_\bot\|_s  \| \widehat z_1\|_0 \| \widehat z_2\|_0 + 
\| z_\bot\|_0\big(  \| \widehat z_1\|_s \| \widehat z_2\|_0 + \| \widehat z_1\|_0 \| \widehat z_2 \|_s \big)\,,
  $$
  and   for any $k \in \Z_{\geq 3}$,  $\widehat z_1, \ldots, \widehat z_k \in h^s_c$,
  $$
  \| d^k \nabla {\cal P}_3^{(2b)}(z)[\widehat z_1, \ldots, \widehat z_k]\|_{s + 1} \lesssim_{s, k} \,\,
  \sum_{j = 1}^k \| \widehat z_j\|_s\prod_{i \neq j} \|\widehat z_i \|_0 + 
  \| z_\bot\|_s \prod_{j = 1}^k \|\widehat z_j \|_0\,.
  $$ 
  \end{lemma}
  \begin{proof}
  The lemma follows by \eqref{caniggia 0}, \eqref{espressione finale cal H Omega tau 1},  
  and Lemmata \ref{stime campo vettoriale ausiliario}, \ref{lemma stima flusso correttore}, 
  \ref{stime Omega bot (0) (IS)}, \ref{stime R xj yj (zS)}\,. 
\end{proof}
  
  \bigskip
  
  \noindent
  {\bf Term ${\cal P}_2^{(1)}$\, :} Recall that the Hamiltonian ${\cal P}_2^{(1)}$ was introduced in \eqref{perturbazione composizione con Psi L}. For $z \in  {\cal V}'_\delta \cap h^0_r$ one has $\Psi_C(z) = z + B_C(z)$
  and hence the Taylor expansion 
  of $ {\cal P}_2^{(1)}(\Psi_C(z))$ around $z$  reads
  \begin{equation}\label{cal P 12 Psi C}
  {\cal P}_2^{(1)}(\Psi_C(z)) = {\cal P}_2^{(1)}(z) + {\cal P}_3^{(2c)}(z)\,, \quad  
  {\cal P}_3^{(2c)}(z):= \int_0^1 \big( \nabla {\cal P}_2^{(1)}(z + t B_C(z))\,,\, B_C(z) \big)_r\, d t\,.
  \end{equation}
  The following lemma holds:
  \begin{lemma}\label{composizione cal P 12 Psi C}
  For any $s \in \Z_{\geq 0}$, the Hamiltonian 
  ${\cal P}_2^{(1)} \circ \Psi_C : {\cal V}'_\delta \cap h^s_r \to \R$ is real analytic. 
  Moreover,  the Hamiltonian ${\cal P}_3^{(2c)}$, defined in \eqref{cal P 12 Psi C},
   satisfies the following tame estimates: for any $s \in \Z_{\geq 0}$, 
   $z \in {\cal V}'_\delta \cap h^s_r$, $\widehat z \in h^s_c$,
  $$
  \| \nabla {\cal P}_3^{(2c)}(z)\|_{s} \lesssim_s \,  \| z_\bot\|_s \| z_\bot\|_0\,, \quad 
  \| d  \nabla {\cal P}_3^{(2c)}(z)[\widehat z]\|_{s } \lesssim_s \, 
  \| z_\bot\|_s \| \widehat z\|_0 + \| z_\bot\|_0 \| \widehat z\|_s \,,
  $$
  and  for any $k \in \Z_{\geq 2}$, $\widehat z_1, \ldots, \widehat z_k \in h^s_c$,
  $$
  \| d^k \nabla {\cal P}_3^{(2c)}(z)[\widehat z_1, \ldots, \widehat z_k]\|_{s} \lesssim_{s, k} \,
  \sum_{j = 1}^k \| \widehat z_j\|_s\prod_{i \neq j} \|\widehat z_i \|_0 +
   \| z_\bot\|_s \prod_{j = 1}^k \|\widehat z_j \|_0\,.
  $$ 
   \end{lemma}
  \begin{proof}
  The lemma follows by differentiating ${\cal P}_3^{(2c)}$ and applying Corollary \ref{corollario correttore simplettico} and Lemma \ref{lemma stima cal P2 P3} $(i)$.  
  \end{proof}
  
  \bigskip
  
  \noindent
  {\bf Term ${\cal P}_3^{(1)} $ :} By  \eqref{perturbazione composizione con Psi L},
   ${\cal P}_3^{(1)}$ is given by $ {\cal T}_3^{(1)}\big(z_S, \, d \Psi^{nls}(\Pi_S z)[\Pi_\bot z] \big)$
  where  ${\cal T}_3^{(1)}$ is the Taylor remainder term of order three, introduced in \eqref{resto ordine 3 taylor H4 nls}.
  Using the estimates of ${\cal P}_3^{(1)}$ of  Lemma \ref{lemma stima cal P2 P3} $(ii)$,
  the Hamiltonian ${\cal P}_3^{(1)} \circ \Psi_C$ can be estimated as follows: 
  \begin{lemma}\label{composizione cal P 13 Psi C}
  For any $s \in \Z_{\geq 0}$, ${\cal P}_3^{(1)} \circ \Psi_C : {\cal V}'_\delta \cap h^s_r \to \R$ is real analytic. Moreover, the following tame estimates hold: for any $s \in \Z_{\geq 1}$, 
  $z \in {\cal V}'_\delta \cap h^s_r$, $\widehat z \in h^s_c$,
  $$
  \| \nabla ({\cal P}_{3}^{(1)} \circ \Psi_C)(z)\|_{s} \lesssim_s \, \| z_\bot\|_s \| z_\bot\|_0\,, \quad 
  \| d  \nabla ({\cal P}_{3}^{(1)} \circ \Psi_C)(z)[\widehat z]\|_{s } \lesssim_s \,
  \| z_\bot\|_s \| \widehat z\|_0 + \| z_\bot\|_0 \| \widehat z\|_s \,,
  $$
  and  for any $k \in \Z_{\geq 2}$, $\widehat z_1, \ldots, \widehat z_k \in h^s_c$,
  $$
  \| d^k \nabla({\cal P}_{3}^{(1)} \circ \Psi_C)(z)[\widehat z_1, \ldots, \widehat z_k]\|_{s} \lesssim_{s, k} \,
  \sum_{j = 1}^k \| \widehat z_j\|_s\prod_{i \neq j} \|\widehat z_i \|_0 + \| z_\bot\|_s \prod_{j = 1}^k \|\widehat z_j \|_0\,.
  $$
  \end{lemma}
  \begin{proof}
  The lemma follows by differentiating the Hamiltonian ${\cal P}_3^{(1)} \circ \Psi_C$ and applying Corollary \ref{corollario correttore simplettico} and Lemma \ref{lemma stima cal P2 P3} $(ii)$.
  \end{proof}

\bigskip

\noindent
By \eqref{forma semifinale H nls circ Psi}, \eqref{h nls circ PsiC}, \eqref{parte pericolosa Psi C}, \eqref{cal P 12 Psi C} one gets that the Hamiltonian ${\cal H}^{(2)} := {\cal H}^{(1)} \circ \Psi_C = {\cal H}^{nls} \circ \Psi_L \circ \Psi_C$ has the form 
\begin{equation}\label{cal H2 nls}
{\cal H}^{(2)} (z) = H^{nls}(I_S, 0) + \frac12 \big( \Omega_\bot^{nls}(I_S, 0)[z_\bot], z_\bot \big)_r 
+ {\cal P}_2(z) + {\cal P}_3(z)
\end{equation}
where for any $z \in  {\cal V}'_\delta \cap h^0_r$,
\begin{align}
{\cal P}_2 (z)  := \nabla_S \, h^{nls}(z_S) \cdot \pi_S B_2^{C}(z) + {\cal P}_2^{(1)}(z)\,,   
\label{definizione cap P2 (2)} \\
{\cal P}_3 (z)  := {\cal P}^{(2a)}_3(z) +{\cal P}_3^{(2b)}(z) + {\cal P}_3^{(2c)}(z) + {\cal P}_3^{(1)}(\Psi_C(z))\,.  \label{definizione cap P2 (3)}
\end{align}
Note that ${\cal P}_2 $ is quadratic with respect to $z_\bot$, whereas ${\cal P}_3$
is a remainder term of order three in  $z_\bot$. 
Being quadratic with respect to $z_\bot$, ${\cal P}_2$ can be written as 
$$
{\cal P}_2 (z) = \frac12 \big( d_\bot (\nabla_\bot {\cal P}_2(\Pi_S z) )  [z_\bot], \, z_\bot \big)_r\,.
$$
  We prove the following
      \begin{lemma}\label{cancellazione finale termini quadratici}
  The Hamiltonian ${\cal P}_2$ vanishes on $  {\cal V}'_\delta \cap h^0_r$. 
    \end{lemma}
  \begin{proof}
 By Corollary \ref{corollario correttore simplettico}, $\Psi_C(\Pi_S z) = \Pi_S z$ and $d \Psi_C(\Pi_S z) = {\rm Id}$. 
 Hence by the chain rule and formula \eqref{differenziale mappa Psi L}, the map $\Psi = \Psi_L \circ \Psi_C$ satisfies 
  \begin{equation}\label{identity Psi and Psi nls}
  d \Psi(\Pi_S z) = d \Psi_L (\Pi_S z) = d \Psi^{nls}(\Pi_S z)\,.
  \end{equation}
Recall that we denoted by $\widehat w(t)$ the solution of equation \eqref{maradona 2},
obtained by linearizing the dNLS equation along  $w(t) = \Psi^{nls}(\Pi_S z(t))$ 
with initial data $\widehat w(0) = d\Psi^{(nls)}(\Pi_S z (t)) (0, \widehat z_\bot^{(0)})$ and
by $\widehat z(t) = (0, \widehat z_\bot(t))$ the one of the equation
obtained by linearizing  the dNLS equation, expressed in Birkhoff coordinates (cf \eqref{equazione integrabile NLS}),
along $(z_S(t), 0) = \Pi_S z(t)$ with initial data $(0, \widehat z_\bot^{(0)})$.
 Since $\Psi^{nls}$ is symplectic, $\widehat w(t) = d\Psi^{nls}(\Pi_S z(t))[\widehat z(t)]$.
 We remark that $(z_S(t), 0) = \Pi_S z(t)$ is also a solution of the Hamiltonian equation 
 $\partial_t z^{(2)} = J \nabla {\cal H}^{(2)}(z^{(2)})$ with ${\cal H}^{(2)}$ given by \eqref{cal H2 nls}.
 Denote by $\widehat z^{(2)}(t) = (0, \widehat z^{(2)}_\bot(t))$ the solution
 of the equation obtained by linearizing $\partial_t z^{(2)} = J \nabla {\cal H}^{(2)}(z^{(2)})$
 along $\Pi_S z(t)$ with the same initial data $(0, \widehat z_\bot^{(0)})$ as above. 
 Since $\Psi$ is symplectic, $\widehat w(t) = d\Psi(\Pi_S z(t))[\widehat z^{(2)}(t)]$,
implying together with  
 $d \Psi(\Pi_S z) =  d \Psi^{nls}(\Pi_S z)$ (cf \eqref{identity Psi and Psi nls} above) 
 that $\widehat z^{(2)}(t) = \widehat z(t)$ for any $t$.
By \eqref{linearization in z variable},  $\widehat z_\bot (t)$ satisfies
  \begin{equation}\label{pane e vino 1}
  \partial_t  \widehat z_\bot(t) = J  {\Omega}_\bot^{nls}(I_S, 0) [\widehat z_\bot(t)] 
  \end{equation}
 whereas by \eqref{cal H2 nls}, one has
  \begin{equation}\label{pane e vino 2}
  \partial_t \widehat z^{(2)}_\bot (t) =  
  J d_\bot \nabla {\cal H}^{(2)}(\Pi_Sz(t))  [z^{(2)}_\bot(t)] = 
  J \Omega_\bot^{nls}(I_S, 0) [\widehat z^{(2)}_\bot (t)] \, + \,
  J d_\bot \nabla_\bot {\cal P}_2(\Pi_S z(t))[\widehat z^{(2)}_\bot(t)]  \,.
  \end{equation}
  In particular, it follows that  $d_\bot \nabla_\bot {\cal P}_2(\Pi_S z(0))[\widehat z^{(0)}_\bot]  = 0$.
   Since  ${\cal P}_2(z)$ is quadratic in $z_\bot$ and the initial data 
   $z_S(0) \in \pi_S ({\cal V}'_\delta \cap h^0_r)$, \, $\widehat z^{(0)}_\bot \in h^0_{\bot c} $ are
   arbitrary, it follows that 
  $ {\cal P}_2(z) = 0$ for any $z \in  {\cal V}'_\delta \cap h^0_r$,  which proves the claimed statement.
  \end{proof}
As a consequence of  Lemma \ref{cancellazione finale termini quadratici},
formula \eqref{cal H2 nls} becomes
\begin{equation}\label{forma finalissima hamiltoniana trasformata}
{\cal H}^{(2)}(z) = H^{nls}(I_S, 0) + \frac12 \big( {\Omega}^{nls}_\bot(I_S, 0)[z_\bot], \, z_\bot \big)_r + 
{\cal P}_3(z)\,.
\end{equation}

\noindent
The Hamiltonian ${\cal P}_3$, introduced in \eqref{definizione cap P2 (3)}, satisfies the following tame estimates. 
\begin{lemma}[Tame estimates of ${\cal P}_3$] \label{stime finali grado 3 perturbazione}For any $s \in \Z_{\geq 0}$, the Hamiltonian ${\cal P}_3 : {\cal V}'_\delta \cap h^s_r \to \R$ is real analytic and  
satsfies the following tame estimates: 
for any $z \in {\cal V}_\delta' \cap h^s_r$, $\widehat z \in h^s_c$, 
\begin{align}
&  \| \nabla {\cal P}_3(z)\|_s \lesssim_s \| z_\bot\|_s \| z_\bot\|_0\,, \quad \| d \nabla {\cal P}_3(z) [\widehat z] \|_s \lesssim_s \| z_\bot\|_s \| \widehat z\|_0 + \| z_\bot\|_0 \| \widehat z\|_s  \nonumber
\end{align}
and for any $k \in \Z_{\geq 2}$, $\widehat z_1, \ldots, \widehat z_k \in h^s_c$,
$$
\| d^k \nabla {\cal P}_3(z)[\widehat z_1, \ldots, \widehat z_k]\|_s \lesssim_s \sum_{j = 1}^k \| \widehat z_j\|_s \prod_{i \neq j} \| \widehat z_i\|_0 + \| z_\bot\|_s \prod_{j = 1}^k \| \widehat z_j\|_0\,.
$$
\end{lemma} 
\begin{proof}
The claimed statements follow from Lemmata \ref{stime tame h3 nls}, 
\ref{lemma cal P Omega C}, \ref{composizione cal P 12 Psi C}, and \ref{composizione cal P 13 Psi C}.
\end{proof} 

  
 \subsection{Summary of the proof of Theorem \ref{modified Birkhoff map} } \label{summary proof}
  
 Theorem \ref{modified Birkhoff map} is a direct consequence of Propositions \ref{Lemma finale correttore simplettico}, \ref{stima tame cal A Psi}, formula \eqref{forma finalissima hamiltoniana trasformata}, and 
  Lemma \ref{stime finali grado 3 perturbazione}. 
  
  
  \section{Proof of Theorem \ref{comparison Kuksin}}\label{comparison with Kuksin's map}  
Within this proof, it is convenient to use complex Birkhoff 
coordinates, given by $\zeta_n := (x_n - \ii y_n) /  \sqrt 2,$ $n \in \mathbb Z.$
A solution $z(t) = (x(t), y(t))$ of the dNLS equation in Birkhoff coordinates then satisfies the equations
\begin{equation}\label{dnls in complex Birkhoff}
\partial_t \zeta_n= - \ii \omega_n^{nls} \zeta_n, \quad n \in \mathbb Z\,,
\end{equation}
where
\[
\omega_n^{nls} \equiv \omega_n^{nls}(I_S, I_\bot) = \partial _{I_n} H^{nls}(I_S, I_\bot)\,.
\]
Linearize  \eqref{dnls in complex Birkhoff} at 
a solution $\zeta(t)$ of the form $( \zeta_S(t), 0)$. For initial data of the form 
$\widehat \zeta(0) = (0, \widehat \zeta_\bot(0))$, the corresponding  
solution
$\widehat \zeta(t) = (\widehat \zeta_S (t), \widehat \zeta_\bot(t))$
of the linearized equation satisfies 
\[ \widehat \zeta_S (t) \equiv 0\,, \qquad
\partial_t \widehat \zeta_n (t)= - \ii \omega_n^{nls}(I_S, 0) \, \widehat \zeta_n(t), \,\, \,  n \in S^\bot.
\]
Note that the latter equation are reduced to constant coefficients and hence
\[
\widehat \zeta_\bot(t) \,  = \, 
(e^{-\ii \omega_n(I_S, 0) t} \widehat \zeta_n(0))_{n \in S^\bot}\,.
\]
Since $\Psi^{nls}$ is symplectic, the solution of the equation,
obtained by linearizing the dNLS equation along $\Psi^{nls}(z_S(t), 0)$,
with initial data $d\Psi^{nls}(0, \widehat \zeta_\bot(0))$,
is given by
\[
\widehat w(t) = d\Psi^{nls}(z_S(t), 0)[0, \widehat \zeta_\bot (t)] 
\]
We now consider the special solutions 
$\widehat \zeta^{\pm , j}(t) = e^{\pm \ii\omega_j(I_S, 0) t} \,\widehat \zeta^{\pm , j}(0)$, $j \in S^\bot$, 
corresponding to the  initial data
 \[
 \widehat \zeta^{\pm , j}(0)= ( e^{(1,j)} \pm \ii e^{(2, j)} ) / \sqrt 2\,, \qquad 
e^{(1, j)}= ((\delta_{n j})_{n \in \mathbb Z}, 0), \,\,\, e^{(2, j)}= (0, (\delta_{n j})_{n \in \mathbb Z})\,.
\]
Note that these solutions are periodic in time and that $d\Psi^{nls}(z(t))[\widehat \zeta^{\pm , j}(t)] $
 can be written as
\[ 
\widehat w^{\pm, j}(t) = e^{\pm \ii\omega_j(I_S, 0) t} \, d\Psi^{nls}(z_S(t), 0)[\widehat \zeta^{\pm , j}(0)]\,.
\]
In the terminology  of \cite{K},
$\widehat w^{+, j}(t), \,\, \widehat w^{-, j}(t)$, $j \in S^\bot$, are Floquet solutions
with Floquet exponents $\pm \omega_j(I_S, 0)$.  Furthermore, by Theorem \ref{Theorem Birkhoff coordinates},
\[
\{ \, d\Psi^{nls}(z_S(t), 0)[\widehat \zeta^{\sigma , j}(0)]\, : \,  j \in S^\bot, \, \sigma \in \{\pm\} \}
\]
is a complete set of Floquet solutions in the sense of \cite{K}. One then concludes that 
up to normalisations (cf Appendix B) and natural identifications
(such as the identifications of action angle with Birkhoff coordinates), 
the map $\Phi_1$,
obtained by applying the scheme of construction of \cite{K} to the dNLS equation, coincides with the map
\[
\mathbb R^S \times \mathbb R^S \to {\cal L}(h^s_{\bot r}, H^s_r)\,, \quad
z_S \,  \mapsto \, d\Psi^{nls}(z_S, 0)\big|_{ h^s_{\bot r}}\,.
\]
Since according to \cite{K},  the map $\Phi(z)$ can be chosen of the form $\Psi^{nls}(z_S, 0) + \Phi_1(z)$ and since
the symplectic corrector $\Psi_C$ is constructed following the scheme in \cite{K}, one concludes that 
again up to normalisations and natural identifications, 
$\Psi = \Psi_L \circ \Psi_C$ coincides with the map $\Phi \circ \phi$ obtained by applying the scheme 
of \cite{K} to the dNLS equation.  \hfill $\square$

\begin{remark} In the terminology of \cite{K}, the system of the Floquet exponents $\pm \omega_j(I_S, 0)$, $j \in S^\bot$, is nonresonant -- see e.g. \cite{BKM} where the relevant properties of the dNLS frequencies are discussed.

\end{remark}

   
\section{Appendix A: a version of the Poincar\'e lemma}\label{appendice A}
  We follow the general approach of \cite{Lang}, Chapter V,  and restrict to the finite dimensional setup as the extension to infinite dimension is straightforward by restriction, see \cite{K}, Lemma 1.1. Let $E = \R^n$ and denote by $L_a^r(E)$ the space of multilinear continuous alternating forms of degree $0 \le r \le n$. Let $U \subseteq E$ be an open nonempty set and consider 
  $$
  \omega : U \to L_a^r(E)\,.
  $$
  For any $z \in U$, denote by 
  $$
  \omega(z)[\xi_1, \ldots, \xi_r] \in \R
  $$
  the value of $\omega(z)$ when evaluated at $\xi_1, \ldots, \xi_r \in E$. Similarly, if $\xi_j = \xi_j(z) \in E$, $j = 1, \ldots, r$, are vector fields on $U$, then we denote by $\omega[\xi_1, \ldots, \xi_r]$ the function 
  $$
  U \to \R \,, \quad z \mapsto \omega(z)[\xi_1(z), \ldots, \xi_r(z)] \,.
  $$
Furthermore, we denote by $\omega'(z) \cdot \xi$, $\xi \in E$, the alternating $r$-form 
  \begin{equation}\label{directional derivative}
  \partial_\e \mid_{ \e = 0} \, \omega(z + \e \xi) \in L^r_a(E)\,.
  \end{equation}
  The exterior differential $d \omega$ of $\omega$, evaluated at $z \in U$, $\xi_1, \ldots, \xi_{r + 1} \in E$, 
  is then given by the formula
 \begin{equation}\label{formula differenziale esterno}
  \sum_{j = 1}^{r + 1} (- 1)^{j + 1} \omega'(z) \cdot \xi_j [\xi_1, \ldots, \xi_{j - 1}, \xi_{j + 1}, \ldots, \xi_{r + 1}]\, ,
  \end{equation}
  also referred to as Cartan's formula. Let us now consider the case where 
  $$
  E = \R^{n_1} \times \R^{n_2}\,, \quad n = n_1 + n_2\,, \quad n_2 \geq 1\,,
  $$
  $$
  U = U_1 \times U_2 \subseteq \R^{n_1} \times \R^{n_2}\,, \quad 
  $$
  and $U_2$ is a ball in $\R^{n_2}$ centered at $0$. We denote the elements of $U$ by $z = (x, y)$ and the ones of $E$ by $\xi = (v, w) \in \R^{n_1} \times \R^{n_2}$. 
  For any $r$-form $ \omega$ on $U$, denote by $  \omega_{\cal C}$
   the $(r - 1)$-form on $U$, obtained by the cone construction: for any $x \in U_1,$ $y \in U_2$,
   $v_1, \dots , v_{r-1} \in \R^{n_1}$, and $w_1, \dots , w_{r-1} \in \R^{n_2}$, 
   \begin{equation}\label{r - 1 forma lemma di poincare}
   \omega_{\cal C}(x, y)[(v_1, w_1), \ldots, (v_{r - 1}, w_{r - 1})] = 
  \int_0^1  \omega(x, t y)[(0, y), (v_1, t w_1), \ldots, (v_{r - 1}, t w_{r - 1})]\,d t\,.
   \end{equation}
   Note that since $U_2$ is a ball in $\R^{n_2}$, centered at $0$, for any $0 \le t \le 1$,
    $(x, t y)$ is in $U_1 \times U_2$ and hence $ \omega(x, t y)$ in \eqref{r - 1 forma lemma di poincare}  
   is well defined.
   \begin{lemma}[Poincar\'e  lemma]\label{lemma di poincare}
  Assume that $\omega$ is a $r$-form on $U = U_1 \times U_2$, with $1 \le r \leq n$ and $n_2 \geq 1$, 
 satisfying
 \begin{equation}\label{condition H}
 \omega (x, 0) [(v_1, 0), \ldots, (v_r, 0)] = 0\,, \quad \forall \,\, x \in U_{1}\,, \quad \forall \,\, v_1, \ldots, v_r \in \R^{n_1}\,.
  \end{equation}
 Then
  \begin{equation}\label{identita cruciale lemma poincare}
   d ( \omega_{\cal C} )+ (d \omega)_{\cal C} =  \omega\,.
   \end{equation}
 In particular, if in addition $\omega$ is closed, $d \omega = 0$, then 
 $ d ( \omega_{\cal C} ) =  \omega$\,.
  \end{lemma}

\section{Appendix B: formulas for $d\Psi^{nls}(z_S, 0) [ (0, z_\bot) ]$ }\label{appendice B}
Note that for $z = (z_S, z_\bot)$ with $z_S \in \R^S  \times \R^S$ and $z_\bot = ( (x_j)_{j \in S^\bot}, (y_j)_{j \in S^\bot}) \in h^0_{\bot r}$,
\[
d\Psi^{nls}(z_S, 0) [ (0, z_\bot) ] = \sum_{j \in S^\bot} x_j d\Psi^{nls}(z_S, 0) [e^{(1, j)}] \,
+ \, \sum_{j \in S^\bot} y_j d\Psi^{nls}(z_S, 0) [e^{(2, j)}]
\]
where for any $j \in S^\bot,$
\[
e^{(1, j)} = ((\delta_{nj})_{n \in \Z}, 0), \qquad  e^{(2, j)} = (0, (\delta_{nj})_{n \in \Z})\,.
\]
It turns out that for $j \in S^\bot$, $d\Psi^{nls}(z_S, 0) [e^{(1, j)}]$ and $d\Psi^{nls}(z_S, 0) [e^{(2, j)}]$
can be computed quite explicitly. Consider the Hamiltonian equation with Hamiltonian
given by the coordinate function $x_j$, $\partial_t w = \ii {\mathbb J} \partial x_j$, and
denote by $w(t)$ its solution with initial data $w(0) = (z_S, 0)$. Then $z(t):= \Phi^{nls}(w(t))$
solves 
\begin{equation}\label{equation z}
\partial_t z = d \Phi^{nls} (w(t)) \, \partial_t w(t) = d \Phi^{nls} (w(t)) \, \ii {\mathbb J} \partial x_j\,.
\end{equation}
Since by Theorem \ref{Theorem Birkhoff coordinates}, $\Phi^{nls}$ is symplectic, 
one has $\partial_t z = J e^{(1, j)} = e^{(2,j)}$. When combined with \eqref{equation z} it implies that
$d \Psi^{nls} (z(t))[e^{(2, j)}] = \ii {\mathbb J} \partial x_j$. Similarly, one derives the corresponding 
identity for the coordinate function $y_j$. When evaluated at $t=0$ we then obtain
\[
d \Psi^{nls} ((z_S, 0))[e^{(2, j)}] = \ii {\mathbb J} \partial x_j  = (- \ii \partial_{v} x_j, \ii \partial_u x_j)\,, \qquad 
d \Psi^{nls} ((z_S, 0))[e^{(1, j)}] = \ii {\mathbb J} \partial y_j = (- \ii \partial_{v} y_j, \ii \partial_u y_j)\,.
\]
By the definition of $x_j,$ $y_j$ in \cite{GK}, p 113, one has for a potential $w \in H^0_r$ 
with Birkhoff coordinates $(z_S, 0)$ (referred to as $S-$gap potential)
\[
x_j = \frac{\xi_j}{ \sqrt 8} (e^{\ii \beta_j} {\frak z}_j + e^{- \ii \beta_j} {\frak z}_j)\,, \qquad
y_j = \frac{\xi_j}{ \sqrt 8 \,  \ii} (e^{\ii \beta_j} {\frak z}_j - e^{- \ii \beta_j} {\frak z}_j)\,,
\]
where ${\frak z}_j^\pm = \gamma_j e^{\pm \ii \eta_j}$ if $\gamma_j \ne 0$
and ${\frak z}_j^\pm = 0$ otherwise. We refer to \cite{GK} for the definitions of $\xi_j$,
$ \eta_j$, and  $\beta_j$. Since $w$ is assumed to be a $S-$gap potential, it follows that for any $j \in S^\bot,$
\[
\partial x_j = \frac{\xi_j}{ \sqrt 8} (e^{\ii \beta_j} \partial {\frak z}^+_j + e^{- \ii \beta_j} \partial {\frak z}^-_j)\,, \qquad
\partial y_j = \frac{\xi_j}{ \sqrt 8 \,  \ii} (e^{\ii \beta_j} \partial {\frak z}^+_j - e^{- \ii \beta_j} \partial {\frak z}^-_j)\,,
\]
where by formula (17.3) in \cite{GK},
\[
\partial {\frak z}^\pm_j = 2 ( \partial \tau_j - \partial \mu_j )
\pm \big(\, 
\ii 2 \delta(\mu_j) \partial \phi_j  \, +  \, 2 \phi_j \,  ( \ii \partial \delta \mid_{\lambda = \mu_j} + \ii \dot \delta (\mu_j) \partial \mu_j ) 
\, \big)\,.
\]
We refer to \cite{GK} for the definitions of the various quantities as well as for formulas
of the gradients in the latter expression. Each of the two components of these gradients are shown to be
a linear combination of quadratic expressions in the entries of the fundamental solution $M = M(x, \lambda)$ 
of the Zakharov Shabat operator 
\[
L := \begin{pmatrix}
\ii  & 0\\
0 & - \ii 
\end{pmatrix} \partial_x+ 
 \begin{pmatrix}
0 & u\\
\bar u & 0
\end{pmatrix}\,, \qquad w = (u, v) = (u, \bar u)\,.
\]
In fact, in \cite{GKP}, it has been proved that
\[
\partial {\frak z}^\pm_j = \big( \, (K_{j 2} \pm \ii H_{j 2})^2, \,  (K_{j 1} \pm \ii H_{j 1})^2  \big)
\]
where 
\[
H_j = (H_{j1}, H_{j2}) = \frac{1}{\| M_1 + M_2 \|_{L^2}}(M_1 + M_2)\mid_{\lambda = \mu_j}
\]
denotes the $L^2-$normalized eigenfunction
of $L$ for the Dirichlet eigenvalue $\mu_j$, $M_1,$ $M_2$ are the two columns of $M$,
and $K_j = (K_{j1}, K_{j2}) $ is the $L^2-$normalized solution of $LF = \mu_j F,$ which is $L^2-$orthogonal to $H_j$
and satisfies the additional normalization condition $-\ii (K_{j1}(0) - K_{j2}(0)) > 0. $


\vspace{1.0cm}

\noindent
T. Kappeler, 
Institut f\"ur Mathematik, 
Universit\"at Z\"urich, Winterthurerstrasse 190, CH-8057 Z\"urich;\\
${}\qquad$  email: thomas.kappeler@math.uzh.ch \\

\noindent
R. Montalto, 
Institut f\"ur Mathematik, 
Universit\"at Z\"urich, Winterthurerstrasse 190, CH-8057 Z\"urich;\\
${}\qquad$ email: riccardo.montalto@math.uzh.ch

\end{document}